\documentclass[a4paper, 11pt, twoside, leqno]{article}

\usepackage[T1]{fontenc}
\usepackage[utf8x]{inputenc}
\usepackage[english]{babel}
\usepackage{lmodern}               
\usepackage{amsmath,amsthm,amssymb}
\usepackage{mathrsfs,esint,bbm,stmaryrd}
\usepackage{enumitem,graphicx,xcolor,subcaption}
\usepackage[textwidth=16.1cm,centering,headheight=24pt]{geometry}
\usepackage{authblk}

\newtheorem{theorem}{Theorem}           

\newtheorem{lemma}[theorem]{Lemma}
\newtheorem{prop}[theorem]{Proposition}

\theoremstyle{definition}              

\theoremstyle{remark}                  
\newtheorem{step}{Step}

\newtheorem{remark}{Remark}

\makeatletter
\renewcommand*{\@fnsymbol}[1]{\ensuremath{\ifcase#1\or *\or \ddagger \or
   \mathsection\or  **\or \dagger\dagger
   \or \ddagger\ddagger \else\@ctrerr\fi}}
\makeatother

\DeclareMathOperator{\tr}{tr}                                       
\DeclareMathOperator{\dist}{dist}                                   
\DeclareMathOperator{\sign}{sign}                                   
\DeclareMathOperator{\curl}{curl}                                  

\DeclareMathOperator{\BV}{BV}
\DeclareMathOperator{\SBV}{SBV}

\DeclareMathOperator{\loc}{loc}

\newcommand{\abs}[1]{\left| #1 \right|}                             
\newcommand{\norm}[1]{\left\| #1 \right\|}                          
\newcommand{\meas}{\mathbin{\vrule height 1.6ex depth 0pt width
0.13ex\vrule height 0.13ex depth 0pt width 1.3ex}}
\newcommand{\csubset}{\subset\!\subset}                             

\DeclareMathAlphabet{\mathpzc}{OT1}{pzc}{m}{it}

\newcommand{\D}{\mathrm{D}}       

\renewcommand{\d}{\mathrm{d}}
\renewcommand{\o}{\mathrm{o}}
\renewcommand{\O}{\mathrm{O}}

\newcommand{\N}{\mathbb{N}}       
\newcommand{\R}{\mathbb{R}}
\newcommand{\Z}{\mathbb{Z}}
\newcommand{\C}{\mathbb{C}} 
\renewcommand{\SS}{\mathbb{S}}
\renewcommand{\P}{\mathbf{P}}    
\newcommand{\Q}{\mathbf{Q}}     
\newcommand{\M}{\mathbf{M}}
\newcommand{\I}{\mathbf{I}}
\newcommand{\n}{\mathbf{n}}
\newcommand{\m}{\mathbf{m}}

\newcommand{\f}{\mathbf{f}}
\newcommand{\g}{\mathbf{g}}

\newcommand{\pp}{\partial}
\newcommand{\dd}{\mathrm{d}}

\renewcommand{\u}{\mathbf{u}}
\newcommand{\NN}{\mathscr{N}}     

\newcommand{\F}{\mathscr{F}}
\renewcommand{\H}{\mathscr{H}}

\newcommand{\eps}{\varepsilon}

\SetSymbolFont{stmry}{bold}{U}{stmry}{m}{n}  
\newcommand{\nnu}{{\boldsymbol{\nu}}}

\newcommand{\ttau}{{\boldsymbol{\tau}}}

\newcommand{\Sz}{\mathscr{S}_0^{2\times 2}}
\newcommand{\Qb}{\Q_{\mathrm{bd}}}
\newcommand{\Mb}{\M_{\mathrm{bd}}}

\renewcommand{\v}{\mathbf{v}}

\renewcommand{\S}{\mathrm{S}}

\definecolor{lightblue}{rgb}{0.22,0.45,0.70}   
\definecolor{darkgray}{gray}{0.4}    
\definecolor{lightgray}{gray}{0.8}

\title{Two-dimensional Ferronematics, Canonical Harmonic Maps\\
and Minimal Connections}

\author{Giacomo Canevari\thanks{
Dipartimento di Informatica, 
Universit\`a di Verona,
Strada le Grazie 15, 37134 Verona, Italy.\\
\emph{E-mail address}: \texttt{giacomo.canevari@univr.it}
}, \
Apala Majumdar\thanks{
Department of Mathematics and Statistics, University of Strathclyde,
Livingstone Tower, 26 Richmond Street, Glasgow G1 1XH, United Kingdom.\\
\emph{E-mail address}: 
\texttt{apala.majumdar@strath.ac.uk}
}, \
Bianca Stroffolini\thanks{Dipartimento di Matematica e Applicazioni, Universit\`a Federico II di Napoli, via Cintia, 80126 Napoli, Italy.\\
\emph{E-mail address}: 
\texttt{bstroffo@unina.it}
} \ 
and Yiwei Wang\thanks{Department of Mathematics, University of California, Riverside, Riverside, CA 92508, USA. \\
\emph{E-mail address}: 
\texttt{yiweiw@ucr.edu}
}}

\date{\today}

\begin{document}

\maketitle

\begin{abstract}
 We study a variational model for ferronematics in two-dimensional domains, in the ``super-dilute'' regime. 
 The free energy functional consists of a reduced Landau-de Gennes energy for the nematic order parameter, a Ginzburg-Landau type energy for the spontaneous magnetisation, and a coupling term that favours the co-alignment of the nematic director and the magnetisation. 
 In a suitable asymptotic regime, we prove that the nematic order parameter converges to a canonical harmonic map with non-orientable point defects, while the magnetisation converges to a singular vector field, with line defects that connect the non-orientable point defects in pairs, along a minimal connection.
 
 
\end{abstract}

\numberwithin{equation}{section}
\numberwithin{definition}{section}
\numberwithin{theorem}{section}
\numberwithin{remark}{section}
\numberwithin{example}{section}

\section{Introduction}

Nematic liquid crystals (NLCs) are classical examples of mesophases or liquid crystalline phases that combine fluidity with the directionality of solids \cite{dg}. The nematic molecules are typically asymmetric in shape e.g. rod-shaped, wedge-shaped etc., and these molecules tend to align along certain locally preferred directions in space, referred to as {\bf{directors}}. Consequently, NLCs have a direction-dependent response to external stimuli such as electric fields, magnetic fields, temperature and incident light. Notably, the directionality or anisotropy of NLC physical and mechanical responses make them the working material of choice for a range of electro-optic applications \cite{lagerwall-scalia}. 
However, the magnetic susceptibility of NLCs is much weaker than their dielectric anisotropy, typically by several orders of magnitude \cite{Brochard}. Hence, NLCs exhibit a much stronger response to applied electric fields than their magnetic counterparts and as a result,  NLC devices are mainly driven by
electric  fields. This
naturally raises a question as to whether we can enhance the magneto-nematic coupling and induce a spontaneous magnetisation  by the introduction of magnetic nanoparticles (nanoparticles with magnetic moments) in nematic media, even without external magnetic fields. If implemented successfully, these magneto-nematic systems would have a much stronger response to applied magnetic fields, compared to conventional nematic systems, rendering the possibility of magnetic-field driven NLC systems in the physical sciences and engineering.

This idea was first  introduced in 1970 by
Brochard and de Gennes in their pioneering work on ferronematics \cite{Brochard} and these composite systems of magnetic nanoparticle (MNP)-dispersed nematic media are referred to as ferronematics in the literature \cite{burylov1,burylov2, Brochard}. The system has two order parameters --- the Landau-de Gennes (LdG) $\Q$-tensor order parameter 
to describe the nematic orientational anisotropy and the spontaneous magnetisation, $\mathbf{M}$,
induced by the suspended MNPs.
Brochard and de Gennes suggested that the nematic directors, denoted by $\n$, can be controlled by the
surface-induced mechanical coupling between NLCs and
MNPs. Equally, the spontaneous magnetisation, $\mathbf{M}$ profiles can be tailored by the nematic anisotropy through the MNP-NLC interactions, and this two-way coupling can stabilise exotic morphologies and defect patterns.

We work with dilute ferronematic suspensions relevant for a uniform suspension of MNPs in a nematic medium, such that the distance between pairs of MNPs is much larger than the individual MNP sizes and the volume fraction of the MNPs is small, building on the models introduced in \cite{burylov1, burylov2} and then in \cite{bisht2019, bisht2020}. In these dilute systems, the MNP-MNP interactions and the MNP-NLC interactions are absorbed by an empirical magneto-nematic coupling energy. These coupling energies can also be rigorously derived from homogenisation principles, as elucidated in the recent work \cite{canevarizarnescu}.
We work with two-dimensional, simply-connected and smooth domains $\Omega$, in a reduced LdG framework for which the $\Q$-tensor order parameter is a symmetric, traceless $2\times 2$ matrix and $\mathbf{M}$ is a two-dimensional vector field. This reduced approach can be rigorously justified using $\Gamma$-Convergence techniques (see \cite{GolovatyMontero} since in three dimensions, the LdG $\Q$-tensor order parameter is a symmetric, traceless $3\times 3$ matrix with five degrees of freedom). We use the effective re-scaled free energy for ferronematics, inspired by the experiments and results in \cite{Merteljetal} and proposed in \cite{bisht2019, bisht2020}. This energy has three components --- a reduced LdG free energy for NLCs,  a Ginzburg-Landau free energy for the magnetization
and a homogenised magneto-nematic coupling term:
\begin{equation}\label{main}
 \F_{\eps} (\Q, \M) := \int_{\Omega} 
 \left(\frac12 |\nabla \Q|^2 + \frac{\xi}2 |\nabla \M|^2 
 + \frac1{\eps^2} f(\Q, \M)\right) \mathrm{d} x
\end{equation}
In two dimensions, we have
\[
 f(\Q, \, \M) := \frac14 (|\Q|^2-1)^2
 + \frac{\xi}4 (|\M|^2-1)^2 - c_0 \Q \M\cdot \M .
\]
We work with a dimensionless model where $\eps^2$ is interpreted as a material-dependent, geometry-dependent and temperature-dependent positive elastic constant, $\xi$ is a ratio of the relative strength of the magnetic and NLC energies and $c_0$ is a coupling parameter. $\xi$ is necessarily positive, positive $c_0$ coerces co-alignment of $\n$ and $\M$ whereas $c_0 <0$ coerces $\n$ to be perpendicular to $\M$ \cite{bisht2020}. We only consider positive $c_0$ in this paper.

For dilute suspensions, $\eps$ and $\xi$ are necessarily small. In \cite{bisht2020}, the authors study stable critical points of this effective ferronematic free energy on square domains, with Dirichlet boundary conditions for both~$\Q$ and~$\M$. Their work is entirely numerical but does exhibit a plethora of exotic morphologies for different choices of $\eps$, $\xi$ and $c_0$. They demonstrate stable nematic point defects accompanied by both line defects and point defects in $\M$, and there is considerable freedom to manipulate the locations, multiplicity and dimensionality of defect profiles by simply tuning the values of $\xi$ and $c_0$. In particular, the numerical results clearly show that line defects or jump sets are observed in stable $\M$-profiles for small $\xi$ and $c_0$, whereas orientable point defects are stabilised in $\M$ for relatively large $\xi$ and $c_0$. Motivated by these numerical results, we study a special limit of the effective free energy in \eqref{main}, for which both $\xi$ and $c_0$ are proportional to $\eps$ and we study the profile of the corresponding energy minimizers in the $\eps \to 0$ limit, subject to Dirichlet boundary conditions for $\Q$ and $\M$. This can be interpreted as a ``super-dilute'' limit of the ferronematic free energy for which the magnetic energy is substantially weaker than the NLC energy, and the magneto-nematic coupling is weak. In the ``super-dilute'' limit, ``$\eps$'' is the only model parameter and $\xi$, $c_0$ are defined by the constants of proportionality which are fixed, and hence $\eps \to 0$ is the relevant asymptotic limit. Our main result shows that in this distinguished limit, the minimizing $\Q$-profiles are essentially canonical harmonic maps with a set of non-orientable nematic point defects, dictated by the topological degree of the Dirichlet boundary datum. This is consistent with previous powerful work in \cite{BaumanParkPhillips} in the context of the LdG theory and unsurprising since the LdG energy is the dominant energy. The minimizing $\M$-profiles are governed by a Modica-Mortola type of problem, quite specific to this super-dilute limit \cite{GiaquintaModicaSoucek-I}. They exhibit short line defects connecting pairs of the non-orientable nematic defects, consistent with the numerical results in \cite{bisht2020}. These line defects or jump sets in $\M$ are minimal connections between the nematic defects, and the location of the defects is determined by a modified renormalisation energy, which is the sum of a Ginzburg-Landau type renormalisation energy and a minimal connection energy. The modified renormalisation energy delicately captures the coupled nature of our problem, which makes it distinct and technically more complex than the usual LdG counterpart.

We complement our theoretical results with some numerical results for stable critical points of the ferronematic free energy, on square domains with topologically non-trivial Dirichlet boundary conditions for $\Q$ and $\M$. The converged numerical solutions are locally stable, and we expect multiple stable critical points for given choices of $\eps$, $\xi$ and $c_0$.
The numerical results are sensitive to the choices of $\eps$ and $c_0$, but there is evidence that the numerically computed stable solutions do indeed converge to a canonical harmonic $\Q$-map and a $\M$-profile closely tailored by the corresponding $\Q$-profile. The $\Q$-profile has a discrete set of non-orientable nematic defects and the $\M$-profile exhibits line defects connecting these nematic defects, in the $\eps \to 0$ limit. Whilst the practical relevance of such studies remains uncertain, it is clear that strong theoretical underpinnings are much needed for systematic scientific progress in this field, and our work is a first powerful step in an exhaustive study of ferronematic solution landscapes \cite{yin2020construction} (also see recent work in \cite{dalby2022}, \cite{maity2022}).

The paper is organised as follows. In Section~\ref{sect:statement}, we set up our problem and state our main result, recalling the key notions of a canonical harmonic map and a minimal connection. In Section~\ref{sect:changevar}, we state and prove some key technical preliminary results. In Section~\ref{sect:main}, we prove the six parts of our main theorem, including convergence results for the energy-minimizing $\Q$ and $\M$-profiles in different function spaces, and the convergence of the jump set of the energy-minimizing $\M$ to a minimal connection between pairs of non-orientable nematic defects, in the $\eps \to 0$ limit. The defect locations are captured in terms of minimizers of a modified renormalized energy, which is the sum of the Ginzburg-Landau renormalized energy and a minimal connection energy. The modified renormalized energy is derived from sharp lower and upper bounds for the energy minimizers in the $\eps \to 0$ limit, in Sections~\ref{sect:Gamma_liminf} and \ref{ssect:Gamma-limsup}. In Section~\ref{sect:numerics}, we present some numerical results and conclude with some perspectives in Section~\ref{sect:conclusions}.

\section{Statement of the main result}
\label{sect:statement}

Let~$\Sz$ be the set of $2\times 2$, real, symmetric, trace-free matrices,
equipped with the scalar product $\Q\cdot\P := \tr(\Q\P) = Q_{ij}P_{ij}$
and the induced norm~$\abs{\Q}^2 := \tr(\Q^2) = Q_{ij}Q_{ij}$.
Let~$\Omega\subseteq\R^2$ be a bounded, Lipschitz, simply connected domain. The ``super-dilute'' limit of the ferronematic free energy is defined by 
\[
\xi = \eps; \quad c_0 = \beta \eps\]
where $\beta$, $\eps$ are \emph{positive} parameters.
For~$\Q\colon\Omega\to\Sz$ and~$\M\colon\Omega\to\R^2$,
we define the functional
\begin{equation}\label{energy}
 \F_\eps (\Q, \, \M) := \int_{\Omega} \left(\frac{1}{2} \abs{\nabla\Q}^2
  + \frac{\eps}{2} \abs{\nabla\M}^2 
  + \frac{1}{\eps^2} f_\eps(\Q, \, \M)\right) \d x,
\end{equation}
where the potential~$f_\eps$ is given by
\begin{equation} \label{f}
 f_\eps(\Q, \, \M) := \frac{1}{4} (|\Q|^2-1)^2
 + \frac{\eps}{4} (|\M|^2-1)^2 - \beta\eps \Q \M\cdot \M + \kappa_\eps
\end{equation}
and~$\kappa_\eps\in\R$ is a constant, uniquely determined
by imposing that~$\inf f_\eps = 0$.

We consider minimisers of~\eqref{energy} subject 
to the Dirichlet boundary condition
\begin{equation} \label{bc}
 \Q = \Qb, \quad \M = \Mb \qquad \textrm{on } \partial\Omega,
\end{equation}
We assume that $\Qb\in C^1(\partial\Omega, \, \Sz)$, 
$\Mb\in C^1(\partial\Omega, \, \R^2)$ are 
($\eps$-independent) maps such that
\begin{equation} \label{hp:bc} 
 \abs{\Mb} = (\sqrt{2}\beta + 1)^{1/2},
 \qquad \Qb = \sqrt{2}\left(\frac{\Mb\otimes\Mb}{\sqrt{2}\beta + 1}
- \frac{\I}{2}\right)
\end{equation}
at any point of~$\partial\Omega$. Here~$\I$
is the~$2\times 2$ identity matrix.
{ The assumption~\eqref{hp:bc} implies that the potential~$f_\eps$, evaluated on the boundary datum~$(\Qb, \, \Mb)$, takes \emph{nonzero} but \emph{small} values --- that is, we have
$f_\eps(\Qb, \, \Mb) > 0$ for~$\eps\ > 0$ but~$f_\eps(\Qb, \, \Mb)\to 0$
as~$\eps\to 0$. (For details of this computation, see
Lemma~\ref{lemma:fepsapp} in Appendix~\ref{app:feps}.)}

Throughout this paper, we will denote by~$(\Q^*_\eps, \, \M^*_\eps)$
a minimiser of the functional~\eqref{energy} subject to
the boundary conditions~\eqref{bc}.
By routine arguments, minimisers exist and they satisfy
the Euler-Lagrange system of equations
\begin{align}
  -\Delta\Q^*_\eps + \dfrac{1}{\eps^2}(\abs{\Q^*_\eps}^2 - 1)\Q^*_\eps 
  - \dfrac{\beta}{\eps}\left(\M^*_\eps\otimes\M^*_\eps 
  - \dfrac{\abs{\M^*_\eps}^2}{2}\I\right)  &= 0  \label{EL-Q} \\
  -\Delta\M^*_\eps + \dfrac{1}{\eps^2}(\abs{\M^*_\eps}^2 - 1)\M^*_\eps 
  - \dfrac{2\beta}{\eps^2}\Q^*_\eps\M^*_\eps &= 0. \label{EL-M}
\end{align}

We denote as~$\NN$ the unit circle in the space 
of~$\Q$-tensors, that is,
\begin{equation} \label{N}
 \NN := \left\{\Q\in\Sz\colon \abs{\Q} = 1 \right\}
\end{equation}
Equivalently, $\NN$ may be described as
\begin{equation} \label{Nbis}
 \NN = \left\{\sqrt{2}\left(\n\otimes\n - \frac{\I}{2}\right)\colon \n\in\SS^1 \right\}
\end{equation}
As~$\Sz$ is a real vector space of dimension~$2$,
the set~$\NN$ is a smooth manifold, diffeomorphic
to the unit circle~$\SS^1\subseteq\C$. A diffeomorphism 
is given explicitely by
\begin{equation} \label{NS1}
 \NN\to\SS^1, \qquad \Q\mapsto\mathbf{q} := \sqrt{2} (Q_{11}, \, Q_{12})
\end{equation}
By assumption, the domain~$\Omega\subseteq\R^2$
is bounded and convex, so its boundary~$\partial\Omega$
is parametrised by a simple, closed, Lipschitz curve
--- in particular, $\partial\Omega$ is homeomorphic to
the circle~$\SS^1$. Therefore, the boundary data~$(\Qb, \, \Mb)$
carries a well-defined topological degree
\begin{equation} \label{d}
 d := \deg(\Qb, \, \partial\Omega) = \deg(\Mb, \, \partial\Omega) \in\Z
\end{equation}
In principle, for a continuous map $\Q\colon\partial\Omega\to\NN$,
the degree may be a half-integer, that is
$\deg(\Q, \, \partial\Omega)\in\frac{1}{2}\Z$. However,
the boundary datum~$\Qb$ is orientable, by assumption~\eqref{hp:bc}
--- in fact, it is oriented by~$\Mb$. This explains why~$d$,
in our case, is an integer.

\begin{remark} \label{rk:bc}
 { The results in this paper ---in particular, our main result, Theorem~\ref{th:main} below --- remain true for slightly different choices of the boundary conditions. For instance, we could consider minimisers of the functional~\eqref{energy}
 in the class of maps~$\Q\in W^{1,2}(\Omega, \, \Sz)$ that satisfy~$\Q =\Qb$ on~$\partial\Omega$, where the boundary datum~$\Qb$ takes the form
 \begin{equation} \label{Qn}
  \Qb = \sqrt{2}\left(\n_{\mathrm{bd}}\otimes\n_{\mathrm{bd}} - \frac{\mathbf{I}}{2}\right)
  \qquad \textrm{for some } \n_{\mathrm{bd}}\in C^1(\partial\Omega, \, \R^2)
 \end{equation}
 and~$\deg(\mathbf{n}_{\mathrm{bd}}, \, \partial\Omega) = d$,
 but we do not impose any relation between~$\n_{\mathrm{bd}}$ and the value of~$\M$ at the boundary.
 In this case, minimisers of the functional will satisfy the natural (Neumann) boundary condition $\partial_\nu\M_\eps = 0$ on~$\partial\Omega$ for the~$\M$-component, where~$\partial_\nu$ is the outer normal derivative. The arguments carry over to this case, with no essential change (see also Remark~\ref{rk:Neumann}).
 }
\end{remark}

\paragraph*{The canonical harmonic map and the renormalised energy.}
In order to state our main result, we recall some 
terminology introduced by Bethuel, Brezis and H\'elein~\cite{BBH}. 
Although the results in~\cite{BBH} are stated in terms of complex-valued maps,
as opposed to~$\Q$-tensors, they do extend to our setting, 
due to the change of variable~\eqref{NS1}.
Let~$a_1, \, \ldots, \, a_{2\abs{d}}$ be distinct points in~$\Omega$
(with~$d$ given by~\eqref{d}).
We say that a map~$\Q^*\colon\Omega\to\NN$ is 
a \emph{canonical harmonic map} with singularities
at~$(a_1, \, \ldots, \, a_{2\abs{d}})$ and boundary datum~$\Qb$
if the following conditions hold:
\begin{enumerate}[label=(\roman*)]
 \item $\Q^*$ is smooth in~$\Omega\setminus\{a_1, \, \ldots, \, a_{2\abs{d}}\}$,
 continuous in~$\overline{\Omega}\setminus\{a_1, \, \ldots, \, a_{2\abs{d}}\}$
 and~$\Q^* = \Qb$ on~$\partial\Omega$;
 \item for any~$\sigma> 0$ small enough and 
 any~$j\in\{1, \, \ldots, \, 2\abs{d}\}$, we have
 \[
  \deg(\Q^*, \, \partial B_\sigma(a_j)) = \frac{\sign(d)}{2}
 \]
 \item $\Q^*\in W^{1,1}(\Omega, \, \NN)$ and
 \[
  \partial_j \left(Q^*_{11} \, \partial_j Q^*_{12} 
   - Q^*_{12} \, \partial_j Q^*_{11}\right) = 0 
 \]
 in the sense of distributions in { the whole of}~$\Omega$.
 (Here and in what follows, we adopt Einstein's notation for the sum). 
\end{enumerate}

If~$B\subseteq\Omega\setminus\{a_1, \, \ldots, \, a_{2\abs{d}}\}$
is a ball that does not contain any singular point of~$\Q^*$,
then~$\Q^*$ can written in the form
\begin{equation} \label{Q_lifting}
 \Q^* = \frac{1}{\sqrt{2}}
  \left(\begin{matrix}
    \cos\theta^* & \sin\theta^* \\
    \sin\theta^* & -\cos\theta^*
  \end{matrix} \right)
  \qquad \textrm{in } B,
\end{equation}
where~$\theta^*\colon B\to\R$ is a smooth function.
(Equation~\eqref{Q_lifting} follows from~\eqref{N}, 
by classical lifting results in topology.)
Then, the equation~(iii) above can be written in the form
\begin{equation} \label{theta-harmonic}
 -\Delta\theta^* = 0 \qquad \textrm{in } B.
\end{equation}
In other words, a canonical harmonic map can be written locally,
away from its singularities, in terms of a harmonic function.

The canonical harmonic map with singularities 
at~$(a_1, \, \ldots, \, a_{2\abs{d}})$
and boundary datum~$\Qb$ exists and is unique,  
see~\cite[Theorem~I.5, Remark~I.1]{BBH}. 
The canonical harmonic map satisfies
$\Q^*\in W^{1,p}(\Omega, \, \NN)$ for any~$p\in [1, \, 2)$,
but~$\Q^*\notin W^{1,2}(\Omega, \, \NN)$. Nevertheless, the limit
\begin{equation} \label{renormalised}
 \mathbb{W}(a_1, \, \ldots, \, a_{2\abs{d}})
 := \lim_{\sigma\to 0} \left(\frac{1}{2} 
  \int_{\Omega\setminus\bigcup_{j=1}^{2\abs{d}} B_\sigma(a_j)} 
  \abs{\nabla\Q^*}^2 \, \d x - 2\pi\abs{d}\abs{\log\sigma}\right)  
\end{equation}
exists and is finite (see~\cite[Theorem~I.8]{BBH}).
Following the terminology in~\cite{BBH}, the function~$\mathbb{W}$
is called the \emph{renormalised energy}.

\paragraph*{Minimal connections between singular points.}

Given distinct points~$a_1$, $a_2$, \ldots, $a_{2\abs{d}}$ in~$\R^2$,
we define a \emph{connection} for $\{a_1, \, \ldots, \, a_{2\abs{d}}\}$ 
as a finite collection of straight line segments
$\{L_1, \, \ldots, \, L_{\abs{d}}\}$
such that each~$a_j$ is an endpoint of exactly one
of the segments~$L_k$. In other words, the line segments~$L_j$
connects the points~$a_i$ in pairs. We define
\begin{equation} \label{minconn-intro}
 \mathbb{L}(a_1, \, \ldots, \, a_{2\abs{d}}) := \min\left\{
 \sum_{j = 1}^{\abs{d}} \H^1(L_j)\colon 
 \{L_1, \, \ldots, \, L_{\abs{d}}\} \textrm{ is a connection for }
 \{a_1, \, \ldots, \, a_{2\abs{d}}\} \right\}
\end{equation}
Here and throughout the paper, $\H^1$ denotes the $1$-dimensional Hausdorff measure (i.e., length).
We say that a connection~$\{L_1, \, \ldots, \, L_d\}$ 
is minimal if it is a minimiser for the 
right-hand side of~\eqref{minconn-intro}.
A notion of minimal connection, similar to~\eqref{minconn-intro},
was already introduced in~\cite{BrezisCoronLieb, AlmgrenBrowderLieb}.
However, the minimal connection was defined in~\cite{BrezisCoronLieb}
by taking the orientation into account --- that is, 
half of the points~$a_1, \, \ldots, \, a_{2\abs{d}}$ were given 
positive multiplicity~$1$, the other half were given negative
multiplicity~$-1$, and the segments~$L_j$ were required
to match points with opposite multiplicity. By constrast,
here we do not distinguish between positive and negative 
multiplicity for the points~$a_i$ and any segment of 
endpoints~$a_i$, $a_k$ is allowed. (In the language of
Geometric Measure Theory, the minimal connection was defined
in~\cite{BrezisCoronLieb} as the solution of a $1$-dimensional
Plateau problem with integer multiplicity, while~\eqref{minconn-intro}
is a $1$-dimensional Plateau problem modulo~$2$.)

\paragraph*{The main result.}

We prove a convergence result for minimisers~$(\Q^*_\eps, \, \M^*_\eps)$
of~\eqref{energy}, subject to the boundary conditions~\eqref{bc}--\eqref{hp:bc},
in the limit as~$\eps\to 0$. 
We denote by~$\SBV(\Omega, \, \R^2)$ the space of maps
$\M = (M_1, \, M_2)\colon\Omega\to\R^2$ whose
components~$M_1$, $M_2$ are special functions
of bounded variation, as defined by De Giorgi
and Ambrosio~\cite{DeGiorgiAmbrosio}. 
The distributional derivative~$\D\M$ of a map~$\M\in\SBV(\Omega, \, \R^2)$ can be decomposed as
\[
 \D\M = \nabla\M \, \mathscr{L}^2(\d x) + (\M^+ - \M^-)\otimes\nu_{\M}\,(\H^1\meas\S_{\M}) 
\]
where~$\nabla\M\colon\Omega\to\R^{2\times 2}$ is the absolutely continuous component of~$\D\M$, $\mathscr{L}^2(\d x)$ is the Lebesgue measure on~$\R^2$, $\S_{\M}$ is the jump set of~$\M$, $\M^+$, $\M^-$ are the traces of~$\M$ on either side of the jump and~$\nu_{\M}$ is the unit normal to the jump set.
(See e.g.~\cite{AmbrosioFuscoPallara} for more details).

\begin{theorem} \label{th:main}
 Let~$\Omega\subseteq\R^2$ be a bounded, Lipschitz, simply connected domain.
 Assume that the boundary data satisfy~\eqref{hp:bc}.
 Let~$(\Q^*_\eps, \, \M^*_\eps)$ be a minimiser of~\eqref{energy}
 subject to the boundary condition~\eqref{bc}. 
 Then, there exists a (non-relabelled) subsequence, 
 maps~$\Q_*\colon\Omega\to\NN$, $\M^*\colon\Omega\to\R^2$
 and distinct points~$a^*_1, \, \ldots, \, a^*_{2\abs{d}}$ 
 in~$\Omega$ such that the following holds:
 \begin{enumerate}[label=(\roman*)]
  \item $\Q^*_\eps\to\Q^*$ strongly in~$W^{1,p}(\Omega)$
   for any~$p < 2$;
  \item $\Q^*$ is the canonical harmonic map with 
   singularities~$(a^*_1, \, \ldots, \, a^*_{2\abs{d}})$
   and boundary datum~$\Qb$;
  \item $\M^*_\eps\to\M^*$ strongly in~$L^p(\Omega)$ 
   for any~$p<+\infty$;
  \item $\M^*\in\SBV(\Omega, \, \R^2)$ and it satisfies
  \[
    \abs{\M^*} = (\sqrt{2}\beta + 1)^{1/2},
    \qquad \Q^* = \sqrt{2}\left(\frac{\M^*\otimes\M^*}{\sqrt{2}\beta + 1}
     - \frac{\I}{2}\right)
  \]
  at almost every point of~$\Omega$.
 \end{enumerate}
 In addition, if the domain~$\Omega$ is convex, then
 \begin{enumerate}[label=(\roman*), resume]
  \item there exists a minimal connection~$(L^*_1, \, \ldots, \, L^*_{\abs{d}})$
  for~$(a^*_1, \, \ldots, \, a^*_{2\abs{d}})$ such that
  the jump set of~$\M^*$ coincides with $\bigcup_{j=1}^{\abs{d}} L^*_j$
  (up to sets of zero length);
  \item $(a^*_1, \, \ldots, \, a^*_{2\abs{d}})$
  minimises the function
  \[
   \mathbb{W}_{\beta}(a_1, \, \ldots, \, a_{2\abs{d}}) := 
    \mathbb{W}(a_1, \, \ldots, \, a_{2\abs{d}}) 
    + \frac{2\sqrt{2}}{3} \left(\sqrt{2}\beta + 1\right)^{3/2}
     \mathbb{L}(a_1, \, \ldots, \, a_{2\abs{d}})
  \]
  among all the $(2\abs{d})$-uples~$(a_1, \, \ldots, \, a_{2\abs{d}})$ 
  of distinct points in~$\Omega$.
 \end{enumerate}
\end{theorem}

\begin{remark} 
 { Theorem~\ref{th:main} implies that~$\M^*$ is a locally 
 harmonic map, away from the closure of its jump set,
 into the circle of radius~$(\sqrt{2}\beta + 1)^{1/2}$. In other words,
 if~$B$ is a ball that does not intersect the 
 closure of the jump set of~$\M^*$, then~$\M^*$
 can locally be written in the form~$\M^* = (\sqrt{2}\beta + 1)^{1/2}(\cos\phi^*, \, \sin\phi^*)$ for some scalar function~$\phi^*\colon B\to \R$ that satisfies $-\Delta\phi^*=0$ in~$B$. See Proposition~\ref{prop:harmonicM*} for the details.}
\end{remark}

\begin{remark}
{ Let us discuss the extremal case of the renormalized energy $\mathbb{W}_{\beta}(a_1, \, \ldots, \, a_{2\abs{d}})$. When $\beta \to +\infty$, the function $\mathbb{W}_{\beta}$ would be minimized by choosing $(L^*_1, \, \ldots, \, L^*_{\abs{d}})$ to be zero, meaning that the singular points will move toward each other. In the case where $\beta=0$ instead, the coupling term in the potential wouldn't be present. Therefore, we would have two decoupled Ginzburg-Landau problems.         }

\end{remark}
\begin{remark}
 Point defects and line defects connecting point defects do appear for energy minimizers in other variational models e.g.~continuum models for a complex-valued map in \cite{GoldmanMerletMillot} or for discrete models in ~\cite{Badal_et_al, BadalCicalese}. However, the mathematics is substantially different to our model problem for which we have two order parameters~$\Q$ and~$\M$, and a non-trivial coupling energy, which introduces substantive technical challenges.
\end{remark}

\section{Preliminaries}
\label{sect:changevar}

First, we state a few properties of the potential~$f_\eps$,
defined in~\eqref{f}. We define
\begin{equation} \label{k*}
 \kappa_* := \frac{\beta}{2\sqrt{2}} \left(\sqrt{2}\beta + 1\right) \! .
\end{equation}

\begin{lemma} \label{lemma:feps}
 The potential~$f_\eps$ satisfies the following properties.
 \begin{enumerate}[label=(\roman*)]
  \item The constant~$\kappa_\eps$ in~\eqref{f}, uniquely defined by imposing
  the condition~$\inf f_\eps = 0$, satisfies
  \[
   \kappa_\eps = \frac{1}{2} \left(\beta^2 + \sqrt{2} \beta\right) \eps
   + \kappa_*^2 \, \eps^2 + \o(\eps^2) 
  \]
  In particular, $\kappa_\eps\geq 0$ for~$\eps$ small enough.
  
  \item If~$(\Q, \, \M)\in\Sz\times\R^2$ is such that
  \[
   \abs{\M} = (\sqrt{2}\beta + 1)^{1/2},
  \qquad \Q = \sqrt{2}\left(\frac{\M\otimes\M}{\sqrt{2}\beta + 1} 
   - \frac{\I}{2}\right)
  \]
  then $f_\eps(\Q, \, \M) = \kappa_* \, \eps^2 + \o(\eps^2)$.
  
  \item If~$\eps$ is sufficiently small, then
  \begin{align*} 
   \frac{1}{\eps^2} f_\eps(\Q, \, \M) 
    &\geq \frac{1}{4\eps^2}(\abs{\Q}^2 - 1)^2
     - \frac{\beta}{\sqrt{2}\eps} \abs{\M}^2 
     \, \abs{\abs{\Q} - 1} 
  \end{align*}
  and
  \begin{equation} \label{potential_comparison}
   \frac{1}{\eps^2} f_\eps(\Q, \, \M) 
    \geq \frac{1}{8\eps^2}(\abs{\Q}^2 - 1)^2 
    - \beta^2\abs{\M}^4 
  \end{equation}
  for any~$(\Q, \, \M)\in\Sz\times\R^2$.
 \end{enumerate}
\end{lemma}

The proof of Lemma~\ref{lemma:feps} is contained in Appendix~\ref{app:feps}.

In the rest of this section, we describe an alternative
expression for the functional~\eqref{energy},
which will be useful in our analysis.
Let~$G\subseteq\Omega$ be a smooth, simply connected
subdomain. Let~$(\Q_\eps, \, \M_\eps)_{\eps>0}$ be any sequence 
in~$W^{1,2}(G, \, \Sz)\times W^{1,2}(G, \, \R^2)$ 
(not necessarily a sequence of minimisers) that satisfies
\begin{gather}
 \int_{G} \left(\frac{1}{2}\abs{\nabla\Q_\eps}^2 
  + \frac{1}{4\eps^2}(\abs{\Q_\eps}^2 - 1)^2 \right) \d x 
  \lesssim\abs{\log\eps} \label{hp:chvar-energy} \\
 \abs{\Q_\eps(x)} \geq \frac{1}{2}, \quad 
  \abs{\M_\eps(x)} \leq A \qquad
  \textrm{for any } x\in G, \ \eps > 0, \label{hp:chvar-abs}
\end{gather}
where~$A$ is some positive constant that does not depend on~$\eps$.
As we have assumed that~$G$ is simply connected 
and that~$\abs{\Q_\eps}\geq 1/2$ in~$G$, we can apply
lifting results~\cite{BethuelZheng, BethuelChiron, BallZarnescu}
and write~$\Q_\eps$ in the form
\begin{equation} \label{spectral}
 \Q_\eps = \frac{\abs{\Q_\eps}}{\sqrt{2}} 
 \left(\n_\eps\otimes\n_\eps - \m_\eps\otimes\m_\eps\right) 
 \qquad \textrm{in } G.
\end{equation}
Here~$(\n_\eps, \, \m_\eps)$ is an orthonormal 
set of eigenvectors for~$\Q_\eps$ with
$\n_\eps\in W^{1,2}(G, \, \SS^1)$, $\m_\eps\in W^{1,2}(G, \, \SS^1)$.
We define the vector field~$\u_\eps\in W^{1,2}(G, \, \R^2)$ as
\begin{equation} \label{uQM}
 (u_\eps)_1 := \M_\eps\cdot\n_\eps, \qquad
 (u_\eps)_2 := \M_\eps\cdot\m_\eps
\end{equation}
so that $\M_\eps = (u_\eps)_1 \, \n_\eps + (u_\eps)_2 \, \m_\eps$.
Our next result expresses the energy~$\mathscr{F}_\eps(\Q_\eps, \, \M_\eps; \, G)$
in terms of the variables~$\Q_\eps$ and~$\u_\eps$. We define the functions
\begin{align}
 g_\eps(\Q) &:= \frac{1}{4\eps^2}(\abs{\Q}^2 - 1)^2 
   - \frac{2\kappa_*}{\eps} (\abs{\Q} - 1) + \kappa_*^2 \label{geps} \\
 h(\u) &:= \frac{1}{4} (\abs{\u}^2 - 1)^2
  - \frac{\beta}{\sqrt{2}}(u_1^2 - u_2^2)
  + \frac{\beta^2 + \sqrt{2}\beta}{2}   \label{h}
\end{align}
for any~$\Q\in\Sz$ and any~$\u=(u_1, \, u_2)\in\R^2$.
We recall that~$\kappa_*$ is the constant defined by~\eqref{k*}.

\begin{remark} \label{rk:u-unique}
 The vector fields~$\n_\eps$, $\m_\eps$ are determined
 by~$\Q_\eps$ only up to their sign --- Equation~\eqref{spectral}
 still holds if we replace~$\n_\eps$ by~$-\n_\eps$ or~$\m_\eps$ by~$-\m_\eps$.
 Therefore, the unit vector~$\u_\eps$ is uniquely 
 determined by~$\Q_\eps$, $\M_\eps$ only \emph{up to the sign of its components}
 $(u_\eps)_1$, $(u_\eps)_2$. However, the quantity
 $h(\u_\eps)$ is is well-defined, irrespective of the choice
 of the orientations for~$\n_\eps$, $\m_\eps$, because
 $h(-u_1, \, u_2) = h(u_1, \, -u_2) = h(u_1, \, u_2)$.
\end{remark}

\begin{prop} \label{prop:chvar}
 Let~$(\Q_\eps, \, \M_\eps)_{\eps>0}$ be a sequence
 in~$W^{1,2}(G, \, \Sz)\times W^{1,2}(G, \, \R^2)$ that
 satisfies~\eqref{hp:chvar-energy} and~\eqref{hp:chvar-abs}.
 Let~$\u_\eps$ be defined as in~\eqref{uQM}. Then, we have
 \begin{equation*}
  \begin{split}
   \F_\eps(\Q_\eps, \, \M_\eps; \, G)
   &= \int_G\left(\frac{1}{2}\abs{\nabla\Q_\eps}^2
   + g_\eps(\Q_\eps) \right) \d x 
   + \int_G\left(\frac{\eps}{2}\abs{\nabla\u_\eps}^2
   + \frac{1}{\eps}h(\u_\eps) \right) \d x  + R_\eps
  \end{split}
 \end{equation*}
 where the remainder term~$R_\eps$ satisfies
 \begin{equation} \label{changevar-R}
  \abs{R_\eps} \lesssim \eps^{1/2} \abs{\log\eps}^{1/2}
  \left(\int_G \left(\frac{\eps}{2}\abs{\nabla\u_\eps}^2
   + \frac{1}{\eps}h(\u_\eps)\right) \d x \right)^{1/2} + \o(1)
 \end{equation}
 as~$\eps\to 0$.
\end{prop}

In other words, the change of variables~\eqref{uQM}
transforms the functional into a sum of two decoupled terms,
which can be studied independently, and a remainder term, 
which is small compared to the other ones.
Before we proceed with the proof of Proposition~\ref{prop:chvar},
we state some properties of the functions~$g_\eps$, $h$
defined in~\eqref{geps}, ~\eqref{h} respectively.
These properties are elementary, but will be useful later on.

\begin{lemma} \label{lemma:geps}
 The function~$g_\eps\colon\Sz\to\R$ is non-negative and satisfies
 \begin{equation*}
  g_\eps(\Q) = \left(\frac{1}{\eps} (\abs{\Q} - 1) - \kappa_*\right)^2
  + \frac{1}{\eps^2}(\abs{\Q} - 1)^2 \left(\frac{1}{4}(\abs{\Q} + 1)^2 - 1\right)
 \end{equation*}
 for any~$\Q\in\Sz$.
\end{lemma}
\begin{proof} 
 We have
 \[
  \begin{split}
   g_\eps(\Q)
   &= \frac{1}{\eps^2}(\abs{\Q} - 1)^2 
   - \frac{2\kappa_*}{\eps} (\abs{\Q} - 1) + \kappa_*^2
   + \frac{1}{4\eps^2}(\abs{\Q}^2 - 1)^2 - \frac{1}{\eps}(\abs{\Q} - 1)^2 \\
   &= \left(\frac{1}{\eps} (\abs{\Q} - 1) - \kappa_*\right)^2
   + \frac{1}{\eps^2}(\abs{\Q} - 1)^2 \left(\frac{1}{4}(\abs{\Q} + 1)^2 - 1\right)
  \end{split}
 \]
 If~$\abs{\Q}\geq 1$, then $(\abs{\Q} + 1)^2 \geq 4$ and hence,
 $g_\eps(\Q)\geq 0$. On the other hand,
 if~$\abs{\Q}\leq 1$, then all the terms in~\eqref{geps}
 are non-negative.
\end{proof}

\begin{lemma} \label{lemma:h}
 The function~$h\colon\R^2\to\R$ is non-negative and its 
 zero-set~$h^{-1}(0)$ consists exactly of two points, 
 $\u_{\pm} := (\pm (\sqrt{2}\beta + 1)^{1/2}, \, 0)$.
 Moreover, the Hessian matrix of~$h$ at both~$\u_+$
 and~$\u_-$ is strictly positive definite.
\end{lemma}
\begin{proof}
 For any~$\u\in\R^2$, we have $h(\u)\geq h(\abs{\u}, \, 0)$
 and the inequality is strict if~$u_2\neq 0$.
 Therefore, it suffices to study~$h$
 on the line~$u_2 = 0$. We have
 \begin{equation*} 
  h(u_1, \, 0) = \frac{1}{4}
   \left(u_1^2 - 1 - \sqrt{2}\beta\right)^2
 \end{equation*}
 so~$h(\u)\geq 0$ for any~$\u\in\R^2$, with equality if and
 only if~$\u = (\pm (\sqrt{2}\beta + 1)^{1/2}, \, 0)$.
 Moreover,
 \[
  \nabla^2 h(\u_+) = \nabla^2 h(\u_-) = 
   \left(\begin{matrix}
     2 + 2\sqrt{2}\beta & 0 \\
     0                  & 2\sqrt{2}\beta
    \end{matrix}\right)
 \]
 so the lemma follows.
\end{proof}

\begin{proof}[Proof of Proposition~\ref{prop:chvar}]
 For simplicity of notation, 
 we omit the subscript~$\eps$ from all the variables.
 
 \setcounter{step}{0}
 \begin{step}
  Let~$k\in\{1, \, 2\}$. 
  We have~$\M = u_1\n + u_2\m$ and hence,
  \begin{equation} \label{chv0}
   \partial_k\M = (\partial_k u_1)\n + u_1\partial_k\n
   + (\partial_k u_2)\m + u_2\partial_k\m.
  \end{equation}
  We raise to the square both sides of~\eqref{chv0}.
  We apply the identities 
  \begin{equation} \label{chv0.5}
   \n\cdot\partial_k\n = \m\cdot\partial_k\m = 0, \qquad
   \n\cdot\partial_k\m + \m\cdot\partial_k\n=0, \qquad
   \partial_k\n\cdot\partial_k\m = 0
  \end{equation}
  which follow by differentiating the orthonormality conditions
  $\abs{\n}^2 = \abs{\m}^2 = 1$, $\n\cdot\m=0$.
  (In particular, the first identity in~\eqref{chv0}
  implies that~$\partial_k\n$ is parallel to~$\m$ and~$\partial_k\m$
  is parallel to~$\n$, so~$\partial_k\n\cdot\partial_k\m=0$.) 
  We obtain
  \begin{equation} \label{chv1}
   \abs{\partial_k\M}^2 = \abs{\partial_k\u}^2 
   + 2(u_1 \, \partial_k u_2 - u_2 \, \partial_k u_1) \m\cdot\partial_k\n
   + \abs{\u}^2\abs{\partial_k\n}^2
  \end{equation}
  We consider the potential term~$f_\eps(\Q, \, \M)$. 
  Since~$(\n, \, \m)$ is an orthonormal basis of~$\R^2$, we have
  \begin{equation} \label{hu1}
   \abs{\u} = \abs{\M}, \qquad
   \frac{\Q}{\abs{\Q}}\M\cdot\M 
   = \frac{1}{\sqrt{2}}\left(u_1^2 - u_2^2\right)
  \end{equation}
  By substituting~\eqref{hu1} into the definition~\eqref{f}
  of~$f_\eps$, and recalling~\eqref{geps}, \eqref{h}, we obtain
  \begin{equation} \label{hf}
   \begin{split}
    \frac{1}{\eps^2} f_\eps(\Q, \, \M) 
     &= \frac{1}{4\eps^2}(\abs{\Q}^2 - 1)^2 + \frac{1}{\eps} h(\u)
     + \frac{\beta}{\sqrt{2}\,\eps} (1 - \abs{\Q}) \, (u_1^2 - u_2^2) \\
     &\qquad\qquad + \frac{\kappa_\eps}{\eps^2} 
      - \frac{1}{2\eps}(\beta^2 + \sqrt{2}\beta) \\
     &= g_\eps(\Q) + \frac{1}{\eps} h(\u)
     + \frac{\abs{\Q} - 1}{\eps} 
      \left(2\kappa_* - \frac{\beta}{\sqrt{2}}
      (u_1^2 - u_2^2)\right)\\
     &\qquad\qquad + \frac{\kappa_\eps}{\eps^2} 
      - \frac{1}{2\eps}(\beta^2 + \sqrt{2}\beta) - \kappa_*^2
   \end{split}
  \end{equation}
  Combining~\eqref{chv1} with~\eqref{hf}, we obtain
  \begin{equation} \label{changevar}
   \begin{split}
   \F_\eps(\Q, \, \M; \, G)
   &= \int_G\left(\frac{1}{2}\abs{\nabla\Q}^2 + g_\eps(\Q)\right) \d x  
   + \int_G\left(\frac{\eps}{2}\abs{\nabla\u}^2
   + \frac{1}{\eps}h(\u) \right) \d x  \\
   &\qquad + \eps \sum_{k=1}^2
   \int_G \left(u_1 \, \partial_ku_2 - u_2 \, \partial_k u_1\right) 
   \m\cdot\partial_k\n \, \d x  
   + \frac{\eps}{2}\int_G\abs{\u}^2\abs{\nabla\n}^2 \, \d x  \\
   &\qquad + 
   \int_G \frac{\abs{\Q} - 1}{\eps} 
     \left(2\kappa_* - \frac{\beta}{\sqrt{2}} (u_1^2 - u_2^2)\right) \d x  \\
   &\qquad + \left(\frac{\kappa_\eps}{\eps^2} 
      - \frac{1}{2\eps}(\beta^2 + \sqrt{2}\beta) - \kappa_*^2\right) \abs{G}
   \end{split}
  \end{equation}
  where~$\abs{G}$ denotes the area of~$G$. We estimate separately
  the terms in the right-hand side of~\eqref{changevar}.
 \end{step}
 
 \begin{step}
  In view of the identity~$\n\otimes\n + \m\otimes\m =\I$,
  Equation~\eqref{spectral} can be written as
  \begin{equation} \label{spectralbis}
   \frac{\Q}{\abs{\Q}}  = \sqrt{2}
   \left(\n\otimes\n - \frac{\I}{2}\right) 
  \end{equation}
  We differentiate both sides of~\eqref{spectralbis} and
  compute the squared norm of the derivative.
  Recalling the assumption~\eqref{hp:chvar-abs}, 
  after routine computations we obtain
  \begin{equation} \label{chv2}
   \abs{\partial_k\n} = \frac{1}{2} 
   \abs{\partial_k\left(\frac{\Q}{\abs{\Q}}\right)}
   \lesssim \frac{\abs{\partial_k\Q}}{\abs{\Q}}
   \lesssim \abs{\partial_k\Q}
  \end{equation}
  Thanks to~\eqref{chv2}, we can estimate
  \begin{equation*}
   \begin{split}
    \eps \abs{\sum_{k=1}^2
     \int_G \left(u_1 \, \partial_ku_2 - u_2 \, \partial_k u_1\right) 
     \m\cdot\partial_k\n\, \d x}
    &\lesssim  \eps \norm{\u}_{L^\infty(G)} 
     \norm{\nabla\u}_{L^2(G)} \norm{\nabla\Q}_{L^2(G)}
   \end{split}
  \end{equation*}
  By our assumptions~\eqref{hp:chvar-energy}, \eqref{hp:chvar-abs}, 
  the~$L^\infty$-norm of~$\u$ is bounded and the~$L^2$-norm
  of~$\nabla\Q$ is of order~$\abs{\log\eps}^{1/2}$ 
  at most. Therefore, we obtain
  \begin{equation} \label{changevar-R1}
   \begin{split}
    \eps \abs{\sum_{k=1}^2
     \int_G \left(u_1 \, \partial_ku_2 - u_2 \, \partial_k u_1\right) 
     \m\cdot\partial_k\n \, \d x}
    &\lesssim  \eps^{1/2} \abs{\log\eps}^{1/2} 
    \left(\eps\int_G \abs{\nabla\u}^2 \right)^{1/2}
   \end{split}
  \end{equation}
  Equations~\eqref{hp:chvar-energy}, \eqref{hp:chvar-abs}
  and~\eqref{chv2} imply
  \begin{equation} \label{changevar-R2}
   \begin{split}
     \frac{\eps}{2}\int_G\abs{\u}^2\abs{\nabla\n}^2 \, \d x 
     \lesssim \eps \abs{\log\eps} \to 0 \qquad \textrm{as } \eps\to 0.
   \end{split}
  \end{equation}
  Moreover, Lemma~\ref{lemma:feps} gives
  \begin{equation} \label{changevar-R3}
   \begin{split}
    \left(\frac{\kappa_\eps}{\eps^2} 
      - \frac{1}{2\eps}(\beta^2 + \sqrt{2}\beta) 
      - \kappa_*^2\right) \abs{G} \to 0 \qquad \textrm{as } \eps\to 0.
   \end{split}
  \end{equation}
 \end{step}
 
 \begin{step} 
  By Lemma~\ref{lemma:h}, the function~$h$
  has two strict, non-degenerate minima at the
  points~$\u_{\pm} := (\pm(\sqrt{2}\beta + 1)^{1/2}, \, 0)$.
  As a consequence, 
  { for any~$\u\in\R^2$ such that~$\abs{\u}\leq A$
  (where~$A > 0$ is the constant from~\eqref{hp:chvar-abs}),}
  we must have
  \[
   \begin{split}
    h(u)\geq C_A \min\left\{(\u - \u_+)^2, \, (\u - \u_-)^2\right\}
    &= C_A\left(\abs{u_1} - (\sqrt{2}\beta + 1)^{1/2}\right)^2 + C_A u_2^2 \\
    &= C_A\frac{\left(u_1^2 - \sqrt{2}\beta - 1\right)^2}
     {\left(\abs{u_1} + (\sqrt{2}\beta + 1)^{1/2}\right)^2} + C_A u_2^2 \\
    &\geq C_A\left(u_1^2 - \sqrt{2}\beta - 1\right)^2 + C_A u_2^2 
   \end{split}
  \]
  { for some constant~$C_A$ that depends only on~$A$ and~$\beta$. 
  Then, for any~$\u\in\R^2$ with~$\abs{\u}\leq A$} we have
  \begin{equation} \label{chv3}
   \begin{split}
    \abs{2\kappa_* - \frac{\beta}{\sqrt{2}} (u_1^2 - u_2^2)}^2
    &\leq C^\prime_A \abs{2\kappa_* - \frac{\beta}{\sqrt{2}} (u_1^2 - u_2^2)} \\
    &\leq C^\prime_A \left(\frac{\beta}{\sqrt{2}} \abs{\sqrt{2}\beta + 1 - u_1^2}
     + \frac{\beta}{\sqrt{2}} u_2^2\right)
     \leq \frac{C^\prime_A \, \beta}{\sqrt{2} \, C_A} h(u)
   \end{split}
  \end{equation}
  for some (possibly) different constant~$C_A^\prime$, 
  still depending on~$A$ and~$\beta$ only.
  The assumption~\eqref{hp:chvar-abs} { and
  the property~\eqref{hu1}} guarantee that~$\u$
  { satisfies~$\abs{\u}\leq A$ almost everywhere in~$G$.} 
  Therefore,
  we can apply~\eqref{chv3} to estimate 
  \begin{equation*} 
   \begin{split}
    \int_G \frac{\abs{\Q} - 1}{\eps} 
     \left(2\kappa_* - \frac{\beta}{\sqrt{2}} (u_1^2 - u_2^2)\right) \d x 
     \lesssim \left(\frac{1}{\eps^2} \int_G (\abs{\Q} - 1)^2\, \d x \right)^{1/2}
     \left(\int_G h(\u)\, \d x  \right)^{1/2}
   \end{split}
  \end{equation*}
  The elementary inequality~$(x - 1)^2 \leq (x^2 - 1)^2$,
  which applies to any~$x\geq 0$, implies
  \begin{equation} \label{changevar-R4}
   \begin{split}
    \int_G \frac{\abs{\Q} - 1}{\eps} 
     \left(2\kappa_* - \frac{\beta}{\sqrt{2}} (u_1^2 - u_2^2)\right) \d x 
     &\lesssim \left(\frac{1}{\eps^2}
      \int_G (\abs{\Q}^2 - 1)^2\, \d x \right)^{1/2}
      \left(\int_G h(\u)\, \d x \right)^{1/2} \\
     &\hspace{-.16cm} \stackrel{\eqref{hp:chvar-energy}}{\lesssim}
      \eps^{1/2}\abs{\log\eps}^{1/2}
      \left(\frac{1}{\eps}\int_G h(\u)\, \d x \right)^{1/2}
   \end{split}
  \end{equation}
  The proposition follows by~\eqref{changevar}, \eqref{changevar-R1}, 
  \eqref{changevar-R2}, \eqref{changevar-R3} and~\eqref{changevar-R4}.
  \qedhere
 \end{step}
\end{proof}

\section{Proof of Theorem~\ref{th:main}}
\label{sect:main}

\subsection{Proof of Statement~(i): compactness for~$\Q^*_\eps$}

In this section, we prove that the $\Q^*_\eps$-component of
the minimisers converges to a limit, up to extraction of subsequences.
The results in this section are largely based on the analysis
in~\cite{BBH}. 
Throughout the paper, we denote by~$(\Q^*_\eps, \, \M^*_\eps)$ 
a minimiser of the functional~\eqref{energy}, 
subject to the boundary condition~\eqref{bc}.
We recall that the boundary data are of class~$C^1$ 
and satisfy the assumption~\eqref{hp:bc}. 
Routine arguments show that minimisers exist and that they satisfy the Euler-Lagrange equations~\eqref{EL-Q}--\eqref{EL-M}.

\begin{lemma} \label{lemma:max}
 The maps~$\Q^*_\eps$, $\M^*_\eps$ are smooth inside~$\Omega$
 and Lipschitz up to the boundary of~$\Omega$.
 Moreover, there exists an~$\eps$-independent constant~$C$ such that
 \begin{gather}
  \norm{\Q^*_\eps}_{L^\infty(\Omega)} 
   + \norm{\M^*_\eps}_{L^\infty(\Omega)} \leq C \label{max-QM} \\
  \norm{\nabla\Q^*_\eps}_{L^\infty(\Omega)}
   + \norm{\nabla\M^*_\eps}_{L^\infty(\Omega)}
   \leq \frac{C}{\eps}. \label{max-gradients}
 \end{gather}
\end{lemma}
\begin{proof}
 Elliptic regularity theory implies, via a bootstrap argument,
 that~$(\Q^*_\eps, \, \M^*_\eps)$ is smooth in the interior of~$\Omega$
 and continuous up to the boundary. Now, we prove~\eqref{max-QM}.
 We take the scalar product of both sides of~\eqref{EL-Q} with~$\Q^*_\eps$:
 \begin{equation} \label{max-Q}
  -\Delta\left(\frac{\abs{\Q^*_\eps}^2}{2}\right) + \abs{\nabla\Q^*_\eps}^2
  + \dfrac{1}{\eps^2}(\abs{\Q^*_\eps}^2 - 1)\abs{\Q^*_\eps}^2
  - \dfrac{\beta}{\eps}\Q^*_\eps\M^*_\eps\cdot\M^*_\eps = 0 
 \end{equation}
 In a similar way, by taking the scalar product 
 of~\eqref{EL-M} with~$\M^*_\eps$, we obtain
 \begin{equation} \label{max-M}
  -\Delta\left(\frac{\abs{\M^*_\eps}^2}{2}\right)
  + \abs{\nabla\M^*_\eps}^2
  + \dfrac{1}{\eps^2}(\abs{\M^*_\eps}^2 - 1)\abs{\M^*_\eps}^2 
  - \dfrac{2\beta}{\eps^2}\Q^*_\eps\M^*_\eps\cdot\M^*_\eps = 0. 
 \end{equation}
 By adding~\eqref{max-Q} and~\eqref{max-M},
 and rearranging terms, we deduce
 \begin{equation} \label{max1}
  \begin{split}
   \eps^2\Delta\left(\frac{\abs{\Q^*_\eps}^2 + \abs{\M^*_\eps}^2}{2}\right)
   \geq (\abs{\Q^*_\eps}^2 - 1)\abs{\Q^*_\eps}^2
   + (\abs{\M^*_\eps}^2 - 1)\abs{\M^*_\eps}^2 
   - \beta (\eps + 2)\Q^*_\eps\M^*_\eps\cdot\M^*_\eps 
  \end{split}
 \end{equation}
 The right-hand side of~\eqref{max1} is strictly positive
 if~$\abs{\Q^*_\eps}^2 + \abs{\M^*_\eps}^2 \geq C$, for some 
 (sufficiently large) constant~$C$ that depends 
 on~$\beta$ but not on~$\eps$. Therefore, \eqref{max-QM}
 follows from the maximum principle. The inequality~\eqref{max-gradients}
 follows by~\cite[Lemma~A.1 and Lemma~A.2]{BBH-degree_zero}.
\end{proof}

\begin{prop} \label{prop:energy_upperbd}
 Minimisers~$(\Q^*_\eps, \, \M^*_\eps)$ of~$\F_\eps$
 subject to the boundary conditions~$\Q = \Qb$, $\M = \Mb$
 on~$\partial\Omega$ satisfy
 \[
  \F_\eps(\Q^*_\eps, \, \M^*_\eps) \leq 2\pi\abs{d} \abs{\log\eps} + C,
 \]
 where~$d\in\Z$ is the degree of~$\M^*_\eps$ and~$C$
 is a constant that depends only on~$\Omega$, $\Qb$, $\Mb$
 (not on~$\eps$).
\end{prop}
\begin{proof}
We first consider the case~$d=1$.
Consider balls $B_1 := B(a_1, R)$, $B_2 := B(a_2, R)$, of centres~$a_1$, $a_2$ and radius~$R>0$, that are mutually disjoint. Since we have assumed
that the degree of the boundary datum~$\Qb$ is~$d=1$,
there exists a map~$\tilde{\Q}\colon\Omega\setminus(B_1\cup B_2)\to\NN$
that is smooth (up to the boundary of~$\Omega\setminus(B_1\cup B_2)$),
satisfies~$\tilde{\Q}=\Qb$ on~$\partial\Omega$ and has degree~$1/2$
on~$\partial B_1$ and~$\partial B_2$.
We define a comparison map ${\Q}_{\eps}$
as follows:
\begin{align*}
 \Q(x) := \begin{cases}
 \tilde{\Q}(x) \quad  &\mbox{if } x\in \Omega\setminus(B_1\cup B_2)\\
 \Q^1_{\eps}(x) \quad & \mbox{if } x=a_1+\rho e^{i\theta}\in B_1 = B(a_1, R)\\
 \Q^2_{\eps}(x) \quad & \mbox{if } x=a_2+\rho e^{i\theta}\in B_2 = B(a_2, R).
 \end{cases}
\end{align*} 
where~$\Q^1_\eps$, $\Q^2_\eps$ are given as
$$\Q^1_{\eps}(a_1+\rho e^{i\theta}):=\sqrt{2}s_{\eps}(\rho)\left(\n^1(\theta)\otimes \n^1(\theta)-\frac{\I}2\right) \!, \qquad \n^1(\theta)=e^{i\theta/2}$$
$$\Q^2_{\eps}(a_2+\rho e^{i\theta}):=\sqrt{2}s_{\eps}(\rho)\left(\n^2(\theta)\otimes \n^2(\theta)-\frac{\I}2\right) \! , \qquad \n^2(\theta)=e^{i\theta/2},$$
and~$s_{\eps}(\rho)$ is the truncation at $1$,
$s_\eps(\rho) := \min\{\frac{\rho}{\eps}, 1\}$.
A direct computation yields
\begin{equation} \label{ub1}
 \frac12 \int_{\Omega} |\nabla \Q_{\eps}|^2 
 \d x\leq2\pi \log\left(\frac{R}{\eps}\right) + C
\end{equation}
for some constant~$C$ that does not depend on~$\eps$.
Indeed, since~$\tilde{\Q}$ is regular on~$\Omega\setminus(B_1\cup B_2)$
and takes values in the manifold~$\NN$, the energy of~$\Q_\eps$
on~$\Omega\setminus(B_1\cup B_2)$ is an $\eps$-independent constant,
whereas the contribution of~$\Q_\eps^1$, $\Q_\eps^2$
is reminiscent of the Ginzburg-Landau functional and can be computed explicitely.

Next, we construct the component~$\M_\eps$.
Let~$\Lambda$ be the straight line segment of endpoints~$a_1$, $a_2$.
Thanks to Lemma~\ref{lemma:goodlifting} in Appendix~\ref{app:lifting},
there exists a vector field~$\tilde{\M}_\eps\in\SBV(\Omega, \, \R^2)$
such that
\begin{equation} \label{liftM-1}
 |\tilde{\M}_\eps|=(\sqrt{2} \beta +1)^{\frac12}, \qquad
 \Q_\eps=\sqrt{2} \left( \frac{ \tilde{\M}_\eps \otimes\tilde{\M}_\eps}{\sqrt{2} \beta +1}-\frac{\I}2\right)
\end{equation}
a.e. in~$\Omega$ and, moreover, satisfies
$S_{\tilde{\M}_\eps} = \Lambda$, up to negligible sets.
In particular, $\tilde{\M}_\eps$ is smooth in a neighbourhood of~$\partial\Omega$.
By comparing~\eqref{hp:bc} with~\eqref{liftM-1}, it follows that either~$\tilde{\M}_\eps = \Mb$ on~$\partial\Omega$ or~$\tilde{\M}_\eps = -\Mb$ on~$\partial\Omega$. Up to a change of sign, we will assume without loss of generality that~$\tilde{\M}_\eps = \Mb$ on~$\partial\Omega$.
In order to define our competitor~$\M_\eps$, we need
to regularise~$\tilde{\M}_\eps$ near its jump set.
We define
\[
 \M_\eps(x) := \min\left\{\frac{\dist(x, \Lambda)}{\eps}, \, 1\right\}
 \tilde{\M}_\eps(x)
 \qquad\textrm{for any } x\in\Omega.
\]
For~$\eps$ small enough, we have $\M_\eps = \tilde{\M}_\eps = \Mb$ on~$\partial\Omega$.
The absolutely continuous part of gradient~$\nabla\tilde{\M}_\eps$
can be estimated by differentiating both sides of~\eqref{liftM-1},
by the BV-chain rule; it turns out that
$|\nabla\tilde{\M}_\eps| = c\abs{\nabla\Q_\eps}$,
up to an (explicit) constant factor~$c$ that does not 
depend on~$\eps$. By explicit computation, we have
\[
 \eps\abs{\nabla\M_\eps(x)}^2 \lesssim 
 \frac{1}{\eps}\chi_\eps(x) + \eps\abs{\nabla\Q_\eps(x)}^2
 \qquad\textrm{for any } x\in\Omega,
\]
where~$\chi_\eps\colon\Omega\to\R$ is defined as
$\chi_\eps(x) := 1$ if~$\dist(x, \, \Lambda)\leq \eps$
and~$\chi_\eps(x) := 0$ otherwise. Then, due to~\eqref{ub1}, we have
\begin{equation} \label{ub2}
 \eps\int_\Omega\abs{\nabla\M_\eps}^2 \, \d x
 \leq C\left(1 + \eps\abs{\log\eps}\right)
 \leq C
\end{equation}
Finally, we compute the potential.
We need to consider three different contributions.
At a point~$x\in\Omega\setminus (B(a_1, \eps)\cup B(a_2, \eps))$ such that~$\dist(x, \, \Lambda) > \eps$,
we have~$f_\eps(\Q_\eps(x), \, \M_\eps(x)) = \O(\eps^2)$
due to~\eqref{liftM-1} and Lemma~\ref{lemma:feps}.
At a point~$x\in\Omega\setminus (B(a_1, \eps)\cup B(a_2, \eps))$ such that $\dist(x, \, \Lambda) < \eps$ ,
we have~$|\Q_\eps(x)| = 1$ and hence, $f_\eps(\Q_\eps(x), \, \M_\eps(x)) = \O(\eps)$.  At a point~$x\in B(a_1, \eps)\cup B(a_2, \eps)$,
the potential~$f_\eps(\Q_\eps(x), \, \M_\eps(x))$
is bounded by a constant that does not depend on~$\eps$.
Therefore, we have
\begin{equation} \label{ub3}
 \int_\Omega f_\eps(\Q_\eps, \, \M_\eps) \, \d x \lesssim \eps^2
\end{equation}
Together, \eqref{ub1}, \eqref{ub2} and~\eqref{ub3} imply
\[
 \F_\eps(\Q^*_\eps, \, \M^*_\eps)
 \leq \F_\eps(\Q_\eps, \, \M_\eps)
 \leq 2\pi \abs{\log\eps} + C
\]
for some constant~$C$ that does not depend on~$\eps$.
The proof in case~$d\neq 1$ is similar,
except that in the definition of~$\tilde{\Q}$,
we need to consider~$2\abs{d}$ 
{ pairwise disjoint} balls~$B_1, \, B_2, \, \ldots \, B_{2\abs{d}}$,
each of them carrying a topological degree of~$\sign(d)/2$.
The set~$\Lambda$ is defined as a union of segments
that connects the centres of the balls~$B_1, \, B_2, \, \ldots \, B_{2\abs{d}}$
(for instance, a minimal connection --- see Appendix~\ref{app:lifting}).
\end{proof}

{ The following estimate is well-known estimate in the Ginzburg-Landau
literature~\cite{DelPinoFelmer}.
\begin{lemma} \label{lemma:delpino}
 There exists an~$\eps$-independent constant~$C$ such that
 \begin{equation*} 
  \frac{1}{4\eps^2}\int_{\Omega} 
  (\abs{\Q^*_{\eps}}^2 - 1)^2 \, \d x \leq C
 \end{equation*}
 for any~$\eps$.
\end{lemma}
Lemma~\ref{lemma:delpino} is a direct consequence of~Theorem~1.1
in~\cite{DelPinoFelmer}.}
A compactness result for the~$\Q_\eps^*$-component of minimisers
can also be obtained by appealing to results in the
Ginzburg-Landau theory. 
Given a (closed) ball~$\bar{B}_\rho(a)\subseteq\Omega$ 
such that $\abs{\Q^*_\eps} \geq 1/2$ on~$\partial B_\rho(a)$, the map
\[
 \frac{\Q^*_\eps}{\abs{\Q^*_\eps}}\colon\partial B_\rho(a)
 \simeq\SS^1\to \NN\simeq\R\mathrm{P}^1
\]
is well-defined and continuous and hence, its topological degree
is well-defined as an element of~$\frac{1}{2}\Z$. We denote 
the topological degree of~$\Q^*_\eps/|\Q^*_\eps|$
on~$\partial B_\rho(a)$ by~$\deg(\Q^*_\eps, \, \partial B_\rho(a_j))$.
We recall that~$d$ is the degree of the boundary datum, as given in~\eqref{d}.

\begin{lemma} \label{lemma:compactnessQ}
 There exist distinct points~$a^*_1$, \ldots, $a^*_{2\abs{d}}$ in~$\Omega$,
 distinct points $b^*_1$, \ldots, $b^*_K$ in~$\overline{\Omega}$ 
 and a (non-relabelled) subsequence such that
 the following statement holds. For any~$\sigma> 0$ sufficiently small
 there exists~$\eps_0(\sigma) > 0$ such that,
 if~$0 < \eps \leq \eps_0(\sigma)$, then
 \begin{align}
  &\frac{1}{2} \leq \abs{\Q^*_\eps(x)} \leq \frac{3}{2}
   \qquad \textrm{for any } 
   x\notin \bigcup_{j=1}^{2\abs{d}} B_\sigma(a^*_j) 
   \cup \bigcup_{k=1}^K B_\sigma(b^*_k) \label{clearingout} \\
  &\deg(\Q^*_\eps, \, \partial B_\sigma(a^*_j)) 
   = \frac{1}{2}\sign(d), \qquad
  \deg(\Q^*_\eps, \, \partial(B_\sigma(b^*_k)\cap\Omega))
   = 0 \label{degree}
 \end{align}
 for any~$j \in \{1, \, \ldots, \, 2\abs{d}\}$,
 any~$k\in\{1, \, \ldots, \, K\}$. Moreover,
 for any~$\sigma$ sufficiently small and
 any~$0 < \eps \leq \eps_0(\sigma)$, there holds
 \begin{equation} \label{energybd}
  \mathscr{F}_\eps\left(\Q_\eps, \, \M_\eps; \, 
   \Omega\setminus\cup_{j=1}^{2\abs{d}} B_\sigma(a^*_j)\right) 
   \leq 2\pi\abs{d} \abs{\log\sigma} + C
 \end{equation}
 where~$C$ is a positive constant~$C$
 that does not depend on~$\eps$, $\sigma$.
 Finally, there exists a limit map~$\Q^*\colon\Omega\to\NN$ such that
 \begin{equation} \label{QW1p2}
  \Q^*_\eps\rightharpoonup\Q^* \qquad \textrm{weakly in } W^{1,p}(\Omega)
  \textrm{ for any } p < 2 \textrm{ and in } W^{1,2}_{\mathrm{loc}}(\Omega\setminus\{a^*_1, \, \ldots, \, a^*_{2\abs{d}}\}).
 \end{equation}
\end{lemma}
\begin{proof} 
 The analysis of the~$\Q^*_\eps$-component can be recast
 in the classical Ginzburg-Landau setting, by { means} of a 
 change of variables. We define~$\mathbf{q}^*_{\eps}\colon\Omega\to\R^2$ as
 \begin{equation} \label{GLchangevar}
  \mathbf{q}^*_{\eps} := \sqrt{2} ((Q^*_\eps)_{11}, \, (Q^*_\eps)_{12})
 \end{equation}
 Since~$\Q^*_\eps$ is symmetric and trace-free, we have
 $\abs{\mathbf{q}^*_{\eps}} = \abs{\Q^*_\eps}$ and
 $\abs{\nabla\mathbf{q}^*_{\eps}} = \abs{\nabla\Q^*_\eps}$.
 With the help of Lemma~\ref{lemma:feps}, we deduce
 \begin{equation*}
  \begin{split}
   E_{\eps}(\mathbf{q}^*_{\eps})
   &:= \int_{\Omega} \left(\frac{1}{2}\abs{\nabla\mathbf{q}^*_{\eps}}^2
    + \frac{1}{8\eps^2}(\abs{\mathbf{q}^*_{\eps}}^2 - 1)^2\right) \d x 
   \stackrel{\eqref{potential_comparison}}{\leq} 
   \F_\eps(\Q^*_\eps, \, \M^*_\eps) 
    + \beta^2\int_{\Omega}\abs{\M^*_\eps}^4 \,  \d x  
  \end{split}
 \end{equation*}
 The terms at the right-hand side can be bounded by
 Proposition~\ref{prop:energy_upperbd} and
 Lemma~\ref{lemma:max}, respectively. We obtain
 \begin{equation} \label{compQ1}
  \begin{split}
   E_{\eps}(\mathbf{q}^*_{\eps})
   \leq 2\pi\abs{d}\abs{\log\eps} + C,
  \end{split}
 \end{equation}
 where~$C$ is an~$\eps$-independent constant.
 Moreover, due to the boundary condition~\eqref{bc}
 and~\eqref{hp:bc}, $\mathbf{q}^*_{\eps}$ restricted 
 to the boundary~$\partial\Omega$ coincides with
 an~$\eps$-independent map of class~$C^1$. More precisely,
if we identify vectors in~$\R^2$ with complex numbers so that $\Mb$ is identified with a complex number, $\Mb = {\Mb}_{1} + \mathrm{i} {\Mb}_{2}$,
 then a routine computation shows
 \[
  \mathbf{q}^*_\eps = \frac{\Mb^2}{\sqrt{2}\beta + 1} 
  \qquad \textrm{on } \partial\Omega
 \]
 (the square is taken in the sense of complex numbers).
 In particular, $\abs{\mathbf{q}^*_{\eps}} = 1$ on~$\partial\Omega$ and
 \begin{equation} \label{degreeq}
  \deg(\mathbf{q}^*_\eps, \, \partial\Omega)
  = 2\deg(\Mb, \, \partial\Omega) \stackrel{\eqref{d}}{=} 2d.
 \end{equation}
 Now, \eqref{clearingout}, \eqref{degree}, \eqref{energybd} follow
 from classical results in the Ginzburg-Landau literature 
 (see e.g~\cite[Theorem~2.4]{Lin96}, \cite[Proposition~1.1]{Lin99},
 \cite[Theorems~1.2 and~1.3]{Jerrard}, \cite[Theorem~1]{Sandier}).
 Moreover, the arguments in~\cite[Theorem~1.1]{Struwe}
 prove that, for any~$p\in (1, \, 2)$, there exists a constant~$C_p$
 such that
 \begin{equation} \label{QW1p}
  \int_{\Omega} \abs{\nabla\Q^*_\eps}^p \, \d x  \leq C_p.
 \end{equation}
 for any~$\eps$ sufficiently small. Then, \eqref{QW1p2} follows from~\eqref{energybd} and~\eqref{QW1p},
 by means of a compactness argument.
\end{proof}

In order to complete the proof of Statement~(i) in Theorem~\ref{th:main},
it only remains to show that the convergence~$\Q^*_\eps\to\Q^*$
is not only weak, but also strong in~$W^{1,p}(\Omega)$.
The proof of this fact relies on an auxiliary lemma.
We consider the function~$g_\eps\colon\Sz\to\R$ defined in~\eqref{geps}.

\begin{lemma} \label{lemma:interp}
 Let~$B = B_r(x_0)\subseteq\Omega$ be an open ball.
 Suppose that~$\Q^*_\eps\rightharpoonup \Q^*$ 
 weakly in~$W^{1,2}(\partial B)$ and that
 \begin{equation} \label{hp:interp}
   \begin{split}
     \int_{\partial B} \left(\frac{1}{2}\abs{\nabla\Q^*_\eps}^2
     + g_\eps(\Q^*_\eps) \right) \d\H^1 
     \leq C
   \end{split}
  \end{equation}
  for some constant~$C$ that may depend on the radius~$r$, but not on~$\eps$.
  Then, there exists a map~$\Q_\eps\in W^{1,2}(B, \, \Sz)$ such that
  \begin{gather} 
   \Q_\eps = \Q_\eps^* \quad \textrm{on } \partial B, \qquad 
   \abs{\Q_\eps} \geq \frac{1}{2} \quad \textrm{in } B \label{interp1} \\
   \int_{B}
    \left(\frac{1}{2}\abs{\nabla\Q_\eps}^2 + g_\eps(\Q_\eps) \right) \d x
   \to \frac{1}{2} \int_{B} \abs{\nabla\Q^*}^2 \, \d x \label{interp2}
  \end{gather}
\end{lemma}

The proof of Lemma~\ref{lemma:interp} is given in Appendix~\ref{sect:interp}.

\begin{prop} \label{prop:strongconv}
 As~$\eps\to 0$, we have
 \begin{align}
  \Q^*_\eps\to\Q^* \qquad &\textrm{strongly in } 
   W^{1,2}_{\loc}(\Omega\setminus\{a^*_1, \, \ldots,
   \, a^*_{2\abs{d}}, \, b_1, \, \ldots, \, b_K\})
   \label{strongconvW12} \\
  \Q^*_\eps\to\Q^* \qquad &\textrm{strongly in } W^{1,p}(\Omega)
   \ \textrm{ for any } p\in[1, \, 2) \label{strongconvW1p} 
 \end{align}
\end{prop}
\begin{proof}
  Let~$B := B_R(x_0)\csubset\Omega\setminus
  \{a^*_1, \, \ldots, \, a^*_{2\abs{d}}, \, b_1, \, \ldots, \, b_K\}$
  be an open ball. We have~$\abs{\Q^*_\eps}\geq 1/2$ in~$B$,
  so we can apply the change of variables described
  in Section~\ref{sect:changevar}. We consider the vector
  field~$\u^*_\eps\colon B\to\R^2$ defined as in~\eqref{uQM}
  --- that is, we write
  \[
    \Q_\eps^* = \frac{\abs{\Q_\eps^*}}{\sqrt{2}} 
    \left(\n_\eps^*\otimes\n_\eps^* - \m_\eps^*\otimes\m_\eps^*\right) 
    \qquad \textrm{in } B,
  \]
  where~$(\n_\eps^*, \, \m_\eps^*)$ is an orthonormal basis of 
  eigenvectors for~$\Q_\eps$, and we define~$(u^*_\eps)_1 := \M_\eps^*\cdot\n_\eps^*$, $(u^*_\eps)_2 := \M_\eps^*\cdot\m_\eps^*$.
  By Proposition~\ref{prop:chvar}, we have
  \begin{equation} \label{strcnv1}
   \begin{split}
    \F_\eps(\Q^*_\eps, \, \M^*_\eps; \, B)
     &= \int_B \left(\frac{1}{2}\abs{\nabla\Q^*_\eps}^2
     + g_\eps(\Q^*_\eps) \right) \d x \\
     &\qquad\qquad+ \int_B \left(\frac{\eps}{2}\abs{\nabla\u^*_\eps}^2
     + \frac{1}{\eps}h(\u^*_\eps) \right) \d x  + \o_{\eps\to 0}(1)
   \end{split}
  \end{equation}
  where the functions~$g_\eps$ and~$h$ are defined in~\eqref{geps}
  and~\eqref{h}, respectively. (The remainder term~$R_\eps$, 
  given by Proposition~\ref{prop:chvar}, tends to zero as~$\eps\to 0$, 
  due to~\eqref{changevar-R} and the energy bound~\eqref{energybd}).
  By Lemma~\ref{lemma:compactnessQ}, we know that
  $\F_\eps(\Q^*_\eps, \, \M^*_\eps; \, B)\leq C$
  for some constant~$C$ that depends on the ball~$B$, but not on~$\eps$.
  By Fubini theorem, and possibly up to extraction of a subsequence,
  we find a radius~$r\in (R/2, \, R)$ such that
  \begin{equation} \label{strcnv2}
   \begin{split}
     \int_{\partial B_r(x_0)} \left(\frac{1}{2}\abs{\nabla\Q^*_\eps}^2
     + g_\eps(\Q^*_\eps) \right) \d\H^1 
     + \frac{1}{2} \int_{\partial B_r(x_0)} \abs{\nabla\Q^*}^2 \, \d\H^1 
     \leq \frac{C}{R}
   \end{split}
  \end{equation}
  with~$C$ that does not depend on~$\eps$.
  Moreover, without loss of generality we can assume that
  $\Q_{\eps}^*\rightharpoonup\Q^*$ weakly in~$W^{1,2}(\partial B_r(x_0))$.
  Let~$B^\prime := B_r(x_0)$. By Lemma~\ref{lemma:interp}, there exists a 
  map~$\Q_\eps\in W^{1,2}(B^\prime, \, \Sz)$ such that
  \begin{gather} 
   \Q_\eps = \Q_\eps^* \quad \textrm{on } \partial B^\prime, \qquad 
   \abs{\Q_\eps} \geq \frac{1}{2} \quad \textrm{in } B^\prime \label{strcnv7} \\
   \int_{B^\prime}
    \left(\frac{1}{2}\abs{\nabla\Q_\eps}^2 + g_\eps(\Q_\eps) \right) \d x
   \to \frac{1}{2} \int_{B^\prime} \abs{\nabla\Q^*}^2 \, \d x \label{strcnv8}
  \end{gather}
  Thanks to~\eqref{strcnv7}, we can write
  \[
    \Q_\eps = \frac{\abs{\Q_\eps}}{\sqrt{2}} 
    \left(\n_\eps\otimes\n_\eps - \m_\eps\otimes\m_\eps\right) 
    \qquad \textrm{in } B^\prime,
  \]
  where~$(\n_\eps, \, \m_\eps)$ is an orthonormal basis of 
  eigenvectors for~$\Q_\eps$. We define 
  \[
   \M_\eps := (u^*_\eps)_1 \, \n_\eps + (u^*_\eps)_2 \, \m_\eps
   \qquad \textrm{in } B^\prime.
  \]
  The pair~$(\Q_\eps, \, \M_\eps)$ is an admissible competitor
  for~$(\Q^*_\eps, \, \M^*_\eps)$: $\Q_\eps = \Q^*_\eps$ on~$\partial B^\prime$
  by construction and, if the orientation of~$\n_\eps$ and~$\m_\eps$ 
  is chosen suitably, then $\M_\eps = \M^*_\eps$ on~$\partial B^\prime$.
  By minimality of~$(\Q^*_\eps, \, \M^*_\eps)$, we have
  $\F_\eps(\Q^*, \, \M^*_\eps; \, B^\prime)
   \leq \F_\eps(\Q_\eps, \, \M_\eps; \, B^\prime)$
  By applying Proposition~\ref{prop:chvar}, we deduce
  \begin{equation} \label{strcnv6}
   \begin{split}
    \int_{B^\prime} \left(\frac{1}{2}\abs{\nabla\Q^*_\eps}^2 
     + g_\eps(\Q^*_\eps) \right) \d x
    &\leq  \int_{B^\prime} \left(\frac{1}{2}\abs{\nabla\Q_\eps}^2 
     + g_\eps(\Q_\eps) \right) \d x + \o_{\eps\to 0}(1) \\
    &\hspace{-.7cm} \stackrel{\eqref{strcnv8}}{=}  \
     \frac{1}{2}\int_{B^\prime}\abs{\nabla\Q^*}^2 \, \d x + \o_{\eps\to 0}(1) 
   \end{split}
  \end{equation}
  As we know already that~$\Q^*_\eps\rightharpoonup\Q^*$ weakly 
  in~$W^{1,2}(B^\prime)$ (by Lemma~\ref{lemma:compactnessQ}), 
  from~\eqref{strcnv6} we deduce that $\Q^*_\eps\to\Q^*$
  strongly in~$W^{1,2}(B^\prime)$ and~\eqref{strongconvW12} follows. 
  
  We turn to the proof of~\eqref{strongconvW1p}.
  Let~$p$, $q$ be such that~$1\leq p < q < 2$.
  Let~$\sigma>0$ be a small parameter, and let
  \[
   U_\sigma := \bigcup_{j=1}^{2\abs{d}} B_\sigma(a^*_j) 
   \cup \bigcup_{k=1}^K B_\sigma(b_k)
  \]
  By the H\"older inequality, we obtain
  \begin{equation*}
   \begin{split}
    \norm{\nabla\Q^*_\eps-\nabla\Q^*}_{L^p(\Omega)}
    \lesssim \abs{U_\sigma}^{1/p - 1/q} 
     \left(\norm{\nabla\Q^*_\eps}_{L^q(U_\sigma)}
     + \norm{\nabla\Q^*}_{L^q(U_\sigma)}\right) 
    + \norm{\nabla\Q^*_\eps-\nabla\Q^*}_{L^2(\Omega\setminus U_\sigma)}
   \end{split}
  \end{equation*}
  Thanks to Lemma~\ref{lemma:compactnessQ} and~\eqref{strongconvW12}, we deduce
  \begin{equation*} 
   \begin{split}
    \limsup_{\eps\to 0}\norm{\nabla\Q^*_\eps-\nabla\Q^*}_{L^p(\Omega)}
    \lesssim \sigma^{2/p - 2/q} 
   \end{split}
  \end{equation*}
  and, as~$\sigma$ may be taken arbitrarily small, \eqref{strongconvW1p} follows.
\end{proof}

\begin{remark}
 As a byproduct of the estimate~\eqref{strcnv6}, we deduce that
 $g_\eps(\Q^*_\eps) \to 0$ strongly in 
 $L^1_{\mathrm{loc}}(\Omega\setminus\{a_1, \, \ldots, \, a_{2\abs{d}}, \, b_1, \, \ldots, \, b_K\})$.
\end{remark}

We state an additional convergence property
for~$\Q^*_\eps$, which will be useful later on.
We recall that the vector product of two vectors~$\u\in\R^2$, $\v\in\R^2$
can be identified with a scalar, $\u\times\v := u_1v_2 - u_2v_1$.
In a similar way, we define the vector product of 
two matrices~$\Q\in\Sz$, $\P\in\Sz$ as
\begin{equation} \label{cross}
 \Q\times\P :=  Q_{11} P_{12} - Q_{12} P_{11} + Q_{21} P_{22} - Q_{22} P_{21} 
 = 2\left(Q_{11} P_{12} - Q_{12} P_{11}\right)
\end{equation}
{ If~$\mathbf{q}_1$, $\mathbf{q}_2$ (respectively,
$\mathbf{p}_1$, $\mathbf{p}_2$) are the columns of~$\Q$ (respectively, $\P$),
then
\[
 \Q\times\P = \mathbf{q}_1\times\mathbf{p}_1 + \mathbf{q}_2\times\mathbf{p}_2
\]
Alternatively, the vector product~$\Q\times\P$
can be expressed in terms of the commutator
$[\Q, \, \P] := \Q\P - \P\Q$, as
\[
 [\Q, \, \P] = (\Q\times \P)
 \left(
 \begin{matrix}
  0 & -1 \\
  1 & 0 \\
 \end{matrix}
 \right)
\]
} Now, for any~$\Q\in (L^\infty\cap W^{1,1})(\Omega, \, \Sz)$, we define
the vector field~$j(\Q)\colon\Omega\to\R^2$ as
\begin{equation} \label{preJac}
 j(\Q) := \frac{1}{2} \left(\Q\times\partial_1\Q, \, 
  \Q\times\partial_2\Q \right) 
\end{equation}
For any~$\Q\in (L^\infty\cap W^{1,1})(\Omega, \, \Sz)$, 
the vector field~$j(\Q)$ is integrable. Therefore, it makes
sense to define
\begin{equation} \label{Jac}
 J(\Q) := \partial_1 (j(\Q))_2 - \partial_2 (j(\Q))_1
\end{equation}
if we take the derivatives in the sense of distributions.
If~$\Q$ is smooth, then~$J(\Q)$ is the 
Jacobian determinant of~$\mathbf{q} := (\sqrt{2}Q_{11}, \sqrt{2}Q_{12})$:
\begin{equation} \label{JacQq}
 J(\Q) = 2 \, \partial_1 Q_{11} \, \partial_2 Q_{12} 
 - 2 \, \partial_2 Q_{11} \, \partial_1 Q_{12}
 = \det\nabla\mathbf{q}
\end{equation}
More generally, for any~$\Q\in (L^\infty\cap W^{1,1})(\Omega, \, \R^2)$,
$J(\Q)$ coincides with the distributional Jacobian of~$\mathbf{q}$
(see e.g.~\cite{JerrardSoner-Jacobians} and the references therein).

\begin{lemma} \label{lemma:Jac}
 We have
 \[
  J(\Q^*_\eps) \to J(\Q^*) 
  = \pi \sign(d) \sum_{j=1}^{2\abs{d}} \delta_{a^*_j}
  \qquad \textrm{in } W^{-1,1}(\Omega) 
 \]
 as~$\eps\to 0$.
\end{lemma}
\begin{proof}
 Let~$\mathbf{q}^* := ({\sqrt{2}}Q^*_{11}, \, {\sqrt{2}}Q^*_{12})$.
 By Lemma~\ref{lemma:compactnessQ}, the vector field~$\mathbf{q}^*$
 belongs to $W^{1,1}(\Omega, \, \SS^1)$ (globally in~$\Omega$) and to 
 $W^{1,2}_{\mathrm{loc}}(\Omega\setminus\{a^*_1, \, \ldots, \, 
 a^*_{2\abs{d}}\}, \, \SS^1)$. At each point~$a^*_j$, $\mathbf{q}^*$
 has a singularity of degree~$2\deg(\Q^*, \, \partial B_\sigma(a^*_j)) 
 = \sign(d)$, due to~\eqref{degree}. By reasoning e.g. 
 as in~\cite[Example~3.1]{JerrardSoner-Jacobians}, we obtain
 \begin{equation} \label{Jac1}
  J(\Q^*) = \pi \sign(d) \sum_{j=1}^{2\abs{d}} \delta_{a^*_j}
 \end{equation}
 It remains to show that $J(\Q^*_\eps)\to J(\Q^*)$ in~$W^{-1,1}(\Omega)$.
 Let~$p\in [1, \, 2)$ and~$q\in (2, \, +\infty]$ be
 such that $1/p + 1/q = 1$. By, e.g., \cite[Theorem~1]{BrezisNguyen}, we have
 \begin{equation} \label{Jac2}
  \begin{split}
   \norm{ J(\Q^*_\eps) - J(\Q^*)}_{W^{-1,1}(\Omega)}
   \leq C \norm{\Q^*_\eps - \Q^*}_{L^q(\Omega)} 
    \left(\norm{\nabla\Q^*_\eps}_{L^p(\Omega)} 
     + \norm{\nabla\Q^*}_{L^p(\Omega)}\right)
  \end{split}
 \end{equation}
 for some constant~$C$ that depends only on~$\Omega$.
 The sequence~$\Q^*_\eps$ is bounded in~$W^{1,p}(\Omega)$, 
 by Lemma~\ref{lemma:compactnessQ}. By compact Sobolev
 embedding, we have~$\Q^*_\eps \to \Q^*$ pointwise a.e.,
 up to extraction of a subsequence. As~$\Q^*_\eps$
 is also bounded in~$L^\infty(\Omega)$, by Lemma~\ref{lemma:max},
 we deduce that~$\Q^*_\eps\to\Q^*$ strongly in~$L^q(\Omega)$
 (via Lebesgue's dominated convergence theorem). Then,
 \eqref{Jac2} implies that~$J(\Q^*_\eps)\to J(\Q^*)$ in~$W^{-1,1}(\Omega)$
 and the lemma follows.
\end{proof}

\subsection{Proof of Statement~(ii): $\Q^*$ is a canonical harmonic map}

Next, we show that~$\Q^*$ is the canonical harmonic map with
singularities at~$(a^*_1, \, \ldots, \, a^*_{2\abs{d}})$
and boundary datum~$\Qb$, as defined in Section~\ref{sect:statement}.
The proof relies on an auxiliary lemma.

\begin{lemma} \label{lemma:equation}
 The minimisers~$(\Q^*_\eps, \, \M^*_\eps)$ satisfy
 \[
  - \partial_j\left(\Q^*_\eps\times\partial_j\Q^*_\eps \right)
  = \frac{\eps}{2} \,
   \partial_j\left(\M^*_\eps\times\partial_j\M^*_\eps\right)
   \qquad \textrm{in } \Omega.
 \]
\end{lemma}
\begin{proof}
 For ease of notation, we drop the subscript~$\eps$
 and the superscript~$^*$ from all the variables.
 We consider the Euler-Lagrange equation for~$\Q$, Equation~\eqref{EL-Q},
 and take the vector product with~$\Q$:
 \begin{equation} \label{lemeq1}
  -\Q\times\Delta\Q 
  - \dfrac{\beta}{\eps}\Q\times\left(\M\otimes\M 
  - \dfrac{\abs{\M}^2}{2}\I\right)  = 0 
 \end{equation}
 We have
 \begin{equation} \label{lediamounnome}
  \Q\times\Delta\Q 
  = \partial_j\left(\Q\times\partial_j\Q\right) - \partial_j\Q\times\partial_j\Q
  = \partial_j\left(\Q\times\partial_j\Q\right)
 \end{equation}
 and
 \begin{equation*}
  \begin{split}
   \Q\times\left(\M\otimes\M - \dfrac{\abs{\M}^2}{2}\I\right)
   = 2Q_{11} M_1 M_2 - Q_{12} M_1^2 + Q_{12} M_2^2
   = \Q\M\times\M
  \end{split}
 \end{equation*}
 so Equation~\eqref{lemeq1} rewrites as
 \begin{equation} \label{lemeq2}
  - \partial_j\left(\Q\times\partial_j\Q\right)
  = \dfrac{\beta}{\eps}\Q\M\times\M
 \end{equation}
 Now, we consider the Euler-Lagrange equation for~$\M$, 
 Equation~\eqref{EL-Q}, and take the vector product with~$\M$:
 \begin{equation} \label{lemeq3}
  -\M\times\Delta\M - \dfrac{2\beta}{\eps^2}\M\times\Q\M = 0.
 \end{equation}
 { Similarly to~\eqref{lediamounnome}}, we
 have~$\M\times\Delta\M = \partial_j(\M\times\partial_j\M)$,
 so~\eqref{lemeq3} can be written as
 \begin{equation} \label{lemeq4}
  \partial_j(\M\times\partial_j\M) 
   = \dfrac{2\beta}{\eps^2}\Q\M\times\M
 \end{equation}
 The lemma follows from~\eqref{lemeq2} and~\eqref{lemeq4}.
\end{proof}

\begin{prop} \label{prop:canonical}
 $\Q^*$ is the canonical harmonic map
 with singularities at $(a^*_1, \, \ldots, \, a^*_{2\abs{d}})$
 and boundary datum~$\Qb$.
\end{prop}
\begin{proof}
 First, we show that~$\Q^*$ satisfies
 \begin{equation} \label{canonical1}
  \partial_j\left(\Q^*\times\partial_j\Q^*\right) = 0
 \end{equation}
 in the sense of distributions in~$\Omega$. 
 To this end, we pass to the limit in both sides of Lemma~\ref{lemma:equation}.
 Let~$p\in(1, \, 2)$. By Lemma~\ref{lemma:compactnessQ}, we have
 $\Q^*_\eps\rightharpoonup\Q^*$ weakly in~$W^{1,p}(\Omega)$
 and, up to extraction of subsequences, pointwise a.e.
 As~$\Q^*_\eps$ is bounded in~$L^\infty(\Omega)$ by Lemma~\ref{lemma:max},
 Lebesgue's dominated convergence theorem implies
 that $\Q^*_\eps\to\Q^*$ strongly in~$L^q(\Omega)$
 for any~$q < +\infty$. As a consequence, we have
 \begin{equation} \label{canonical2}
  \partial_j\left(\Q^*_\eps\times\partial_j\Q^*_\eps\right)
  \rightharpoonup^* \partial_j\left(\Q^*\times\partial_j\Q^*\right)
  \qquad \textrm{as distributions in } \Omega  \textrm{ as } \eps\to 0.
 \end{equation}
 On the other hand, Proposition~\ref{prop:energy_upperbd} implies
 \begin{equation*}
  \norm{\nabla\M^*_\eps}_{L^2(\Omega)}^2 
  \leq \frac{1}{\eps} \mathscr{F}_\eps(\Q^*_\eps, \, \M^*_\eps) 
  \lesssim \frac{\abs{\log\eps}}{\eps}
 \end{equation*}
 As~$\M^*_\eps$ is bounded in~$L^\infty(\Omega)$ by Lemma~\ref{lemma:max},
 we deduce
 \begin{equation*}
  \eps\norm{\M^*_\eps\times\nabla\M^*_\eps}_{L^2(\Omega)}
  \leq \eps \norm{\M^*_\eps}_{L^\infty(\Omega)}
   \norm{\nabla\M^*_\eps}_{L^2(\Omega)}
  \lesssim \eps^{1/2}\abs{\log\eps}^{1/2}\to 0 
 \end{equation*}
 as~$\eps\to 0$. Therefore,
 \begin{equation} \label{canonical3}
  \eps\,\partial_j\left(\M^*_\eps\times\partial_j\M^*_\eps\right)
  \to 0 \qquad \textrm{in } W^{-1,2}(\Omega) \textrm{ as } \eps\to 0.
 \end{equation}
 Combining~\eqref{canonical2} and~\eqref{canonical3} with 
 Lemma~\ref{lemma:equation}, we obtain~\eqref{canonical1}.
 
 To prove that~$\Q^*$ is canonical harmonic, it only remains to check
 that~$\Q^*$ is smooth in~$\Omega\setminus\{a^*_1, \, \ldots, \, a^*_{2\abs{d}}\}$
 and continuous in~$\overline{\Omega}\setminus\{a^*_1, \, \ldots, 
 \, a^*_{2\abs{d}}\}$. Both these properties follow from~\eqref{canonical1}.
 Indeed, let~$G\subseteq\overline{\Omega}\setminus\{a^*_1, \, \ldots,
 \, a^*_{2\abs{d}}\}$ be a simply connected domain. 
 As~$\Q^*\in W^{1,2}(G, \, \NN)$, we can apply lifting
 results (see e.g.~\cite[Theorem~1]{BethuelChiron}) and write
 \begin{equation} \label{can1}
  \Q^*= \frac{1}{\sqrt{2}} \left(
  \begin{matrix}
   \cos\theta^* & \sin\theta^* \\
   \sin\theta^* & -\cos\theta^*
  \end{matrix} \right)
 \end{equation}
 for some scalar function~$\theta^*\in W^{1,2}(G)$. 
 Equation~\eqref{canonical1} may be written in terms of~$\theta^*$ as
 \begin{equation} \label{can2}
  \Delta\theta^* = 0 \qquad \textrm{as distributions in } G.
 \end{equation}
 Therefore, $\theta^*$ is smooth in~$G$ and so is~$\Q^*$.
 In case~$G$ touches the boundary of~$\Omega$, 
 $\theta^*$ is continuous up to~$\partial\Omega$ and hence~$\Q^*$ is.
\end{proof}

\subsection{Proof of Statements~(iii) and~(iv): compactness for~$\M^*_\eps$}

In this section, we prove a compactness result
for the component~$\M^*_\eps$ of a sequence of minimisers.
The proof relies on the change of variables we
introduced in Section~\ref{sect:changevar}.

We recall that in Lemma~\ref{lemma:compactnessQ},
we found a finite number of points~$a^*_1$, \ldots, $a^*_{2\abs{d}}$, $b^*_1$, \ldots, $b^*_K$
such that~$\abs{\Q^*_\eps}$ is uniformly bounded away from zero,
except for some small balls of radius~$\sigma$
around these points. Let
\[
 G\csubset\Omega\setminus\{a^*_1, \, \ldots, \, a^*_{2\abs{d}},
 \, b^*_1, \, \ldots, \, b^*_K\}
\]
be a smooth, simply connected domain. 
The sequence of minimisers~$(\Q^*_\eps, \, \M^*_\eps)$
satisfies the assumptions~\eqref{hp:chvar-energy}--\eqref{hp:chvar-abs},
thanks to Lemma~\ref{lemma:max}, 
Proposition~\ref{prop:energy_upperbd} and Lemma~\ref{lemma:compactnessQ}.
Therefore, we are in position to apply the results 
from Section~\ref{sect:changevar}.
We define the vector field~$\u^*_\eps\colon G\to\R^2$ as in~\eqref{uQM}
--- that is, we write
\begin{equation} \label{compM0}
 \Q^*_\eps = \frac{\abs{\Q^*_\eps}}{\sqrt{2}} 
  \left(\n^*_\eps\otimes\n^*_\eps - \m^*_\eps\otimes\m^*_\eps\right) 
  \qquad \textrm{in } G,
\end{equation}
where~$(\n^*_\eps, \, \m^*_\eps)$ is an orthonormal 
set of eigenvectors for~$\Q^*$ with $\n^*_\eps\in W^{1,2}(G, \, \SS^1)$, 
$\m^*_\eps\in W^{1,2}(G, \, \SS^1)$,
and we define
\begin{equation} \label{uQM*}
 (u_\eps^*)_1 := \M_\eps^*\cdot\n_\eps^*, \qquad
 (u_\eps^*)_2 := \M_\eps^*\cdot\m_\eps^*
\end{equation}
The next lemma is key to prove compactness of
the sequence~$\u^*_\eps$ and, hence, of~$\M^*_\eps$.

\begin{lemma} \label{lemma:boundu}
 Let~$h$ be the function defined by~\eqref{h}. 
 For any simply connected domain~$G\csubset\Omega\setminus\{a^*_1, \, 
 \ldots, \, a^*_{2\abs{d}}, \, b^*_1, \, \ldots, \, b^*_K\}$, there holds
 \[
   \int_G\left(\frac{\eps}{2}\abs{\nabla\u^*_\eps}^2
   + \frac{1}{\eps}h(\u^*_\eps) \right) \d x \leq C, 
 \]
 where~$C$ is a positive constant that depends 
 only on~$\Omega$, $\beta$ and the boundary datum
 (in particular, it is independent of~$\eps$, $G$).
\end{lemma}
\begin{proof}
 By classical lower bounds in the Ginzburg-Landau theory,
 such as \cite[Theorem~1.1]{Jerrard} or~\cite[Theorem~2]{Sandier},
 we have
 \begin{equation} \label{ubdd0}
  \begin{split}
   \int_\Omega\left(\frac{1}{2}\abs{\nabla\Q^*_\eps}^2
   + \frac{1}{4\eps^2}(\abs{\Q^*_\eps}^2 - 1)^2\right)\d x
   \geq 2\pi\abs{d}\abs{\log\eps} - C,
  \end{split}
 \end{equation}
 for some constant~$C$ that depends only on~$\Omega$
 and the boundary datum~$\Qb$. The results in~\cite{Jerrard, Sandier}
 extend to our setting due to change of 
 variables~$\Q^*_\eps \mapsto \mathbf{q}^*_{\eps}$,
 given by~\eqref{GLchangevar}. The coefficient~$2\pi\abs{d}$
 in the right-hand side of~\eqref{ubdd0} depends 
 on this change of variables, which transforms the boundary 
 condition of degree~$d$ for~$\Q^*_\eps$ into a 
 boundary condition of degree~$2d$ for~$\mathbf{q}^*_{\eps}$
 --- see~\eqref{degreeq}. 
 
 From~\eqref{ubdd0} and Lemma~\ref{lemma:delpino}, we deduce
 \begin{equation} \label{ubdd1}
  \frac{1}{2}\int_\Omega \abs{\nabla\Q^*_\eps}^2 \, \d x
   \geq 2\pi\abs{d}\abs{\log\eps} - C
 \end{equation}
 and then, by Proposition~\ref{prop:energy_upperbd},
 \begin{equation} \label{ubdd2}
  \int_\Omega\left(\frac{\eps}{2}\abs{\nabla\M^*_\eps}^2
   + \frac{1}{\eps^2}f_\eps(\Q^*_\eps, \, \M^*_\eps)\right)\d x \leq C,
 \end{equation}
 for some constant~$C$ that depends only on the domain and the boundary data.
 
 Now, we apply Proposition~\ref{prop:chvar}:
 \begin{equation} \label{ubdd3}
  \begin{split}
   \mathscr{F}_\eps(\Q^*_\eps, \, \M^*_\eps; \, G)
   &\geq \int_G\left(\frac{1}{2}\abs{\nabla\Q^*_\eps}^2
   + g_\eps(\Q^*_\eps) \right) \d x \\
   &\qquad\qquad + \frac{1}{2} \int_G\left(\frac{\eps}{2}\abs{\nabla\u^*_\eps}^2
   + \frac{1}{\eps}h(\u^*_\eps) \right) \d x + \o(1)
  \end{split}
 \end{equation}
 We have used~\eqref{changevar-R}
 and the elementary inequality $ab \leq a^2/2 + b^2/2$
 to estimate the remainder term~$R_\eps$.
 From~\eqref{ubdd3}, we obtain
 \begin{equation} \label{ubdd4}
  \begin{split}
   &\int_G\left(\frac{\eps}{2}\abs{\nabla\u^*_\eps}^2
    + \frac{1}{\eps}h(\u^*_\eps) \right) \d x 
    + 2\int_G g_\eps(\Q^*_\eps) \, \d x \\ 
   &\qquad\qquad \leq 2\int_G \left(\frac{\eps}{2}\abs{\nabla\M^*_\eps}^2
    + \frac{1}{\eps^2}f_\eps(\Q^*_\eps, \, \M^*_\eps)\right)\d x + \o(1)
   \stackrel{\eqref{ubdd2}}{\leq} C
  \end{split}
 \end{equation}
 Lemma~\ref{lemma:geps} gives~$g_\eps\geq 0$, so the lemma follows.
\end{proof}

\begin{prop} \label{prop:compactnessM}
 There exist a map~$\M^*\in\SBV(\Omega, \, \R^2)$
 and a (non-relabelled) subsequence such that
 $\M^*_\eps\to\M^*$ a.e.~and strongly
 in~$L^p(\Omega, \, \R^2)$ for any~$p<+\infty$,
 as~$\eps\to 0$. Moreover, $\H^1(\S_{\M^*})<+\infty$
 and~$\M^*$ satisfies
 \begin{gather} 
  \abs{\M^*} = (\sqrt{2}\beta + 1)^{1/2}, \label{M*-abs} \\
  \Q^* = \sqrt{2}\left(\frac{\M^*\otimes\M^*}{\sqrt{2}\beta + 1}
   - \frac{\I}{2} \right) \label{M*}
 \end{gather}
 a.e.~on~$\Omega$.
\end{prop}
\begin{proof}
 Let~$G\csubset\Omega\setminus\{a^*_1, \, \ldots, \, a^*_{2\abs{d}}, 
 \, b^*_1, \, \ldots, \, b^*_K\}$.
 By Proposition~\ref{prop:strongconv},
 we have~$\Q^*_\eps\to\Q^*$ strongly in~$W^{1,2}(G)$ and,
 up to extraction of a subsequence, pointwise a.e. in~$G$.
 By differentiating the identity~\eqref{compM0}, we obtain
 that
 \[
  \abs{\nabla\n_\eps^*}^2 = \abs{\nabla\m^*_\eps}^2
  \lesssim \abs{\nabla\left(\frac{\Q_\eps^*}{\abs{\Q_\eps^*}}\right)}^2
  \lesssim \abs{\nabla\Q_\eps^*}^2
 \]
 (the last inequality follows because~$\abs{\Q_\eps^*}\geq 1/2$ in~$G$,
 by Lemma~\ref{lemma:compactnessQ}). In particular,
 $\n_\eps^*$, $\m_\eps^*$ are bounded in~$W^{1,2}(G)$.
 Therefore, there exists vector fields~$\n^*\in W^{1,2}(G, \, \SS^1)$,
 $\m^*\in W^{1,2}(G, \, \SS^1)$ such that, up to extraction of a subsequence,
 there holds 
 \begin{equation} \label{compM2}
  \n^*_\eps\rightharpoonup\n^*, \quad 
  \m^*_\eps\rightharpoonup\m^* \qquad \textrm{weakly in } W^{1,2}(G)
  \textrm{ and pointwise a.e. in } G.
 \end{equation}
 By passing to the limit pointwise a.e. in~\eqref{compM0}, we obtain
 \begin{equation} \label{compM1}
  \Q^* = \frac{1}{\sqrt{2}} 
   \left(\n^*\otimes\n^* - \m^*\otimes\m^*\right) 
   \qquad \textrm{in } G,
 \end{equation}
 hence~$(\n^*, \, \m^*)$ is an orthonormal set of eigenvectors
 for~$\Q^*$. In fact, $\n^*$, $\m^*$ must be smooth,
 because~$\Q^*$ is smooth (by Proposition~\ref{prop:canonical}).
 
 Lemma~\ref{lemma:boundu}, combined
 with compactness results for the vectorial Modica-Mortola functional
 (see e.g.~\cite{Baldo1990} or~\cite[Theorems~3.1 and~4.1]{FonsecaTartar}),
 implies that there exists a (non-relabelled)
 subsequence and a map~$\u^*\in \BV(G, \, \R^2)$ such that 
 \begin{equation} \label{compM3}
  \u^*_\eps\to\u^* \quad \textrm{strongly in } L^1(G)
  \textrm{ and a.e. in } G, \qquad 
  h(\u^*) = 0 \quad \textrm{a.e. in } G
 \end{equation}
 and
 \begin{equation} \label{compM8}
  \H^1(\S_{\u^*}\cap G) 
   \lesssim \liminf_{\eps\to 0} \int_G \left(\frac{\eps}{2}\abs{\nabla\u^*_\eps}^2 
    + \frac{1}{\eps} h(\u^*_\eps)\right) \d x \leq C
 \end{equation}
 for some constant~$C$ that does not depend on~$G$.
 As~$h(\u^*) = 0$ a.e., necessarily~$\u^*$ must take the form
 \[
  \u^*(x) = \left(\tau(x) \, (\sqrt{2}\beta + 1)^{1/2} , \, 0\right)
  \qquad \textrm{for a.e. } x\in G,
 \]
 where~$\tau(x)\in\{1, \, -1\}$ is a sign (see Lemma~\ref{lemma:h}).
 Since~$\u^*$ takes values in a finite set, the 
 distributional derivative~$\D\u^*$ must be concentrated on~$\S_{\u^*}$,
 so $\u^*\in\SBV(G, \, \R^2)$.
 
 We define
 \begin{equation} \label{compM4}
  \M^* := (u^*)_1 \, \n^* + (u^*)_2\,\m^*
  = \tau \, (\sqrt{2}\beta + 1)^{1/2} \, \n^*
  \qquad \textrm{in } G.
 \end{equation}
 The vector field~$\M^*$ is well-defined
 and does not depend on the choice of 
 the orientation for~$\n^*_\eps$, $\m^*_\eps$ (so long as the orientation
 is chosen consistently as~$\eps\to 0$, in such a way 
 that~\eqref{compM2} is satisfied). Indeed, if we 
 replace~$\n^*_\eps$ by~$-\n^*_\eps$, then also~$(\u^*_\eps)_1$
 will change its sign and the product at the right-hand side 
 of~\eqref{compM4} will remain unaffected. Therefore,
 by letting~$G$ vary in~$\Omega\setminus\{a^*_1, \, \ldots, 
 \, a^*_{2\abs{d}}, \, b^*_1, \, \ldots, \, b^*_K\}$,
 we can define~$\M^*$ almost everywhere in~$\Omega$. 
 An explicit computation, based on~\eqref{compM1} and~\eqref{compM4},
 shows that~$\M^*$ satisfies~\eqref{M*-abs} and~\eqref{M*}.
 Moreover, due to~\eqref{compM2} and~\eqref{compM3},
 we have $\M^*_\eps\to\M^*$ a.e.~in~$G$. As the sequence~$\M_\eps$
 is uniformly bounded in~$L^\infty(\Omega)$ (by Lemma~\ref{lemma:max}),
 Lebesgue's dominated convergence theorem implies that
 $\M^*_\eps\to\M^*$ in~$L^p(\Omega)$ for any~$p<+\infty$.
 
 As we have seen, $\u^*\in\SBV(G, \, \R^2)$ for 
 any~$G\csubset\Omega\setminus\{a^*_1, \, \ldots, \, a^*_{2\abs{d}}, 
 \, b^*_1, \, \ldots, \, b^*_K\}$.
 Therefore, by applying the BV-chain rule
 (see e.g.~\cite[Theorem~3.96]{AmbrosioFuscoPallara})
 to~\eqref{compM6}, and letting~$G$ vary, we obtain
 \begin{equation} \label{compM5}
  \M^*\in\SBV_{\mathrm{loc}}(\Omega\setminus\{a^*_1, \, \ldots, 
   \, a^*_{2\abs{d}}, \, b^*_1, \, \ldots, \, b^*_K\}; \, \R^2)
 \end{equation}
 Moreover, we claim that
 \begin{equation} \label{compM6}
  \M^*\in\SBV(\Omega\setminus\{a^*_1, \, \ldots, 
   \, a^*_{2\abs{d}}, \, b^*_1, \, \ldots, \, b^*_K\}; \, \R^2),
   \qquad \H^1(\S_{\M^*}) < +\infty
 \end{equation}
 Indeed, the absolutely continuous part~$\nabla\M^*$
 of the distributional derivative~$\D\M^*$ can 
 be bounded by differentiating~\eqref{M*}:
 the BV-chain rule implies
 \begin{equation} \label{nablaM*}
  \abs{\nabla\M^*} = \frac{\sqrt{2}\beta + 1}{2}\abs{\nabla\Q^*},
 \end{equation}
 and hence,
 \begin{equation*} 
  \norm{\nabla\M^*}_{L^1(\Omega)} 
   \leq \frac{\sqrt{2}\beta + 1}{2} \norm{\nabla\Q^*}_{L^1(\Omega)} < +\infty
 \end{equation*}
 due to Lemma~\ref{lemma:compactnessQ}.
 The total variation of the jump part of~$\D\M^*$
 is uniformly bounded, too, because of~\eqref{compM8}
 (the constant at the right-hand side of~\eqref{compM8}
 does not depend on~$G$, so we may take the limit 
 as~$G\searrow\Omega$).  Then, \eqref{compM6} follows.
 
 In order to complete the proof, it only remains to show that
 $\M\in\SBV(\Omega, \, \R^2)$. 
 Let~$\varphi\in C^\infty_{\mathrm{c}}(\Omega)$ be a test function,
 and let~$\sigma > 0$ be fixed. We define
 \[
  U_\sigma := \bigcup_{i=1}^{2\abs{d}} B_{\sigma}(a_i)
  \cup \bigcup_{k=1}^K B_\sigma(b_k)
 \]
 We choose a smooth cut-off function~$\psi_\sigma$
 such that $0 \leq \psi_\sigma \leq 1$ in~$\Omega$,
 $\psi_\sigma = 0$ in~$\Omega\setminus U_\sigma$, 
 $\psi_\sigma = 1$ in a neighbourhood of each point 
 $a_1, \, \ldots, \, a_{2\abs{d}}$, $b_1, \, \ldots, \, b_K$,
 and~$\norm{\nabla\psi_\sigma}_{L^\infty(\Omega)}\leq C\sigma$
 for some constant~$C$ that does not depend on~$\sigma$. 
 Then, for~$j\in\{1, \, 2\}$, we have
 \[
  \begin{split}
   \int_{\Omega} \M^* \, \partial_j\varphi
   = \int_\Omega \M^* \, \partial_j\left(\varphi(1 - \psi_\sigma)\right)
   + \int_\Omega \M^* \left(\psi_\sigma \, \partial_j\varphi
    + \varphi \, \partial_j\psi_\sigma\right)
  \end{split}
 \]
 We bound the first term in the right-hand side
 by applying~\eqref{compM6}. To estimate the second term,
 we observe that the integrand is bounded and supported 
 in~$U_\sigma$. Therefore, we obtain
 \begin{equation} \label{compM9}
  \begin{split}
   \int_{\Omega} \M^* \, \partial_j\varphi
   &\lesssim \norm{\varphi}_{L^\infty(\Omega)}
   + \norm{\M^*}_{L^\infty(\Omega)} 
    \norm{\nabla\varphi}_{L^\infty(\Omega)}\abs{U_\sigma}  \\
   &\hspace{2.1cm} + \norm{\M^*}_{L^\infty(\Omega)} \norm{\varphi}_{L^\infty(\Omega)} \norm{\nabla\psi_\sigma}_{L^\infty(\Omega)} \abs{U_\sigma} \\
   &\lesssim \norm{\varphi}_{L^\infty(\Omega)}
   + \norm{\M^*}_{L^\infty(\Omega)} \left(\sigma^2
    \norm{\nabla\varphi}_{L^\infty(\Omega)}
   + \sigma \norm{\varphi}_{L^\infty(\Omega)} \right)
  \end{split}
 \end{equation}
 By taking the limit as~$\sigma\to 0$, we deduce that
 $\M^*\in\BV(\Omega, \, \R^2)$. In fact, we must have~$\M^*\in\SBV(\Omega, \, \R^2)$, because the Cantor part of~$\D\M^*$ cannot be 
 supported on a finite number of points, $a_1, \, \ldots, \, a_{2\abs{d}}$,
 $b_1, \, \ldots, \, b_K$. This completes the proof.
\end{proof}

{ We conclude this section by stating a regularity property
of~$\M^*$. We recall that a harmonic map~$\M$ on a
domain~$U\subseteq\R^2$ with values in a circle of radius~$R > 0$
is a map that can be written in the form~$\M = (R\cos\phi, \, R\sin\phi)$
for some harmonic function~$\phi\colon U\to\R$.
Let~$\overline{\S_{\M^*}}$ be the closure of the jump set of~$\M^*$.

\begin{prop} \label{prop:harmonicM*}
 The map~$\M^*$ is locally harmonic 
 on~$\Omega\setminus\overline{\S_{\M^*}}$, 
 with values in the circle of radius~$(\sqrt{2}\beta + 1)^{1/2}$.
 In particular, $\M^*$ is smooth in~$\Omega\setminus\overline{\S_{\M^*}}$.
\end{prop}
\begin{proof}
 Let~$B\subseteq\Omega$ be an open ball that does not 
 intersect~$\overline{\S_{\M^*}}$
 nor~$\{a_1^*, \, \ldots, \, a^*_{2\abs{d}}, \, b^*_1, \, \ldots, \, b^*_K\}$. 
 Then, we have~$\M^*\in W^{1,2}(B, \, \R^2)$, by construction
 (see, in particular, \eqref{compM2} and~\eqref{compM4}).
 By lifting results (see e.g.~\cite{BethuelZheng, BethuelChiron, BallZarnescu}), $\M^*$ can be written in the form
 $\M^* = (\sqrt{2}\beta + 1)^{1/2}(\cos\phi^*, \, \sin\phi^*)$,
 for some scalar function~$\phi^*\in W^{1,2}(B, \, \R)$.
 On the other hand, the condition~\eqref{M*} shows that~$\phi^*$
 is uniquely determined by~$\Q^*$, up to constant multiples of~$\pi$.
 In particular, we must have~$\phi^* = \theta^*/2 + k\pi$,
 where~$\theta^*$ is the function given by~\eqref{can1}
 and~$k$ is a constant.
 Then, Equation~\eqref{can2} implies that~$-\Delta\phi^* = 0$
 in~$B$ and hence, $\M^*$ is a harmonic map on~$B$ with values
 in the circle of radius~$(2\sqrt{\beta} + 1)^{1/2}$.
 
 Now, let~$B$ be an open ball that does not intersect~$\overline{\S_{\M^*}}$
 nor~$\{a_1^*, \, \ldots, \, a^*_{2\abs{d}}\}$,
 although it may contain one of the points~$b_k$.
 Say, for simplicity, that~$B$ contains exactly
 one of the points~$b_k$. We claim that~$\M^*$
 is harmonic in~$B$, too. Indeed, since~$b_k$ is a singularity 
 of degree zero (see~\eqref{degree}), we can repeat the arguments above 
 and write $\M^* = (\sqrt{2}\beta + 1)^{1/2}(\cos\phi^*, \, \sin\phi^*)$
 in~$B\setminus\{b_k\}$, for some harmonic 
 function~$\phi^*\colon B\setminus\{b_k\}\to\R$.
 By the chain rule, $\abs{\nabla\phi^*}$
 coincides with~$\abs{\nabla\Q^*}$ up to a constant factor
 (see~\eqref{nablaM*}). The map~$\Q^*$ is smooth in a 
 neighbourhood of~$b_k$, because it is canonical harmonic
 with singularities at~$\{a_1, \, \ldots, \, a_{2\abs{d}}\}$.
 Therefore, $\nabla\phi^*$ is bounded in~$B\setminus\{b_k\}$.
 As a consequence, $b_k$ is a removable singularity
 for~$\phi^*$ and, by possibly 
 modifying the value of~$\M^*$ at~$b_k$,
 $\M^*$ is harmonic in~$B$. 
 
 To conclude the proof, it only remains to show
 that the points~$\{a_1^*, \, \ldots, \, a^*_{2\abs{d}}\}$ 
 are contained in~$\overline{\S_{\M^*}}$.
 If any of the points~$a_j$ did not belong to~$\overline{\S_{\M^*}}$,
 then $\M^*$ would be locally harmonic (and hence, smooth) in a
 sufficiently small neighbourhood of~$a_j$, except at the point~$a_j$.
 This is impossible, because $a_j$ is a non-orientable 
 singularity of~$\Q^*$ (see~\eqref{degree}) and there cannot be 
 a map~$\M^*$ that satisfies~\eqref{M*-abs}, \eqref{M*}
 and is continuous in a punctured neighbourhood of~$a_j$.
 Therefore, $a_j\in\overline{\S_{\M^*}}$.
\end{proof}

}

\subsection{Proof of Statements~(v) and~(vi): sharp energy estimates}
\label{sect:Gamma}

In this section, we complete the proof of Theorem~\ref{th:main},
by describing the structure of the jump set of~$\M^*$
and characterising the optimal position of the defects
of~$\Q^*$ (in case the domain~$\Omega$ is convex).
As a byproduct of our arguments, we will also show a refined
energy estimate for the minimisers~$(\Q^*_\eps, \, \M^*_\eps)$,
i.e. Proposition~\ref{prop:min_energy} below.

First, we set some notations. We let
\begin{equation} \label{c*}
 c_\beta := \frac{2\sqrt{2}}{3} \left(\sqrt{2}\beta + 1\right)^{3/2} 
\end{equation}
For any $(2\abs{d})$-uple of distinct points~$a_1$, \ldots, $a_{2\abs{d}}$
in~$\Omega$, we define
\begin{equation} \label{Wbeta}
 \mathbb{W}_\beta(a_1, \, \ldots, \, a_{2\abs{d}})
  := \mathbb{W}(a_1, \, \ldots, \, a_{2\abs{d}})
  + c_\beta \, \mathbb{L}(a_1, \, \ldots, \, a_{2\abs{d}})
\end{equation}
where~$\mathbb{W}$, $\mathbb{L}$ are, respectively,
the Ginzburg-Landau renormalised energy (defined in~\eqref{renormalised})
and the length of a minimal connection (defined in~\eqref{minconn-intro}).
We also recall the definition of the Ginzburg-Landau \emph{core energy},
which was introduced in~\cite{BBH}.
Let~$B_1\subseteq\R^2$ be the unit disk. For any~$\eps> 0$, let
\[
 \gamma(\eps) := \inf\left\{
  \int_{B_1} \left(\frac{1}{2}\abs{\nabla u}^2 
   + \frac{1}{4\eps^2} (\abs{u}^2 - 1)^2 \right)\d x \colon 
  u\in W^{1,2}(B_1, \, \C), \ u(x) = x \textrm{ for } x\in \partial B_1\right\}
\]
It can be proved (see~\cite[Lemma~III.3]{BBH}) that
the function~$\eps\mapsto\gamma(\eps) - \pi\abs{\log\eps}$
is finite in~$(0, \, 1)$ and non-decreasing.
Therefore, the limit
\begin{equation} \label{core_energy}
 \gamma_* := \lim_{\eps\to 0} \left(\gamma(\eps) - \pi\abs{\log\eps}\right)  > 0
\end{equation}
exists and is finite. The number~$\gamma_*$ is the so-called core energy.
In this section, we will prove the following result:

\begin{prop} \label{prop:min_energy}
 If the domain~$\Omega\subseteq\R^2$ is \emph{convex}, then 
 \begin{equation} \label{min_energy}
  \mathscr{F}_\eps(\Q^*_\eps, \, \M^*_\eps)
  = 2\pi\abs{d}\abs{\log\eps} 
   + \mathbb{W}_\beta(a^*_1, \, \ldots, \, a^*_{2\abs{d}})
   + 2\abs{d}\gamma_* + \mathrm{o}(1)
 \end{equation}
 as~$\eps\to 0$.
\end{prop}

We will prove the lower and upper inequality in~\eqref{min_energy}
separately. From now on, we alwasy assume that the domain~$\Omega$
is \emph{convex}.

\subsubsection{Sharp lower bounds for the energy of minimisers}
\label{sect:Gamma_liminf}

The aim of this section is to prove a sharp lower bound
for~$\mathscr{F}_\eps(\Q^*_\eps, \, \M^*_\eps)$.
We know from previous results
(Lemma~\ref{lemma:compactnessQ}, Proposition~\ref{prop:compactnessM}),
that, up to extraction of a subsequence, we have
$\Q^*_\eps\to\Q^*$, $\M^*_\eps\to\M^*$ a.e., where
\[
 \Q^*\in W^{1,2}_{\mathrm{loc}}(\Omega\setminus\{a^*_1, \, \ldots, \, a^*_{2\abs{d}}\}, \, \NN), \qquad
 \M^*\in\SBV(\Omega, \, \R^2)
\]
Due to Lemma~\ref{lemma:delpino}, we may further assume that
\begin{equation} \label{xi*}
 \frac{\abs{\Q^*_\eps} - 1}{\eps}\rightharpoonup \xi_*
 \qquad \textrm{weakly in } L^2(\Omega).
\end{equation}

\begin{prop} \label{prop:Gamma_liminf}
 There holds
 \[
  \begin{split}
  &\liminf_{\eps\to 0}
    \big(\mathscr{F}_\eps(\Q^*_\eps, \, \M^*_\eps) - 
    2\pi\abs{d}\abs{\log\eps} \big)  \\
  &\qquad\qquad\geq \mathbb{W}(a^*_1, \, \ldots, \, a^*_{2\abs{d}})
    + c_\beta\,\H^1(\S_{\M^*})
   + \int_{\Omega}(\xi_* - \kappa_*)^2 \, \d x
    + 2\abs{d}\gamma_* 
  \end{split}
 \]
 where the constants~$c_\beta$, $\kappa_*$ are given,
 respectively, by~\eqref{c*} and~\eqref{k*}.
\end{prop}

The length of the jump set~$\S_{\M^*}$ can be further 
bounded from below, in terms of the singular
points~$a^*_1, \, \ldots, \, a^*_{2\abs{d}}$.
We recall from Section~\ref{sect:statement}
that a connection for~$a^*_1, \, \ldots, \, a^*_{2\abs{d}}$
is a finite collection of straight line segments
$L_1, \, \ldots, \, L_{\abs{d}}$ that connects the
points~$a^*_j$ in pairs, and 
that~$\mathbb{L}(a^*_1, \, \ldots, \, a^*_{2\abs{d}})$
is the minimal length of a connection for the points~$a^*_j$
(see~\eqref{minconn-intro}).
Given two sets~$A$, $B$, we denote by~$A\Delta B$
their symmetric difference, i.e.~$A\Delta B := (A\setminus B)\cup(B\setminus A)$.

\begin{prop} \label{prop:lowerbound_SM}
 We have
 \begin{equation} \label{lowerbound_SM}
  \H^1(\S_{\M^*}) \geq \mathbb{L}(a^*_1, \, \ldots, \, a^*_{2d})
 \end{equation}
 The equality in~\eqref{lowerbound_SM} holds if and only if
 there exists a minimal connection~$\{L^*_1, \, \ldots, \, L^*_{\abs{d}}\}$
 for $\{a^*_1, \, \ldots, a^*_{2\abs{d}}\}$ such that
 \[
  \H^1\left(\S_{\M^*} \, \Delta \, \bigcup_{j=1}^d L^*_j\right) = 0
 \]
\end{prop}

We will give the proof of Proposition~\ref{prop:lowerbound_SM}
in Appendix~\ref{app:lifting}. Here, instead, we focus on the proof
of Proposition~\ref{prop:Gamma_liminf}.

\begin{lemma} \label{lemma:cost}
 Let~$h\colon\R^2\to\R$ be the function defined in~\eqref{h},
 and let~$\u_{\pm} := (\pm(\sqrt{2}\beta + 1)^{1/2}, \, 0)$.
 Then, there holds
 \begin{equation} \label{cost}
  \begin{split}
   &\inf\left\{ \int_0^1 \sqrt{2 h(\u(t))}\abs{\u^\prime(t)}\d t\colon
    \u\in W^{1,1}([0, \, 1], \, \R^2), \ \u(0) = \u_-, \ \u(1) = \u_+\right\}
    = c_\beta
  \end{split}
 \end{equation}
 with~$c_\beta$ given by~\eqref{c*}.
\end{lemma}
\begin{proof}
 Let~$\u\in W^{1,1}([0, \, 1], \, \R^2)$ be such
 that~$\u(0) = \u_-$, $\u(1) = \u_+$. We 
 define~$\tilde{\u}\colon [0, \, 1]\to\R^2$ as
 $\tilde{\u}(t) := (\abs{\u(t)}, \, 0)$. We have
 $h(\tilde{\u}(t))\leq h(\u(t))$ for any~$t$ and
 \[
  \abs{\tilde{\u}^\prime(t)} = \abs{\u^\prime(t)\cdot\frac{\u(t)}{\abs{\u(t)}}}
  \leq \abs{\u^\prime(t)}
 \]
 for a.e.~$t\in[0, \, 1]$ such that~$\u(t)\neq 0$.
 On the other hand, Stampacchia's lemma implies 
 that~$\u^\prime=0$ a.e. on the set~$\u^{-1}(0)$ and similarly,
 $\tilde{\u}^\prime=0$ a.e. on~$\tilde{\u}^{-1}(0)$.
 Therefore, we have
 \[
  \int_0^1 \sqrt{h(\tilde{\u}(t))}\abs{\tilde{\u}^\prime(t)} \d t
  \leq \int_0^1 \sqrt{h(\u(t))}\abs{\u^\prime(t)} \d t.
 \]
 As a consequence, in the left-hand side of~\eqref{cost}
 we can minimise under the additional constraint that~$u_2\equiv 0$,
 without loss of generality. In other words, we have shown 
 that
 \begin{equation*}
  \begin{split}
   I :=& \inf\left\{ \int_0^1 \sqrt{2 h(\u(t))}\abs{\u^\prime(t)}\d t\colon
    \u\in W^{1,1}([0, \, 1], \, \R^2), \ \u(0) = \u_-, \ \u(1) = \u_+\right\} \\
   =& \inf\bigg\{ \int_0^1 \sqrt{2 h(u_1(t), \, 0)} 
    \abs{u_1^\prime(t)} \d t\colon
    u_1\in W^{1,1}(0, \, 1), \  u_1(0) = -\lambda, \ u_1(1) = \lambda \bigg\} 
  \end{split}
 \end{equation*}
 where~$\lambda := (\sqrt{2}\beta + 1)^{1/2}$. 
 Equation~\eqref{h} implies, by an explicit computation,
 \[
  \sqrt{2 h(u_1, \, 0)} = \frac{1}{\sqrt{2}} \abs{\lambda^2 - u_1^2}
 \]
 By making the change of variable~$y = u_1(t)$, we deduce 
 \begin{equation} \label{cost1}
  \begin{split}
   I 
   &\geq \inf\bigg\{ \int_0^1 \sqrt{2 h(u_1(t), \, 0)} 
    \, u_1^\prime(t) \, \d t\colon
    u_1\in W^{1,1}(0, \, 1), \  u_1(0) = -\lambda, \ u_1(1) = \lambda \bigg\} \\
   &= \int_{-\lambda}^{\lambda} \sqrt{2 h(y, \, 0)} \,\d y
   = \frac{1}{\sqrt{2}} \int_{-\lambda}^{\lambda} 
    \left(\lambda^2 - y^2\right)\,\d y
   = \frac{2\sqrt{2}}{3} \lambda^3
  \end{split}
 \end{equation}
 We take as a competitor in~\eqref{cost}
 the map~$\v(t) := (-\lambda + 2t\lambda, \, 0)$.
 By similar computations, we obtain
 \begin{equation} \label{cost2}
  \begin{split}
   I \leq \int_0^1 \sqrt{2 h(\v(t))}\abs{\v^\prime(t)}\d t
   = \frac{2\sqrt{2}}{3} \lambda^3
  \end{split}
 \end{equation}
 and the lemma follows.
\end{proof}

\begin{lemma} \label{lemma:loc_Gamma_liminf}
 Let~$G\csubset\Omega\setminus\{a^*_1, \, \ldots,
 \, a^*_{2\abs{d}}, \, b^*_1, \, \ldots, \, b^*_K\}$
 be a simply connected domain. Then,
 \[
  \begin{split}
   \liminf_{\eps\to 0} \mathscr{F}_\eps(\Q^*_\eps, \, \M^*_\eps; \, G)
   &\geq \frac{1}{2} \int_G \abs{\nabla\Q^*}^2 \, \d x
    + \int_{G}(\xi_* - \kappa_*)^2 \, \d x 
    + c_\beta \, \H^1(\S_{\M^*}\cap G) 
  \end{split}
 \]
\end{lemma}
\begin{proof}
 We make a change of variable, as introduced in Section~\ref{sect:changevar}.
 { This is possible, because we have assumed that~$\Omega$
 is simply connected.}
 Let~$\u^*_\eps$ be the vector field defined in~\eqref{uQM}. 
 By Proposition~\ref{prop:chvar}, we have 
 \begin{equation} \label{locGinf1}
  \begin{split}
   \F_\eps(\Q^*_\eps, \, \M^*_\eps; \, G)
   &= \int_G\left(\frac{1}{2}\abs{\nabla\Q^*_\eps}^2
   + g_\eps(\Q^*_\eps) \right) \d x 
   + \int_G\left(\frac{\eps}{2}\abs{\nabla\u^*_\eps}^2
   + \frac{1}{\eps}h(\u^*_\eps) \right) \d x  + R_\eps
  \end{split}
 \end{equation}
 and the remainder term~$R_\eps$ satisfies
 \begin{equation} \label{locGinf2}
  \abs{R_\eps} \lesssim \eps^{1/2} \abs{\log\eps}^{1/2}
  \left(\int_G \left(\frac{\eps}{2}\abs{\nabla\u^*_\eps}^2
   + \frac{1}{\eps}h(\u^*_\eps)\right) \d x \right)^{1/2} + \o(1)
   \qquad \textrm{as } \eps\to 0.
 \end{equation}
 Lemma~\ref{lemma:boundu} implies that~$R_\eps\to 0$ 
 as~$\eps\to 0$. We estimate separately
 the other terms in the right-hand side of~\eqref{locGinf1}. 
 The weak convergence $\Q^*_\eps\rightharpoonup\Q^*$
 in~$W^{1,2}(G)$ implies
 \begin{equation} \label{locGinf3}
  \liminf_{\eps\to 0} \frac{1}{2}\int_G \abs{\nabla\Q^*_\eps}^2 \, \d x
  \geq  \frac{1}{2}\int_G \abs{\nabla\Q^*}^2 \, \d x
 \end{equation}
  We claim that
 \begin{equation} \label{locGinf6}
  \liminf_{\eps\to 0} \int_G g_\eps(\Q^*_\eps) \, \d x
  \geq \int_G (\xi_* - \kappa_*)^2 \, \d x
 \end{equation}
 Indeed,  Lemma~\ref{lemma:geps} gives
 \begin{equation} \label{locGinf7}
  g_\eps(\Q^*_\eps) = (\xi_\eps - \kappa_*)^2 + \xi_\eps^2 \, \zeta_\eps
 \end{equation}
 where
 \[
  \xi_\eps := \frac{1}{\eps} (\abs{\Q^*_\eps} - 1), 
  \qquad \zeta_\eps := \frac{1}{4}(\abs{\Q^*_\eps} + 1)^2 - 1 \geq - 1
 \]
 Let~$\delta>0$ be a small parameter. 
 By Lemma~\ref{lemma:compactnessQ}, 
 we have~$\abs{\Q^*_\eps}\to\abs{\Q^*} = 1$ a.e.~in~$\Omega$
 and hence, $\zeta_\eps \to 0$ a.e.~in~$G$.
 Therefore, by the Severini-Egoroff theorem, there exists a Borel set
 $\tilde{G}\subseteq G$ such that $|G\setminus\tilde{G}|\leq\delta$
 and~$\zeta_\eps\to 0$ uniformly in~$\tilde{G}$
 as~$\eps\to 0$. Now, we have
 \begin{equation*}
  \begin{split}
   \int_G g_\eps(\Q^*_\eps) \, \d x
   \stackrel{\eqref{locGinf7}}{\geq} 
    \int_{\tilde{G}} (\xi_\eps - \kappa_*)^2 \, \d x
    + \int_{\tilde{G}}\xi_\eps^2 \, \zeta_\eps \, \d x
    + \int_{G\setminus\tilde{G}} (-2\xi_\eps\,\kappa_* + \kappa_*^2) \, \d x
  \end{split}
 \end{equation*}
 The integral of~$\xi_\eps^2 \, \zeta_*$ on~$\tilde{G}$
 tends to zero, because~$\xi_\eps$ is bounded in~$L^2(G)$ 
 (by Lemma~\ref{lemma:delpino}) and $\zeta_\eps\to 0$ uniformly 
 in~$\tilde{G}$. As~$\xi_\eps\rightharpoonup\xi_*$ weakly in~$L^2(G)$
 (see~\eqref{xi*}), we deduce
 \begin{equation*}
  \begin{split}
   \liminf_{\eps\to 0}\int_G g_\eps(\Q^*_\eps) \, \d x
   \geq
    \int_{\tilde{G}} (\xi_* - \kappa_*)^2 \, \d x
    + \int_{G\setminus\tilde{G}} (-2\xi_*\,\kappa_* + \kappa_*^2) \, \d x
  \end{split}
 \end{equation*}
 The area of~$G\setminus\tilde{G}$ may be taken arbitrarily small,
 so~\eqref{locGinf6} follows.

 Finally, for the term in~$\u^*_\eps$, we apply classical
 $\Gamma$-convergence results for the vectorial Modica-Mortola functional
 (see e.g.~\cite{Baldo1990, FonsecaTartar}), as well as Lemma~\ref{lemma:cost}:
 \begin{equation} \label{locGinf4}
  \liminf_{\eps\to 0} \int_G\left(\frac{\eps}{2}\abs{\nabla\u_\eps}^2
   + \frac{1}{\eps}h(\u_\eps) \right) \d x
   \geq c_\beta \, \H^1(\S_{\u^*}\cap G)
 \end{equation}
 where~$\u^*$ is the $L^1(G)$-limit of the sequence~$\u^*_\eps$,
 as in~\eqref{compM3}. By~\eqref{compM4}, we have
 $\S_{\u^*} = \S_{\M^*}$ and hence,
 \begin{equation} \label{locGinf5}
  \liminf_{\eps\to 0} \int_G\left(\frac{\eps}{2}\abs{\nabla\u_\eps}^2
   + \frac{1}{\eps}h(\u_\eps) \right) \d x
   \geq c_\beta \, \H^1(\S_{\M^*}\cap G)
 \end{equation}
 Combining~\eqref{locGinf1}, \eqref{locGinf2}, \eqref{locGinf3},
 \eqref{locGinf6} and~\eqref{locGinf5}, the lemma follows.
\end{proof}

\begin{lemma} \label{lemma:core_Gamma_liminf}
 For any~$\sigma>0$ sufficiently small and 
 any~$j\in\{1, \, \ldots, \, 2\abs{d}\}$, we have
 \[
  \liminf_{\eps\to 0} \left(\mathscr{F}_\eps(\Q^*_\eps, \, \M^*_\eps;
  \, B_\sigma(a^*_j)) - \pi\abs{\log\eps}\right)
  \geq \pi \log\sigma + \gamma_* - C\sigma,
 \]
 where~$\gamma_*$ is the constant given by~\eqref{core_energy} 
 and~$C$ is a constant that does not depend on~$\eps$, $\sigma$. 
\end{lemma}
\begin{proof}
 Take~$\sigma$ is so small that the ball~$B_\sigma(a^*_j)$
 does not contain any other singular point~$a^*_k$, with~$k\neq j$.
 We consider the function~$J(\Q^*_\eps)$ defined in~\eqref{Jac}.
 By Lemma~\ref{lemma:Jac}, we have
 \[
  J(\Q^*_\eps) \to \pi \sign(d) \, \delta_{a^*_j} \qquad
  \textrm{in } W^{-1,1}(B_\sigma(a^*_j)).
 \]
 Then, we can apply pre-existing $\Gamma$-convergence
 results for the Ginzburg-Landau functional ---
 for instance, \cite[Theorem~5.3]{AlicandroPonsiglione}.
 We obtain a (sharp) lower bound for the Ginzburg-Landau energy of~$\Q^*_\eps$: 
 \begin{equation} \label{coreGinf1}
  \begin{split}
   \liminf_{\eps\to 0} \left(\int_{B_\sigma(a^*_j)} 
    \left( \frac{1}{2}\abs{\nabla\Q^*_\eps}^2 
    + \frac{1}{4\eps^2} (\abs{\Q^*_\eps}^2 - 1)^2\right) \d x
    - \pi\abs{\log\eps}\right) \geq \pi\log\sigma + \gamma_*
  \end{split}
 \end{equation}
 On the other hand, Lemma~\ref{lemma:feps} gives
 \begin{equation*} 
  \begin{split}
   \frac{1}{\eps^2} \int_{B_\sigma(a^*_j)} 
    f_\eps(\Q^*_\eps, \, \M^*_\eps) \, \d x 
    &\geq \frac{1}{4\eps^2}\int_{B_\sigma(a^*_j)}
     (\abs{\Q^*_\eps}^2 - 1)^2\, \d x
     - \frac{\beta}{\sqrt{2}\eps} 
     \int_{B_\sigma(a^*_j)} \abs{\M^*_\eps}^2 
     \abs{\abs{\Q^*_\eps} - 1}  \, \d x
  \end{split}
 \end{equation*}
 As~$\M^*_\eps$ is uniformly bounded in~$L^\infty(\Omega)$
 (by Lemma~\ref{lemma:max}), we obtain via the H\"older inequality
 \begin{equation*} 
  \begin{split}
   \frac{1}{\eps^2} \int_{B_\sigma(a^*_j)} 
    f_\eps(\Q^*_\eps, \, \M^*_\eps) \, \d x 
    &\geq \frac{1}{4\eps^2}\int_{B_\sigma(a^*_j)}
     (\abs{\Q^*_\eps}^2 - 1)^2\, \d x
     - C\sigma \left(\frac{1}{\eps^2} 
     \int_{\Omega}(\abs{\Q^*_\eps} - 1)^2 \, \d x\right)^{1/2} \\ 
  \end{split}
 \end{equation*}
 The constant~$C$ here depends only on~$\beta$. 
 Finally, Lemma~\ref{lemma:delpino} implies
 \begin{equation*} 
  \begin{split}
   \frac{1}{\eps^2} \int_{\Omega} (\abs{\Q^*_\eps} - 1)^2 \, \d x
   \leq \frac{1}{\eps^2} \int_{\Omega} (\abs{\Q^*_\eps}^2 - 1)^2 \, \d x
   \leq C 
  \end{split}
 \end{equation*}
 and hence,
 \begin{equation} \label{coreGinf5}
  \begin{split}
   \frac{1}{\eps^2} \int_{B_\sigma(a^*_j)} 
    f_\eps(\Q^*_\eps, \, \M^*_\eps) \, \d x 
   \geq \frac{1}{4\eps^2}\int_{B_\sigma(a^*_j)}
     (\abs{\Q^*_\eps}^2 - 1)^2\, \d x - C\sigma 
  \end{split}
 \end{equation}
 Combining~\eqref{coreGinf1} with~\eqref{coreGinf5}, the lemma follows.
\end{proof}

We can now complete the proof of Proposition~\ref{prop:Gamma_liminf}.

\begin{proof}[Proof of Proposition~\ref{prop:Gamma_liminf}]
 Let~$\sigma>0$ be small enough that the balls~$B_\sigma(a^*_j)$, 
 $B_\sigma(b^*_k)$ are pairwise disjoint. We define
 \begin{equation*}
  \Omega_\sigma := \Omega\setminus\left(
   \bigcup_{j=1}^{2\abs{d}} \bar{B}_\sigma(a^*_j)
   \cup \bigcup_{k=1}^{K} \bar{B}_\sigma(b^*_k)\right)
 \end{equation*}
 We construct open sets~$G_1$, \ldots, $G_N$ with the following properties:
 \begin{enumerate}[label=(\roman*)]
  \item the sets~$G_i$ are pairwise disjoint;
  \item their closures, $\overline{G}_i$, cover all of~$\Omega_\sigma$;
  \item each~$G_j$ is simply connected;
  \item $\H^1(\S_{\M^*}\cap \partial G_i\cap\Omega_\sigma) = 0$ for any~$j$.
 \end{enumerate}
 For instance, we can partition~$\Omega_\sigma$ by considering
 a grid, consisting of finitely many vertical and horizontal lines.
 Since~$\H^1(\S_{\M^*}) < +\infty$ by Proposition~\ref{prop:compactnessM},
 we have
 \begin{equation} \label{Ginf1}
  \H^1\left(\S_{\M^*} \cap (\{c\}\times\R)\right)
  = \H^1\left(\S_{\M^*} \cap (\R\times\{d\})\right) = 0
 \end{equation}
 for all but countably many values of~$c\in\R$, $d\in\R$.
 We choose numbers
 \[
  c_0 < c_1 < \ldots < c_{N_1}, \qquad
  d_0 < d_1 < \ldots < d_{N_2}
 \]
 that satisfy~\eqref{Ginf1}, in such a way that 
 $\overline{\Omega}\subseteq (c_0, \, c_{N_1})\times (d_0, \, d_{N_2})$.
 For a suitable choice of~$c_h$, $d_\ell$, we can make sure that
 no ball~$B_\sigma(a^*_j)$ or~$B_{\sigma}(b^*_k)$ is entirely contained
 in a rectangle of the form $(c_h, \, c_{h+1})\times (d_\ell, \, d_{\ell+1})$,
 and that any rectangle~$(c_h, \, c_{h+1})\times (d_\ell, \, d_{\ell+1})$
 intersects at most one of the balls. Then, the sets
 \[
  G_{h,\ell} := \left((c_h, \, c_{h+1})\times (d_\ell, \, d_{\ell+1})\right)
  \cap\Omega_\sigma
 \]
 are all simply connected and satisfy the properties~(i)--(iv) above.
 We relabel the~$G_{h,\ell}$'s as~$G_i$. 
 
 We apply Lemma~\ref{lemma:loc_Gamma_liminf} on each~$G_i$,
 then sum over all the indices~$i$. We obtain
 \begin{equation} \label{Ginf2}
  \begin{split}
   \mathscr{F}_\eps(\Q^*_\eps, \, \M^*_\eps; \, \Omega_\sigma)
   &\geq \frac{1}{2} \int_{\Omega_\sigma} \abs{\nabla\Q^*}^2 \, \d x
    + \int_{\Omega_\sigma}(\xi_* - \kappa_*)^2 \, \d x \\
   &\qquad\qquad + c_\beta \, \H^1(\S_{\M^*}\cap\Omega_\sigma)
    + \o_{\eps\to 0}(1) 
  \end{split}
 \end{equation}
 On the other hand, Lemma~\ref{lemma:core_Gamma_liminf} implies
 \begin{equation} \label{Ginf3}
  \begin{split}
   \mathscr{F}_\eps(\Q^*_\eps, \, \M^*_\eps; \, B_\sigma(a^*_j)) 
   - \pi\abs{\log\eps} \geq -\pi\abs{\log\sigma} + \gamma_* - C\sigma
   + \o_{\eps\to 0}(1)
  \end{split}
 \end{equation}
 for any~$j\in\{1, \, \ldots, \, 2\abs{d}\}$. 
 Combining~\eqref{Ginf2} with~\eqref{Ginf3}, we obtain
 \begin{equation} \label{Ginf4}
  \begin{split}
   &\mathscr{F}_\eps(\Q^*_\eps, \, \M^*_\eps) - 2\pi\abs{d}\abs{\log\eps}
   \geq \frac{1}{2} \int_{\Omega_\sigma} \abs{\nabla\Q^*}^2 \, \d x
    - 2\pi\abs{d}\abs{\log\sigma} 
    + 2\abs{d}\gamma_* \\
   &\qquad
    + c_\beta \, \H^1(\S_{\M^*}\cap\Omega_\sigma)
    + \int_{\Omega_\sigma}(\xi_* - \kappa_*)^2 \, \d x  
    + \o_{\eps\to 0}(1) + \o_{\sigma\to 0}(1)
  \end{split}
 \end{equation}
 By Proposition~\ref{prop:canonical}, $\Q^*$ is the canonical harmonic map
 with singularities at~$(a^*_1, \, \ldots, \, a^*_{2\abs{d}})$
 and boundary datum~$\Qb$. Then, we can write
 the right-hand side of~\eqref{Ginf4} in terms of the
 renormalised energy, $\mathbb{W}$, defined in~\eqref{renormalised}.
 First, we observe that
 \begin{equation} \label{Ginf5}
  \begin{split}
   \frac{1}{2} \int_{\bigcup_{k=1}^{K}
    B_\sigma(b^*_k)} \abs{\nabla\Q^*}^2 \, \d x \to 0 
    \qquad \textrm{as } \sigma\to 0
  \end{split}
 \end{equation}
 because~$\Q^*\in W^{1,2}_{\loc}(\Omega\setminus
 \{a^*_1, \, \ldots, \, a^*_{2\abs{d}}\}, \, \Sz)$. 
 Then, from~\eqref{Ginf4}, \eqref{Ginf5} and~\eqref{renormalised}
 we deduce
 \begin{equation} \label{Ginf7}
  \begin{split}
   &\mathscr{F}_\eps(\Q^*_\eps, \, \M^*_\eps) - 2\pi\abs{d}\abs{\log\eps}
   \geq \mathbb{W}(a^*_1, \, \ldots, \, a^*_{2\abs{d}}) + 2\abs{d}\gamma_* \\
   &\qquad
    + c_\beta \, \H^1(\S_{\M^*}\cap\Omega_\sigma)
    + \int_{\Omega_\sigma}(\xi_* - \kappa_*)^2 \, \d x  
    + \o_{\eps\to 0}(1) + \o_{\sigma\to 0}(1)
  \end{split}
 \end{equation}
 Now we pass to the limit in both sides of~\eqref{Ginf7},
 first as~$\eps\to 0$, then as~$\sigma\to 0$.
 The proposition follows.
\end{proof}

\subsubsection{Sharp upper bounds}
\label{ssect:Gamma-limsup}

In this section we will prove an upper bound for
the energy of minimizers, namely the following:
 
\begin{prop}\label{gammalimsup}
 Let~$a_1, \, \ldots, \, a_{2\abs{d}}$ be distinct points in~$\Omega$.
 Then, there exist maps~$\Q_\eps\in W^{1,2}(\Omega, \, \Sz)$, 
 $\M_\eps\in W^{1,2}(\Omega, \, \R^2)$ that satisfy the boundary condition~\eqref{bc} and
 \begin{equation}
  \F_\eps (\Q_\eps, \, \M_\eps)\le 2 \pi\abs{d} \abs{\log \eps}
  +\mathbb{W}_\beta(a_1,\cdots, a_{2\abs{d}})
  + 2 \abs{d} \gamma_* + \o_{\eps\to 0}(1)
 \end{equation}
 where~$\mathbb{W}_\beta$ and~$\gamma_*$ are 
 as in~\eqref{Wbeta}, \eqref{core_energy} respectively.
\end{prop}

The proof of Proposition~\ref{gammalimsup}
is based on a rather explicit construction.
For the component~$\Q_\eps$, we follow classical
arguments from the Ginzburg-Landau literature
(see e.g.~\cite{BBH, AlicandroPonsiglione}), with minor modifications.
For the component~$\M_\eps$, we first construct
a vector field~$\tilde{\M}_\eps\colon\Omega\to\R^2$
of constant norm, such that $\tilde{\M}_\eps(x)$
is an eigenvector of~$\Q_\eps(x)$ at each point~$x\in\Omega$.
As~$\Q_\eps$ has non-orientable singularities at the points~$a_j$,
there is no smooth vector field~$\tilde{\M}_\eps$
with this property. However, we can construct
a~$\BV$-vector field~$\tilde{\M}_\eps$, which
jumps along finitely many line segments
that join the points~$a_j$ along a minimal connection 
(see Appendix~\ref{app:lifting}).
Then, we define~$\M_\eps$ by regularising~$\tilde{\M}_\eps$
in a small neighbourhood of the jump set.
The regularisation procedure is reminiscent of the optimal profile 
problem for the Modica-Mortola functional~\cite{ModicaMortola}.

\begin{proof}[Proof of Proposition~\ref{gammalimsup}]
We follow the argument of \cite{AlicandroPonsiglione}, Theorem~5.3.
Let $\sigma>0$ be such that $B_{\sigma}(a_i)$ are disjoints and contained in $\Omega$ and set $\Omega_{\sigma}:=\Omega\setminus\bigcup_{i=1}^{2\abs{d}} B_{\sigma}(a_i)$. 
First, we minimize the functional
\begin{equation} \label{min1}
 (\Q, \, \mathbf{R}_1, \ldots, \mathbf{R}_{2\abs{d}})
 \mapsto \frac12 \int_{\Omega_{\sigma}}|\nabla \Q|^2 \, \d x
\end{equation}
over all maps~$\Q\in H^1(\Omega_\sigma, \, \NN)$
and all rotation matrices~$\mathbf{R}_i\in\mathrm{SO}(2)$ such that
$\Q = \Qb$ on~$\partial\Omega$ and
\[
 \Q(x)= \sqrt{2}\left(\frac{(\mathbf{R}_i(x-a_i))\otimes(\mathbf{R}_i(x - a_i))}{\sigma^2} - \frac{\I}{2}\right)
\  \text{ for} \,  x\in\partial B_{\sigma}(a_i).
\]
We denote by~$m(\sigma)$ the minimum value
and by~$\tilde{\P}_1$, $\tilde{\mathbf{R}}_i$
the minimisers of this functional.
Next, we minimise the Ginzburg-Landau energy,
on a ball~$B_\sigma$ of radius~$\sigma$ centered at the origin,
\begin{equation} \label{min2}
 GL_\eps(\Q, B_{\sigma}) := \int_{B_\sigma} \Big(\frac12 |\nabla \Q|^2 +\frac1{4\eps^2}(1-|\Q|^2)^2\Big) \, \d x
\end{equation}
among all the maps~$\Q\in H^1(B_\sigma, \, \Sz)$ such that
\[
 \Q(x)= \sqrt{2}\left(\frac{x\otimes x}{\sigma^2} - \frac{\I}{2}\right)
\  \text{ for} \,  x\in\partial B_{\sigma}.
\]
We denote by~$\gamma(\eps, \, \sigma)$ the 
minimum value and by~$\tilde{\P}_2$ the minimiser of this functional.

We define a map~$\tilde{\Q}_\eps\in H^1(\Omega, \, \Sz)$
as
\[
 \tilde{\Q}_\eps(x) :=
 \begin{cases}
  \tilde{\P}_1(x) & \textrm{if } x\in\Omega_\sigma \\
  \tilde{\mathbf{R}}_i\tilde{\P}_2(x - a_i)\tilde{\mathbf{R}}_i^{\mathsf{T}} & \textrm{if } x\in B_\sigma(a_i)
 \end{cases}
\]
This map satisfies $\tilde{\Q}_\eps = \Qb$ on~$\partial\Omega$,
$|\tilde{\Q}_\eps|\leq 1$ in~$\Omega$,
$|\tilde{\Q}_\eps|=1 $ in $\Omega_{\sigma}$.
Moreover, thanks to~\cite[Theorem~I.9 and Section~III.1]{BBH},
we have
\begin{align*}
\int_{\Omega}\Big(\frac12 |\nabla \tilde{\Q}_\eps|^2& +\frac1{4\eps^2}(1-|\tilde{\Q}_\eps|^2)^2\Big) \, \d x\\
&=m(\sigma)+2\abs{d}\gamma(\sigma, \eps)\\
&=m(\sigma)+2\abs{d}\gamma_*-2\abs{d}\pi \log\frac{\eps}{\sigma}+\o_{\frac{\eps}{\sigma}\to 0}(1)\\
&=\mathbb{W}(a_1,\cdots, a_{2\abs{d}})-2\pi \abs{d}\log \sigma +2\abs{d} \gamma_*-2\abs{d}\pi \log\frac{\eps}{\sigma}+\o_{\sigma\to 0}(1)+\o_{\frac{\eps}{\sigma}\to 0}(1)\\
&=\mathbb{W}(a_1,\cdots, a_{2\abs{d}})+2\abs{d} \gamma_*+2\abs{d}\pi \log |{\eps}|+\o_{\sigma\to 0}(1)+\o_{\frac{\eps}{\sigma}\to 0}(1).
\end{align*}
We will choose~$\sigma=\sigma_\eps$ in such a way that 
$$\sigma_\eps\to 0,\quad  \frac{\eps}{\sigma_{\eps}}\to 0.$$
Define $\Q_\eps := (1+\eps \kappa_*) \tilde{\Q}_\eps$ on $\Omega$.
Therefore, $|\Q_\eps|=1+\eps \kappa_*$ on $\Omega_{\sigma_\eps} =
\Omega\setminus \bigcup_{j=1}^{2\abs{d}} B_{\sigma_\eps} (a_j)$.
Moreover, we have
\begin{equation}
\abs{\int_{\Omega} |\nabla  \tilde{\Q}_\eps|^2 -|\nabla\Q_\eps|^2} \lesssim \kappa_* \eps \int_{\Omega} |\nabla  \tilde{\Q}_\eps|^2 \lesssim \eps\abs{\log\eps}
\end{equation}
On the other hand, for the Ginzburg-Landau potential we have
\begin{align*}
\frac1{4\eps^2}&\abs{\int_{ \bigcup_{j=1}^{2\abs{d}} B_{\sigma_\eps}(a_j)}\Big( 1-|   \tilde{\Q}_\eps|^2\Big)^2 -\int_{ \bigcup_{j=1}^{2\abs{d}} B_{\sigma_\eps}(a_j)}\Big( 1-|\Q_\eps|^2 \Big)^2}\\
=\frac1{4\eps^2}&\abs{\int_{ \bigcup_{j=1}^{2\abs{d}} B_{\sigma_\eps}(a_j)}\Big(2(1+\eps \kappa_*)^2 -2\Big)|   \tilde{\Q}_\eps|^2- \Big((1+\eps \kappa_* )^4 -1\Big) |   \tilde{\Q}_\eps|^4}=\O(\eps)
\end{align*}
since by construction $|\tilde{\Q}_\eps|\le 1$.
In conclusion, we have:
\begin{equation} \label{limsup1}
 \begin{split}
  \frac12 \int_{\Omega}|\nabla{\Q}_\eps|^2 \, \d x
  & +\frac1{4\eps^2}\int_{\bigcup_{j=1}^{2\abs{d}} B_{\sigma_\eps}(a_j)}
  (1-|{\Q}_\eps|^2)^2 \, \d x\\
  &=\mathbb{W}(a_1,\cdots, a_{2\abs{d}})+2\abs{d} \gamma_*+2\abs{d}\pi \abs{\log\eps}+\o_{\eps\to 0}(1).
 \end{split}
\end{equation}
We will estimate the contribution of the potential 
on~$\Omega_{\sigma_\eps}$ later on.

We construct the component~$\M_\eps$.
Using the results of Appendix~\ref{app:lifting}, we find a 
minimal connection $L_1, \cdots, L_{\abs{d}}$ for $a_1, \cdots, a_{2\abs{d}}$ with $L_i$ pairwise disjoint (see Lemma~\ref{lemma:mindisj}). 
By reasoning as in Lemma~\ref{lemma:goodlifting}, we define
a lifting $\tilde{\M}_\eps\in\SBV(\Omega_{\sigma_\eps}, \, \R^2)$ 
of $\tilde{\Q}_\eps$ --- that is, a vector field~$\tilde{\M}_\eps\colon\Omega_{\sigma_\eps}\to\R^2$ such that~$|\tilde{\M}_\eps|=(\sqrt{2} \beta +1)^{\frac12}$ and 
\begin{equation} \label{liftM}
 \tilde{\Q}_\eps=\sqrt{2} \left( \frac{ \tilde{\M}_\eps \otimes\tilde{\M}_\eps}{\sqrt{2} \beta +1}-\frac{\I}2\right)
\end{equation}
--- which, in addition, satisfies
$S_{\tilde{\M}_\eps}=(\bigcup_{i=1}^{\abs{d} }L_i)\cap\Omega_{\sigma_\eps}$,
up to negligible sets.
By the same arguments as in the proof of 
Proposition~\ref{prop:energy_upperbd}, we can assume with no loss of generality that~$\tilde{\M}_\eps = \Mb$ on~$\partial\Omega$.
In order to define our competitor~$\M_\eps$, we need
to regularise~$\tilde{M}_\eps$ near its jump set. We will 
do this by considering a  Modica-Mortola optimal profile problem.
Define $u\colon [0,\infty]\to\mathbb{R}$ as a minimiser 
for the following variational problem:
\begin{equation}\label{inte}
\min\left\{ \int_0^{+\infty}\Big(\frac12 u'^2 +\frac12 H^2(u) \Big) \d t\colon
 u\colon [0, \, +\infty)\to\R, \ u(0)=0, \ u(+\infty)=(\sqrt{2}\beta+1)^{\frac12}\right\}
\end{equation}
where $H(u) :=\sqrt{2h(u,0)} =\frac1{\sqrt{2}} |\sqrt{2}\beta+1-u^2|$.
A minimiser for~\eqref{inte} exists, by the direct method of the calculus of variations.
The Euler-Lagrange equation for~\eqref{inte} reads as:
$$-u'u''+\frac12 (H^2)'(u)=0$$
This can be rewritten as
$$-\frac{\d}{\d t}\left(\frac{u'^2}{2}\right)+\frac12 \frac{\d}{\d t}(H^2(u))=0$$
that is
\begin{equation} \label{inte2}
 -u'^2+H^2(u)=\mathrm{constant}=0
\end{equation}
due to the conditions at infinity.
We can compute the integral in \eqref{inte}:
\begin{equation} \label{inte3}
 \int_0^{+\infty}\Big(\frac12 u'^2 +\frac12 H(u) \Big) \d t
 \stackrel{\eqref{inte2}}{=} \int_0^{+\infty} u^\prime H(u) \, \d t
 = \int_0^{(\sqrt{2}\beta+1)^{1/2}} H(u) \, \d u
 = \frac12 c_\beta
\end{equation}
where $c_\beta$ is given by \eqref{c*} (see Lemma~\ref{lemma:cost}).

We define the competitor~$\M_\eps$ in~$\Omega_{\sigma_\eps}$ by a suitable
regularisation of~$\tilde{\M}_\eps$ in a neighbourhood
of each singular line segment~$L_j$. To simplify
the notation, we focus on~$L_1$ and we assume without loss of generality, up to rotations and translations, that $L_1=[0, \, a]\times\{0\}$ for some~$a > 0$.
We assume that~$\eps$ is small enough, so that $\sigma_\eps\ll\frac{a}4$.
Let $A_\eps :=[0,a]\times[-\sigma_\eps,\sigma_\eps]\setminus \Big(B_{\sigma_\eps}(0,0)\bigcup B_{\sigma_\eps}(a,0)\Big)$. 
We define 
\begin{equation}
\M_\eps(x) :=
  \begin{cases}
\dfrac{u\left(\frac{|x_2|}{\eps}\right)}{u\left(\frac{\sigma_\eps}{\eps}\right)} \tilde{\M}_\eps(x), &\textrm{in }  A_\eps^1 
\textrm{ (and similarly in each } A^j_\eps\textrm{)} \\[12pt]
\tilde{\M}_\eps(x), &\textrm{on } \Omega_{\sigma_\eps} \setminus \bigcup_{j=1}^{2\abs{d}} A_\eps^j .
 \end{cases}
\end{equation}
For~$\eps$ small enough, we have
$\M_\eps = \tilde{\M}_\eps = \Mb$ on~$\partial\Omega$.
In~$\Omega_{\sigma_\eps} \setminus \bigcup_{j=1}^{2\abs{d}} A_\eps^j$,
we have $|\nabla\M_\eps|^2 \lesssim |\nabla\tilde{\M}_\eps|^2$.
The latter can be estimated by differentiating both sides
of~\eqref{liftM}, by the BV-chain rule; this gives
$\|\nabla\tilde{\M}_\eps\|^2_{L^2(\Omega_{\sigma_\eps})}
\lesssim \|\nabla\tilde{\Q}_\eps\|^2_{L^2(\Omega_{\sigma_\eps})}\lesssim \abs{\log\eps}$.
Let
\[
 \eta_\eps := \frac{\sqrt{2}\beta + 1}{u\left(\frac{\sigma_\eps}{\eps}\right)^2}
\]
We observe that~$\eta_\eps\to 1$ as~$\eps\to 0$, due to the 
condition at infinity in~\eqref{inte}.
We have in $A_\eps^1$:
\begin{equation*}
\frac{\eps}{2}|\nabla\M_\eps|^2
 = \frac{\eta_\eps}{2}\abs{u'\left(\frac{|x_2|}{\eps}\right)}^2+\O( \eps|\nabla\tilde{\M}_\eps|^2)
\end{equation*}
and therefore:
\begin{align*}
\frac{\eps}{2}\int_{A^1_\eps}|\nabla\M_\eps|^2 \d x
&\le \O(\eps|\log\eps|) + \frac{\eta_\eps}{2}\int_{A_\eps^1}\abs{u'\left(\frac{|x_2|}{\eps}\right)}^2 \d x
+\O( \eps\|\nabla\tilde{\M}_\eps\|^2_{L^2(\Omega_{\sigma_\eps})})\\
&= \O(\eps|\log\eps|) +  \frac{\eta_\eps}{2} \H^1(L_1)\cdot\int_{-\sigma_\eps}^{\sigma_\eps}\abs{u'\left(\frac{|x_2|}{\eps}\right)}^2 \d x_2\\
&= \O(\eps|\log\eps|)+  \eta_\eps \, \H^1(L_1)\cdot\int_{0}^{\frac{\sigma_\eps}{\eps}}\abs{u'(t)}^2 \, \d t.
\end{align*}
By repeating this argument on each~$A^j_\eps$, we deduce
\begin{equation} \label{limsup2}
 \frac{\eps}{2}\int_{\Omega_{\sigma_\eps}}|\nabla\M_\eps|^2 \d x
 \leq \O(\eps|\log\eps|)+ \eta_\eps\,\mathbb{L}(a_1, \, \ldots, \, a_{2\abs{d}})
 \cdot\int_{0}^{\frac{\sigma_\eps}{\eps}}\abs{u'(t)}^2 \, \d t.
\end{equation}

Next, we estimate the potential term. 
On $\Omega_{\sigma_\eps} \setminus \bigcup_{j=1}^{\abs{d}} A_\eps^j$,
we have $|\Q_\eps|=1+\kappa_*\eps$ and~$|\M_\eps|=(\sqrt{2}\beta+1)^{\frac12}$.
The identity~\eqref{liftM} can be written as
\begin{equation*}
\frac{\Q_\eps}{1+\kappa_*\eps}=\sqrt{2}\left(\frac{\M_\eps\otimes \M_\eps}{\sqrt{2}\beta+1}-\frac{\I}{2}\right)
\end{equation*}
which implies
\begin{align*}
\Q_\eps\M_\eps\cdot\M_\eps&=\sqrt{2} (1+\kappa_*\eps)\left(\frac{|\M_\eps|^4}{\sqrt{2}\beta+1}-\frac12|\M_\eps|^2\right)
=\frac{\sqrt{2}}{2} (1+\kappa_*\eps)(\sqrt{2}\beta+1).
\end{align*}
In conclusion, at each point of~$\Omega_{\sigma_\eps} \setminus \bigcup_{j=1}^{\abs{d}} A_\eps^j$ we have
\begin{equation} \label{limsup3}
 \begin{split}
  f_\eps(\Q_\eps,\M_\eps) 
  &=\frac14(2\kappa_*\eps+\kappa_*^2\eps^2)^2+\frac{\eps\beta^2}2-\frac{\beta\eps}{\sqrt{2}}(1+\kappa_*\eps)(\sqrt{2}\beta+1)+\kappa_\eps \\
  &=\o_{\eps\to 0}(\eps^2)
 \end{split}
\end{equation}
by taking Lemma~\ref{lemma:feps} into account.
Therefore, the total contribution from the potential 
on~$\Omega_{\sigma_\eps} \setminus \bigcup_{j=1}^{\abs{d}} A_\eps^j$
is negligible.
Let us compute the potential on $A_\eps^j$.
Considering for simplicity the case~$j=1$,
again we have $|\Q_\eps|=1+\kappa_*\eps$, but
\[
 |\M_\eps(x)|=\eta_\eps^{1/2} \, u\left(\frac{|x_2|}{\eps}\right)
\]
Then, \eqref{liftM} can be written as
\begin{equation*}
\Q_\eps=\sqrt{2}(1+\kappa_*\eps)\left(\frac{\M_\eps\otimes \M_\eps}{|\M_\eps|^2}-\frac{\I}{2}\right)
\end{equation*}
which implies
\begin{equation*}
\Q_\eps\M_\eps\cdot \M_\eps=\frac{\sqrt{2}}2 (1+\kappa_*\eps)\, \eta_\eps \, u^2\left(\frac{|x_2|}{\eps}\right)
\end{equation*}
At a generic point~$x\in A_\eps^1$, we have (writing~$v_\eps := \eta_\eps^{1/2} u(|x_2|/\eps)$
for simplicity)
\begin{align*}
f_\eps(\Q_\eps,\M_\eps)&=\frac{\kappa_*\eps^2}{4}  (2+\kappa_*\eps)^2+\frac{\eps}{4} (1-v_\eps^2)-\beta\eps \frac{\sqrt{2}}2(1+\kappa_*\eps)v_\eps^2+\frac12(\beta^2+\sqrt{2}\beta) \eps+\kappa_*^2\eps^2+\o(\eps^2)\\
&=2\kappa_*^2\eps^2+\frac{\eps}4(1-v_\eps^2)^2-\frac{\beta\eps}{\sqrt{2}}v_\eps^2-\frac{\kappa_*\beta\eps^2}{\sqrt{2}}v_\eps^2+\frac12(\beta^2+\sqrt{2}\beta)\eps+\o(\eps^2)\\
&=\O(\eps^2)+\eps\left(h(v_\eps,0)-\frac{\beta^2+\sqrt{2}\beta}{2}\right)+\frac12(\beta^2+\sqrt{2}\beta)\eps\\
&=\O(\eps^2)+\eps h(v_\eps,0)\\
&=\O(\eps^2)+\frac{\eps}2 H^2(v_\eps).
\end{align*}
By repeating this argument on each~$A^j_\eps$,
and taking the integral over~$A^j_\eps$, we obtain
\begin{equation} \label{limsup4}
 \begin{split}
  \frac 1{\eps^2}\int_{\bigcup_{j=1}^{\abs{d}} A_\eps^j} f_\eps(\Q_\eps,\M_\eps) \, \d x&=\int_{\bigcup_{j=1}^{\abs{d}} A_\eps^j}\O(1)+\frac{1}{2\eps} H^2\left(\eta_\eps^{1/2} \, u\left(\frac{|x_2|}{\eps}\right)\right) \d x\\
  &=\O(\sigma_\eps)+\frac{1}{2\eps}\sum_{j=1}^{\abs{d}}\int_{A_\eps^j}H^2\left(\eta_\eps^{1/2} \, u\left(\frac{|x_2|}{\eps}\right)\right) \d x\\
  &=\o_{\eps\to 0}(1) + \mathbb{L}(a_1,\cdots,a_{2{\abs{d}}})\cdot \int_{0}^{\frac{\sigma_\eps}{\eps}}H^2(\eta_\eps^{1/2} \, u(t)) \, \d t.
 \end{split}
\end{equation}
By combining~\eqref{limsup1}, \eqref{limsup2}, \eqref{limsup3} and~\eqref{limsup4}, keeping in mind that~$\eta_\eps\to 1$,
$\sigma_\eps/\eps \to +\infty$ as~$\eps\to 0$, 
and applying Lebesgue's dominated convergence theorem,
we obtain
\begin{equation} \label{limsup5}
 \begin{split}
  \F_\eps(\Q_\eps,\M_\eps;\Omega\setminus \bigcup_{j=1}^{2\abs{d}}B_{\sigma_\eps}(a_j))&\le \o_{\eps\to 0}(1)+\mathbb{L}(a_1,\cdots,a_{2d})\int_{0}^{+\infty}\left(u'^2(t) +  H^2(u(t))\right) \d t\\
  &\qquad +2\pi \abs{d}\abs{\log\eps}+\mathbb{W}(a_1,\cdots, a_{2\abs{d}})+2\abs{d}\gamma_*\\
  &\stackrel{\eqref{inte3}}{=}
  2\pi \abs{d}\abs{\log\eps}+\mathbb{W}_\beta(a_1,\cdots,a_{2\abs{d}})+2\abs{d}\gamma_*+\o_{\eps\to 0}(1).
 \end{split}
\end{equation}

It only remains to define $\M_\eps$ in each ball~$B_{\sigma_\eps}(a_j)$. 
For each~$j$, there exists~$\rho = \rho(j)\in (\sigma_\eps, 2\sigma_\eps)$ 
such that 
\begin{equation*}
\int_{\partial B_{\rho}(a_j)}|\nabla\M_\eps|^2 \, \d\H^1
\le \frac 1{\sigma_\eps}\int_{B_{2\sigma_\eps}(a_j)\setminus B_{\sigma_\eps}(a_j)}|\nabla\M_\eps|^2 \, \d x
= \O\left(\frac{\abs{\log\eps}}{\sigma_\eps}\right)+\O(\frac 1{\eps})
\end{equation*}
Define $\M_\eps$ on $B_{\rho}(a_j)$ as
\begin{equation} \label{limsup6}
\M_\eps(x) := \frac{|x-a_j|}{\rho} \M_\eps\left(\frac{\rho(x-a_j)}{|x-a_j|}\right)
\end{equation}
The vector field~$\M_\eps$ was already defined in~$B_\rho(a_j)\setminus B_{\sigma_\eps}(a_j)$, but we disregard its previous values and
re-define it according to~\eqref{limsup6}. We have
\begin{equation} \label{limsup7}
 \begin{split}
  \eps\int_{B_{\rho}(a_j)} |\nabla\M_\eps|^2 \, \d x
   &\le \sigma_\eps\int_{\partial B_{\rho}(a_j)}\eps|\nabla\M_\eps|^2\,\d\H^1
   + \eps\int_{B_{\rho}(a_j)} \O\left(\frac{1}{\rho^2}\right) \, \d x \\
   &\leq \O(\eps|\log\eps|)+\O(\sigma_\eps)+\O(\eps)\to 0
 \end{split}
\end{equation}
and 
\begin{equation} \label{limsup8}
\frac 1{\eps^2}\int_{B_{\rho}(a_j)} \left(f_\eps(\Q_\eps,\M_\eps)-\frac14(1-|\Q_\eps|^2)^2\right) \d x=\O(\frac{\sigma_\eps^2}{\eps}).
\end{equation}
If we choose $\eps\ll\sigma_\eps\ll\eps^{\frac 12}$, then 
the total contribution of $\M_\eps$ to the energy 
on each ball~$B_\rho(a_j)$ tends to zero as~$\eps\to 0$.
\end{proof}

\begin{remark} \label{rk:Neumann}
 { The proof of Proposition~\ref{gammalimsup} carries over, with no essential modifications, to the case we impose Dirichlet boundary conditions for the $\Q$-component and Neumann boundary conditions for the~$\M$-component, as described in Remark~\ref{rk:bc}. Indeed, while the structure of the (orientable) boundary datum for~$\Q$ is important to the analysis, the boundary condition for~$\M$ does not play a crucial role; the coupling between~$\Q$ and~$\M$ is determined by the potential~$f_\eps$ and not the boundary conditions.  }
\end{remark}

We can now complete the proof of our main result, Theorem~\ref{th:main}.

\begin{proof}[Conclusion of the proof of Theorem~\ref{th:main},
 proof of Proposition~\ref{prop:min_energy}]
 From Proposition~\ref{prop:Gamma_liminf} 
 and Proposition~\ref{gammalimsup}, we deduce
 \begin{equation} \label{Gamma1}
  \begin{split}
   &\mathbb{W}(a^*_1, \, \ldots, \, a^*_{2\abs{d}})
    + c_\beta\,\H^1(\S_{\M^*})
    + \int_{\Omega}(\xi_* - \kappa_*)^2 \, \d x
    + 2\abs{d}\gamma_* \\
   &\hspace{2cm} \leq \liminf_{\eps\to 0}
    \big(\mathscr{F}_\eps(\Q^*_\eps, \, \M^*_\eps) - 
    2\pi\abs{d}\abs{\log\eps} \big)  \\
   &\hspace{2cm} \leq \limsup_{\eps\to 0}
    \big(\mathscr{F}_\eps(\Q^*_\eps, \, \M^*_\eps) - 
    2\pi\abs{d}\abs{\log\eps} \big) \\
   &\hspace{2cm} \leq \mathbb{W}(a_1, \, \ldots, \, a_{2\abs{d}})
   + c_\beta \, \mathbb{L}(a_1, \, \ldots, \, a_{2\abs{d}})
    + 2\abs{d}\gamma_*
  \end{split}
 \end{equation}
 for any $(2\abs{d})$-uple of distinct points~$a_1$, \ldots, $a_{2\abs{d}}$
 in~$\Omega$. In particular, choosing~$a_j = a_j^*$, we obtain
 \begin{equation} \label{Gamma2}
  \H^1(\S_{\M^*}) =  \mathbb{L}(a^*_1, \, \ldots, \, a^*_{2\abs{d}}),
  \qquad \xi_* = \kappa_*
 \end{equation}
 and Proposition~\eqref{prop:min_energy} follows.
 Moreover, Proposition~\ref{prop:lowerbound_SM} and~\eqref{Gamma2}
 imply that the jump set~$\S_{\M^*}$ coincides (up to negligible sets) with
 $\cup_{j=1}^{\abs{d}}L_j$, where~$(L_1, \, \ldots, \, L_{\abs{d}})$
 is a minimal connection for~$(a_1, \, \ldots, \, a_{2\abs{d}})$.
 Finally, from~\eqref{Gamma1} and~\eqref{Gamma2} we deduce
 \begin{equation} \label{Gamma3}
  \mathbb{W}_\beta(a^*_1, \, \ldots, \, a^*_{2\abs{d}})
   \leq \mathbb{W}_\beta(a_1, \, \ldots, \, a_{2\abs{d}})
 \end{equation}
 for any $(2\abs{d})$-uple of distinct points~$a_1$, \ldots, $a_{2\abs{d}}$
 in~$\Omega$ --- that is, $(a^*_1, \, \ldots, \, a^*_{2\abs{d}})$
 minimises~$\mathbb{W}_\beta$.
\end{proof}

\section{Numerics}
\label{sect:numerics}

In this section, we numerically compute some stable critical points of the ferronematic free energy, on square domains with topologically non-trivial Dirichlet boundary conditions for $\Q$ and $\M$. {These numerical results do not directly support our main results on global energy minimizers of \eqref{energy} in the $\epsilon \to 0$ limit, since the numerically computed critical points need not be global energy minimizers, and we expect multiple local and global energy minimizers of \eqref{energy} for $\epsilon >0$. }

Instead of solving the Euler-Lagrange equations directly, we solve a $L^2$-gradient flow associated with the  effective re-scaled free energy for ferronematics ~\eqref{energy}, given by 
\begin{equation}
\frac{\dd}{\dd t} \F_{\eps} (\Q, \M) = - \int_{\Omega} (\eta_1 |\pp_t \Q|^2 + \eta_2 |\pp_t \M|^2) \dd {\bf x}.
\end{equation}
Here $\eta_1 > 0$ and $\eta_2 > 0$ are arbitrary friction coefficients. Due to limited physical data, we do not comment on physically relevant values of $\eps$, $\beta$ and the friction coefficients.
The system of $L^2$-gradient flow equations for $Q_{11}$, $Q_{12}$ and the components, $M_1$, $M_2$ of the magnetisation vector, can be written as
\begin{equation}\label{GD}
  \begin{cases}
    & 2 \eta_1 \, \pp_t Q_{11} =  2 \Delta Q_{11} - \frac{1}{\eps^2} ( 4 Q_{11} (Q_{11}^2 + Q_{12}^2 - 1/2) - \beta \eps (M_1^2 - M_2^2))  \\ 
    & 2 \eta_1 \, \pp_t Q_{12} =  2 \Delta Q_{12} - \frac{1}{\eps^2} ( 4 Q_{12} (Q_{11}^2 + Q_{12}^2 - 1/2) - 2 \beta \eps M_1 M_2 )  \\ 
    & \eta_2 \, \pp_t M_1 = \eps \Delta M_1 - \frac{1}{\eps^2} ( \eps (M_1^2 + M_2^2  - 1) M_1 - \beta \eps ( 2 Q_{11} M_1 + 2 Q_{12} M_2 ) ) \\
    & \eta_2 \, \pp_t M_2 = - \eps \Delta M_2 -  \frac{1}{\eps^2} ( \eps (M_1^2 + M_2^2  - 1) M_2 - \beta \eps ( - 2 Q_{11} M_2 + 2 Q_{12} M_1 ) ),  \\
  \end{cases}
\end{equation}
{ The stationary time-independent or equilibrium solutions of the $L^2$-gradient flow satisfy the original Euler-Lagrange equations of \eqref{energy}.
For non-convex free energies as in \eqref{energy}, there are multiple critical points, with many of them being unstable saddle points \cite{yin2020construction}. One can efficiently compute stable critical points of such free energies by considering the $L^2$-gradient flow associated with the non-convex free energies and these gradient flows converge to a stable critical point, for a given initial condition, thus avoiding the unstable saddle points. 
From a numerical standpoint, the $L^2$-gradient flow can be more straightforward to solve than the nonlinear coupled Euler-Lagrange equations, primarily due to the inclusion of time relaxation in the $L^2$-gradient flow. }

In the following simulations, we take $\eta_1 = 1$ and $\eta_2 = \eps$ and do not offer rigorous justifications for these choices, except as numerical experiments to qualitatively support out theoretical results.
We impose the continuous degree $+k$ boundary condition
\begin{equation}
 \M_b = (\sqrt{2}\beta + 1)^{1/2} (\cos k\theta, \sin k\theta), \quad {\bf Q}_b = \sqrt{2} \begin{pmatrix}
   \frac{1}{2} \cos 2k \theta  & \frac{1}{2} \sin 2k \theta \\[5pt]
   \frac{1}{2} \sin 2k \theta  & \frac{1}{2} \cos 2k \theta \\
 \end{pmatrix},
\end{equation}
where
\begin{equation}
  \theta(x, y) = {\rm atan2} \left({y - 0.5}, \, {x - 0.5} \right) - \pi / 2, \quad (x, y) \in \pp \Omega.
\end{equation}
and ${\rm atan2}(y, x)$ is the 2-argument arctangent that computes the principal value of the argument function applied to the complex number $x + i y$. So $ -\pi \leq {\rm atan2}(y, x) \leq \pi$. For example, if $x>0$, then ${\rm atan2}(y, x) = \arctan\left(\frac{y}{x} \right)$. The initial condition is prescribed to be
 \begin{equation}\label{Initial}
  \M_0 = (\sqrt{2}\beta + 1)^{1/2} (\cos k \theta, \sin k \theta), \quad {\bf Q}_0 = \sqrt{2} \begin{pmatrix}
    \frac{1}{2} \cos 2 k \theta  & \frac{1}{2} \sin 2 k \theta \\[5pt]
    \frac{1}{2} \sin 2 k \theta  & \frac{1}{2} \cos 2 k \theta, \\
  \end{pmatrix}
 \end{equation}
 where
 \begin{equation}
   \theta(x, y) = {\rm atan2} \left({y - 0.5}, \, {x - 0.5} \right) - \pi / 2, \quad (x, y) \in (0, 1)^2.
\end{equation}

{
We solve the $L^2$-gradient flow equation using standard central finite difference methods \cite{iserles2009first}. For the temporal discretization, we employ a second-order Crank-Nicolson method \cite{iserles2009first}. The grid size and temporal step size are denoted by $h$ and $\tau$, respectively. In all our computations, we set $h = 1/50$ and $\tau = 1/1000$. }

\begin{figure}[!h]
  \centering
  \includegraphics[width =  \linewidth]{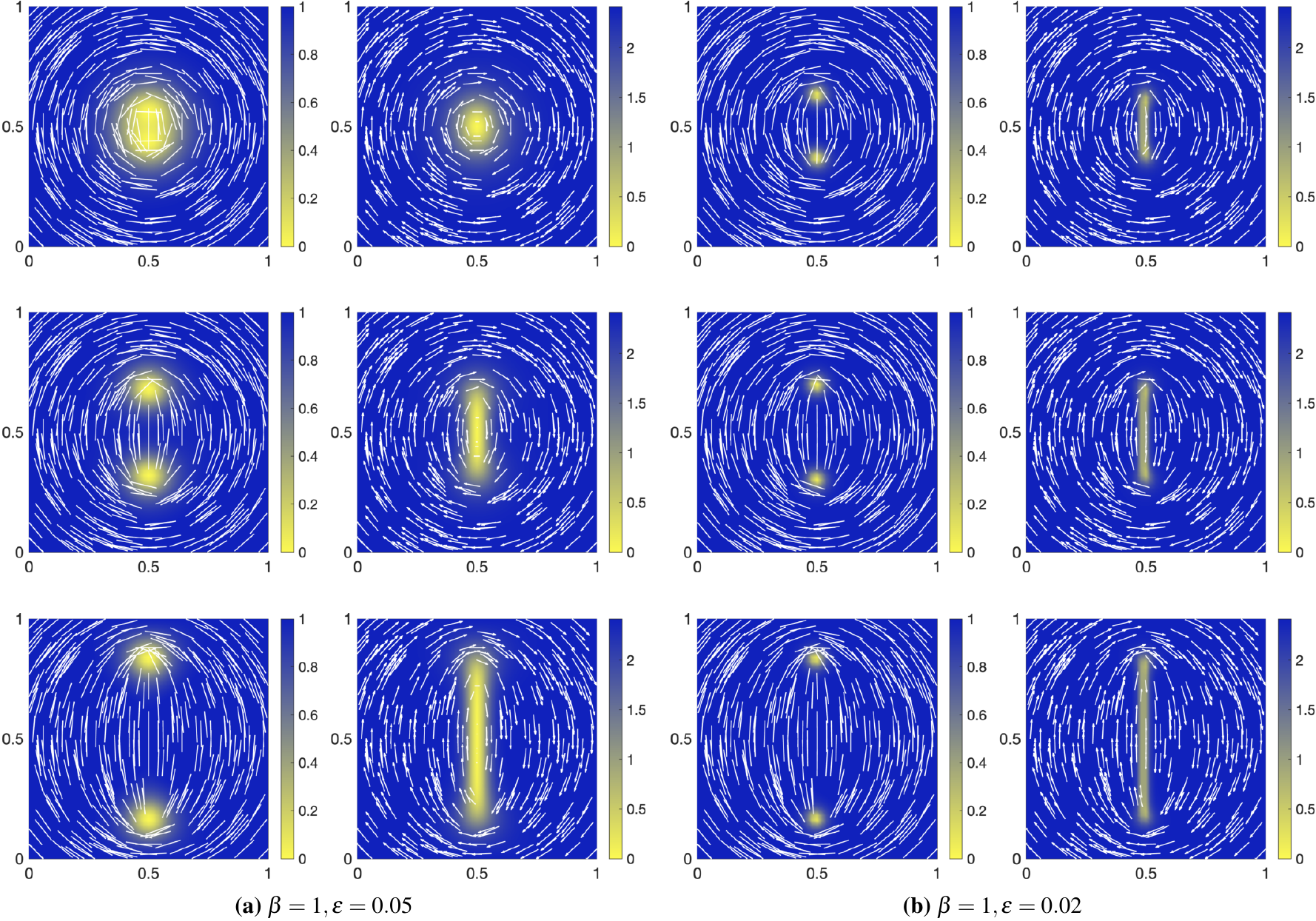}
  \caption{Numerical results for the gradient flows (\ref{GD}) with  (a) $\beta =1, \eps = 0.05$ at $t = 0.02$, $0.05$ and $1$ and (b) $\beta =1, \eps = 0.02$ at $t = 0.02$, $0.05$ and $1$ (Continuous degree +1 boundary condition, $h=1/50$, $\tau = 1/1000$). In each sub-figure, the nematic configuration is shown in the left panel, where the white bars represent nematic field ${\bf n}$ {(the eigenvector of $\Q$ associated with the largest eigenvalue)} and the color represents $\tr \Q^2  =  2(Q_{11}^2 + Q_{12}^2)$; the $\M$-profile is shown in the right panel,
   where the white bars represent magnetic field ${\bf M}$ and the color bar represents $|{\bf M}|^2 = M_1^2 + M_2^2$. }
   \label{fig:Deg1_GD}
\end{figure}



\begin{figure}[!h]
  \centering
  \includegraphics[width =  \linewidth]{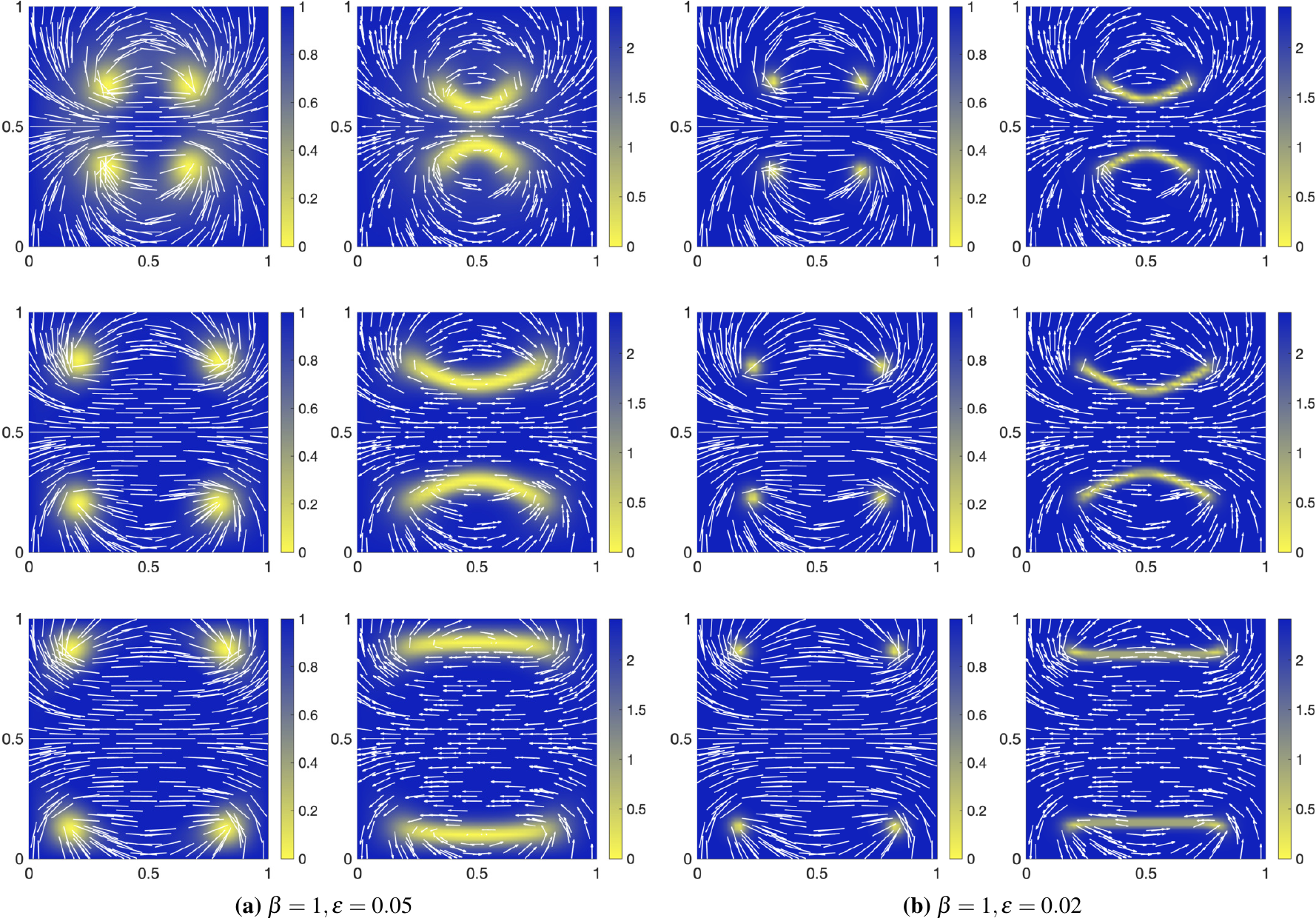}
  \caption{Numerical results for the gradient flows (\ref{GD}) with  (a) $\beta =1, \eps = 0.05$ at $t = 0.02$, $0.05$ and $1$ and (b) $\beta =1, \eps = 0.02$ at $t = 0.02$, $0.05$ and $1$ (Continuous degree +2 boundary condition, $h=1/50$, $\tau = 1/1000$). In each sub-figure, the nematic configuration is shown in the left panel, where the white bars represent nematic field ${\bf n}$ {(the eigenvector of $\Q$ associated with the largest eigenvalue)} and the color represents $\tr \Q^2  =  2(Q_{11}^2 + Q_{12}^2)$; the magnetic configuration is shown in the right panel,
   where the white bars represent magnetic field ${\bf M}$ and the color represents $|{\bf M}|^2 = M_1^2 + M_2^2$. }
   \label{fig:Deg2_GD}
\end{figure}

In Figure~\ref{fig:Deg1_GD}, we plot the {dynamical evolution of the solutions of the gradient flow equations, for $k=1$ boundary conditions, with the initial condition (\ref{Initial}). The time-dependent solutions converge for $t\geq 1$, and we treat the numerical solution at $t=1$ to the converged equilibrium state. }
We cannot conclusively argue that the converged solution is an energy minimizer but it is locally stable, the converged $\Q$-profile has two non-orientable defects and the corresponding $\M$-profile has a jump set composed of a straight line connecting the nematic defect pair, consistent with our theoretical results on global energy minimizers. We consider two different values of $\eps$ and it is clear that the $\Q$-defects and the jump set in $\M$ become more localised as $\eps$ becomes smaller, as expected from the theoretical- results. We have also investigated the effects of $\beta$ on the converged solutions --- the defects become closer as $\beta$ increases.  This is expected since the cost of the minimal connection between the nematic defects increases as $\beta$ increases, and hence the shorter connections require the defects to be closer to each other (at least in a pairwise sense).

In Figure~\ref{fig:Deg2_GD}, we plot the {dynamical evolution of the solutions of the gradient flow equations, for $k=2$ boundary conditions, with the initial condition (\ref{Initial}), and we treat the numerical solution at $t=1$ to be the converged equilibrium state. } Again, the converged solution is locally stable, the $\Q$-profile has four non-orientable defects,the $\M$-profile has two distinct jump sets connecting two pairs of non-orientable nematic defects, and the jump sets are indeed approximately straight lines. Smaller values of $\epsilon$ correspond to the sharp interface limit which induces more localised defects for $\Q$, straighter line defects for $\M$ and larger values of $\beta$ push the defects closer together, all in qualitative agreement with our theoretical results.

{Theorem~\ref{th:main} is restricted to global minimizers of \eqref{energy} in the $\epsilon \to 0$ limit, but the numerical illustrations in Figures~\ref{fig:Deg1_GD} and \ref{fig:Deg2_GD} suggest that Theorem~\ref{th:main} may also partially apply to local energy minimizers of \eqref{energy}. In other words, locally energy minimizing pairs, $(\Q_\eps, \M_\eps)$, may also converge to a pair $(\Q^*, \M^*)$, for which $\Q^*$ is a canonical harmonic map with non-orientable point defects and $\M^*$ has a jump set connecting the non-orientable point defects of $\Q^*$, with the location of the defects being prescribed by the critical point(s) of the normalization energy in Theorem~\ref{th:main}. The numerical illustrations in Figures~\ref{fig:Deg1_GD} and \ref{fig:Deg2_GD} cannot be directly related to Theorem~\ref{th:main}, since we have only considered two small and non-zero values of $\epsilon$ and for a fixed $\beta > 0$, there maybe multiple local and global energy minimizers with different jump sets in $\M$ i.e. different choices of the minimal connection of equal length, or different connections of different lengths between the nematic defect pairs. For example, it is conceivable that a locally stable $\M$-profile also connects the nematic defects by means of straight lines, but this connection is not minimal. There may also be non energy-minimising critical points with orientable point defects in $\M$ tailored by the non-orientable nematic defects. Similarly, there may be non energy-minimising critical points with non-orientable and orientable nematic defects, whose locations are not minimisers but critical points of the modified renormalised energy in Theorem~\ref{th:main}. We defer these interesting questions to future work.}

\section{Conclusions}
\label{sect:conclusions}
We study a simplified model for ferronematics in two-dimensional domains, with Dirichlet boundary conditions, building on previous work in \cite{bisht2020}. The model is only valid for dilute ferronematic suspensions and we do not expect quantitative agreement with experiments. Further, the experimentally relevant choices for the boundary conditions for $\M$ are not well established and our methods can be adapted to other choices of boundary conditions e.g. Neumann conditions for the magnetisation vector. Similarly, it is not clear if topologically non-trivial Dirichlet conditions can be imposed on the nematic directors, for physically relevant experimental scenarios. Having said that, our model problem is a fascinating mathematical problem because of the tremendous complexity of ferronematic solution landscapes, the multiplicity of the energy minimizers and non energy-minimizing critical points, and the multitude of admissible coupled defect profiles for the nematic and magnetic profiles. There are several forward research directions, some of which could facilitate experimental observations of the theoretically predicted morphologies in this manuscript. For example, one could study the experimentally relevant generalisation of our model problem with Dirichlet conditions for $\Q$ and Neumann conditions for $\M$, or study different asymptotic limits of the ferronematic free energy in \eqref{main}, a prime candidate being the $\eps \to 0$ limit for fixed $\xi$ and $c_0$ (independent of $\eps$). This limit, although relevant for dilute suspensions, would significantly change the vacuum manifold~$\NN$ in the $\eps \to 0$ limit. In fact, we expect to observe stable point defects in the energy-minimizing $\M$-profiles for this limit, where $\xi$ and $c_0$ are independent of $\eps$, as $\eps \to 0$. Further, there is the interesting question of how this ferronematic model can be generalised to non-dilute suspensions or to propose a catalogue of magneto-nematic coupling energies for different kinds of MNP-MNP interactions and MNP-NLC interactions. The physics of ferronematics is complex, and it is challenging to translate the physics to tractable mathematical problems with multiple order parameters, and we hope that our work is solid progress in this direction with bright interdisciplinary prospects.

\medskip
\noindent
\textbf{Taxonomy:} GC, BS and AM conceived the project based on a model developed by AM and her ex-collaborators. GC and BS led the analysis, followed by AM. YW performed the numerical simulations, as advised by AM and GC. All authors contributed to the scientific writing.

\medskip
\noindent
\textbf{Acknowledgements:} GC, BS and AM gratefully acknowledge support from the CIRM-FBK (Trento) Research in Pairs grant awarded in 2019, when this collaboration was initiated. GC, BS and AM gratefully acknowledge support from an ICMS Research in Groups grant awarded in 2020, which supported the completion of this project and submission of this manuscript. AM gratefully acknowledges the hospitality provided by the University of Verona in December 2019, GC gratefully acknowledges the hospitality provided by the University Federico~II (Naples) 
under the PRIN project~2017TEXA3H, and AM, BS and GC gratefully acknowledge support from the Erwin Schrodinger Institute in Vienna in December 2019, all of which facilitated this collaboration. AM acknowledges support from the Leverhulme Trust and the University of Strathclyde New Professor's Fund. 
We thank the referee for their careful reading of the manuscript and comments.

\medskip
\noindent
\textbf{Data Availability Statement:}
Data sharing not applicable to this article as no datasets were generated or analysed during the current study.

\medskip
\noindent
\textbf{Conflict of interest statement.}
The authors have no competing interests to declare that are relevant to the content of this article.

\begin{appendix}

\section{Lifting of a map with non-orientable singularities}
\label{app:lifting}

The aim of this section is to prove Proposition~\ref{prop:lowerbound_SM}.
We reformulate the problem in a slightly more general setting.

Let~$a\in\R^2$, and let~$\Q\in W^{1,2}_{\loc}(\R^2\setminus\{a\}, \, \NN)$.
By Fubini theorem and Sobolev embedding, the restriction
of~$\Q$ on the circle~$\partial B_\rho(a)$ is well-defined
and continuous for a.e.~$\rho>0$. Therefore, it makes sense to define
the topological degree of~$\Q$ on~$\partial B_\rho(a)$ as an half-integer,
$\deg(\Q, \, a)\in\frac{1}{2}\Z$. As the notation suggests,
the degree is independent of the choice of~$\rho$:
for a.e.~$0 < \rho_1 < \rho_2$, the degrees of~$\Q$ on~$\partial B_{\rho_1}(a)$
and~$\partial B_{\rho_2}(a)$ are the same.
If~$\Q$ is smooth, this is a consequence of 
the homotopy lifting property; for more general~$\Q\in  W^{1,2}_{\loc}(\R^2\setminus\{a\}, \, \NN)$, this follows from an approximation argument
(based on~\cite[Proposition p.~267]{SchoenUhlenbeck2}).
We will say that~$a$ is a non-orientable singularity of~$\Q$ 
if~$\deg(\Q, \, a)\in \frac{1}{2}\Z\setminus\Z$.

Given an open set~$\Omega\subseteq\R^2$,
a map~$\Q\colon\Omega\to\NN$ and a unit vector field~$\M\colon\Omega\to\SS^1$,
we say that~$\M$ is a lifting for~$\Q$ if
\begin{equation} \label{lifting}
 \Q(x) =\sqrt{2}\left(\M(x)\otimes\M(x) - \frac{\I}{2}\right)
 \qquad \textrm{for a.e. } x\in\Omega.
\end{equation}
Any map~$\Q\in\BV(\Omega, \, \NN)$ 
admits a lifting~$\M\in\BV(\Omega, \, \SS^1)$
(see e.g.~\cite{IgnatLamy}).
The vector field~$\M^*$ given by Theorem~\ref{th:main}
is not a lifting of~$\Q^*$, according to the definition above,
because~$\abs{\M^*}\neq 1$. However, $\abs{\M^*}$ is still 
a positive constant (see Proposition~\ref{prop:compactnessM}),
so we can construct a lifting of unit-norm simply by rescaling.

We focus on properties of the lifting for~$\Q$-tensors of a particular form,
namely, we assume that~$\Q$ has an even number of
non orientable singularities at distinct points~$a_1$, \ldots, $a_{2d}$.
We recall that a \emph{connection} for $\{a_1, \, \ldots, \, a_{2d}\}$ 
as a finite collection of straight line segments
$\{L_1, \, \ldots, \, L_d\}$, with endpoints
in~$\{a_1, \, \ldots, \, a_{2d}\}$,
such that each~$a_i$ is an endpoint of one
of the segments~$L_j$. We recall that
\begin{equation} \label{minconn}
 \mathbb{L}(a_1, \, \ldots, \, a_{2d}) := \min\left\{
 \sum_{i = 1}^d \H^1(L_i)\colon 
 \{L_1, \, \ldots, \, L_d\} \textrm{ is a connection for }
 \{a_1, \, \ldots, \, a_{2d}\} \right\} \!.
\end{equation}
A minimal connection for~$\{a_1, \, \ldots, \, a_{2d}\}$
is a connection that
attains the minimum in the right-hand side
of~\eqref{minconn}.
Given two sets~$A$, $B$, we denote their symmetric
difference as~$A\Delta B := (A\setminus B)\cup (B\setminus A)$.

\begin{prop} \label{prop:minconn}
 Let~$\Omega\subseteq\R^2$ be a bounded, convex domain,
 let~$d\geq 1$ be an integer, and let~$a_1$, \ldots, $a_{2d}$
 be distinct points in~$\Omega$.
 Let~$\Q\in W^{1,1}(\Omega, \, \NN)\cap 
 W^{1,2}_{\loc}(\Omega\setminus\{a_1, \, \ldots, a_{2d}\}, \, \NN)$
 be a map with a non-orientable singularity at each~$a_j$. 
 If~$\M\in\SBV(\Omega, \, \SS^1)$ is a lifting 
 for~$\Q$ such that~$\S_{\M}\csubset\Omega$, then
 \[
  \H^1(\S_{\M}) \geq \mathbb{L}(a_1, \, \ldots, \, a_{2d})
 \]
 The equality holds if and only if
 there exists a minimal connection~$\{L_1, \, \ldots, \, L_d\}$
 for~$\{a_1, \, \ldots, a_d\}$ such that
 $\H^1(\S_\M\Delta \cup_{j=1}^d L_j) = 0$.
\end{prop}

Proposition~\ref{prop:lowerbound_SM} is an immediate consequence 
of Proposition~\ref{prop:minconn}. 
The proof of Proposition~\ref{prop:minconn}
is based on classical results in Geometric Measure Theory,
but we provide it in full detail for the reader's convenience.
Before we prove Proposition~\ref{prop:minconn},
we state a few preliminary results. 

\begin{lemma} \label{lemma:mindisj}
 If~$\{L_1, \, \ldots, \, L_d\}$ is a minimal connection 
 for~$\{a_1, \, \ldots, \, a_{2d}\}$, then 
 the~$L_j$'s are pairwise disjoint.
\end{lemma}
\begin{proof}
 Suppose, towards a contradiction, that
 $\{L_1, \, \ldots, \, L_d\}$ is a minimal connection 
 with~$L_1\cap L_2\neq\emptyset$. 
 The intersection~$L_1\cap L_2$ must be either 
 a non-degenerate sub-segment of both
 $L_1$ and~$L_2$ or a point. If~$L_1\cap L_2$
 is non-degenerate, then~$(L_1 \cup L_2) \setminus (L_1\cap L_2)$
 can be written as the disjoint union of
 two straight line segments, $K_1$ and~$K_2$,
 and
 \[
  \H^1(K_1) + \H^1(K_2) 
  = \H^1((L_1 \cup L_2) \setminus (L_1\cap L_2)) < 
  \H^1(L_1) + \H^1(L_2)
 \]
 This contradicts the minimality of~$\{L_1, \, \ldots, \, L_d\}$.
 Now, suppose that~$L_1\cap L_2$ is a point. By the pigeon-hole principle,
 $L_1\cap L_2$ cannot be an endpoint for either~$L_1$ or~$L_2$.
 Say, for instance, that~$L_1$ is the segment of
 endpoints~$a_1$, $a_2$, while~$L_2$ is the segment of
 endpoints~$a_3$, $a_4$. Let~$H_1$, $H_2$ be the segments
 of endpoints~$(a_1, \, a_3)$, $(a_2, \, a_4)$ respectively.
 Then, by the triangular inequality,
 \[
  \H^1(H_1) + \H^1(H_2) < \H^1(L_1) + \H^1(L_2),
 \]
 which contradicts again the minimality of~$\{L_1, \, \ldots, L_d\}$.
\end{proof}

\begin{lemma} \label{lemma:goodlifting}
 Let~$\Omega\subseteq\R^2$ be a bounded, convex domain
 and let~$a_1$, \ldots, $a_{2d}$ be distinct points
 in~$\Omega$. Let~$\Q\in W^{1,1}(\Omega, \, \NN)\cap 
 W^{1,2}_{\loc}(\Omega\setminus\{a_1, \, \ldots, a_{2d}\}, \, \NN)$
 be a map with a non-orientable singularity at each~$a_j$.
 If~$\{L_1, \, \ldots, \, L_d\}$ is a minimal connection
 for~$\{a_1, \, \ldots, \, a_{2d}\}$, then
 there exists a lifting~$\M^*\in\SBV(\Omega, \, \SS^1)$
 such that~$\H^1(\S_{\M^*}\Delta\cup_{j=1}^d L_j) = 0$.
\end{lemma}
\begin{proof}
 For any~$\rho>0$ and~$j\in\{1, \, \ldots, \, d\}$, we define
 \[
  U_{j,\rho} := \left\{x\in\R^2\colon
  \dist(x, \, L_j) < \rho \right\} \! .
 \]
 and
 \[
  \Omega_\rho :=\Omega\setminus \bigcup_{j=1}^d U_{j,\rho}.
 \]
 Since~$\Omega$ is convex, $L_j\subseteq\Omega$ for any~$j$
 and hence, $U_{j,\rho}\subseteq\Omega$ for any~$j$ 
 and~$\rho$ small enough. Each~$U_{j\,\rho}$
 is a simply connected domain with piecewise smooth boundary.
 Moreover, for~$\rho$ fixed and small, the sets~$U_{j,\rho}$
 are pairwise disjoint, because the~$L_j$'s
 are pairwise disjoint (Lemma~\ref{lemma:mindisj}).
 The trace of~$\Q$ on~$\partial U_{j,\rho}$
 is orientable, because~$\partial U_{j,\rho}$
 contains exactly two non-orientable singularities of~$\Q$.
 Then, for any~$\rho>0$ small enough,
 $\Q_{|\Omega_{\rho}}$ has a lifting~$\M^*_\rho\in W^{1,2}(\Omega_\rho, \, \SS^1)$
 \cite[Proposition~7]{BallZarnescu}. In fact, the lifting is unique up 
 to the choice of the sign~\cite[Proposition~2]{BallZarnescu};
 in particular, if~$0 < \rho_1 < \rho_2$ then we have
 either~$\M^*_{\rho_2} = \M^*_{\rho_1}$ a.e.~in~$\Omega_{\rho_2}$
 or~$\M^*_{\rho_2} = -\M^*_{\rho_1}$ a.e.~in~$\Omega_{\rho_2}$.
 As a consequence, for any sequence~$\rho_k\searrow 0$,
 we can choose liftings~$\M^*_{\rho_k}\in W^{1,2}(\Omega_{\rho_k}, \, \SS^1)$
 of~$\Q^*_{|\Omega_{\rho_k}}$ in such a way that
 $\M^*_{\rho_{k+1}} = \M^*_{\rho_k}$ a.e. in~$\Omega_{\rho_k}$.
 By glueing the~$\M^*_{\rho_k}$'s, we obtain a lifting
 \[
  \M^*\in W^{1,2}_{\loc}(\Omega\setminus\cup_j L_j, \, \SS^1)
 \]
 of~$\Q$.
 By differentiating the identity~\eqref{lifting}, we obtain
 $\sqrt{2}\abs{\nabla\M^*} = \abs{\nabla\Q}$ a.e.~and,
 since~$\nabla\Q\in L^1(\Omega, \, \R^2\otimes\R^{2\times 2})$
 by assumption, we deduce that
 so~$\M^*\in W^{1,1}(\Omega\setminus\cup_j L_j, \, \SS^1)$.
 The set~$\cup_j L_j$ has finite length and~$\M^*$
 is bounded, so we also have~$\M^*\in\SBV(\Omega, \, \SS^1)$
 (see~\cite[Proposition~4.4]{AmbrosioFuscoPallara}).
 
 By construction, we have $\S_{\M^*}\subseteq\cup_j L_j$.
 Therefore, it only remains to prove that 
 $\S_{\M^*}$ contains $\H^1$-almost all of~$\cup_j L_j$.
 Consider, for instance, the segment~$L_1$;
 up to a rotation and traslation, we can assume
 that~$L_1 = [0, \, b]\times\{0\}$ for some~$b>0$.
 Given a small parameter~$\rho>0$ and~$t\in (0, \, b)$,
 we define $K_{\rho, t} := (-\rho, \, t)\times (-\rho, \, \rho)$.
 Fubini theorem implies that, for a.e.~$\rho$ and~$t$,
 $\Q$ restricted to~$\partial K_{\rho, t}$
 belongs to~$W^{1,2}(\partial K_{\rho,t}, \, \NN)$
 and hence, by Sobolev embedding, is continuous.
 Since the segments~$L_j$ are pairwise disjoint
 by Lemma~\ref{lemma:mindisj}, for~$\rho$ small enough
 there is exactly one non-orientable singularity
 of~$\Q$ inside~$K_{\rho, t}$. Therefore,
 $\Q$ is non-orientable on~$\partial K_{\rho,t}$
 for a.e.~$t\in (0, \, b)$ and a.e.~$\rho>0$
 small enough; in particular, there is no continuous
 lifting of~$\Q$ on~$\partial K_{\rho,t}$.
 Since~$\M^*$ is continuous 
 on~$\partial K_{\rho,t}\setminus L_1$
 for a.e.~$\rho$ and~$t$, we conclude that
 $\S_{\M^*}$ contains $\H^1$-almost all of~$L_1$.
\end{proof}

Given a countably $1$-rectifiable set~$\Sigma\subseteq\R^2$
and a $\H^1$-measurable unit vector field
$\ttau\colon\Sigma\to\SS^1$, we say that~$\ttau$
is an orientation for~$\Sigma$ if~$\ttau(x)$ spans the 
(approximate) tangent line of~$\Sigma$ at~$x$, for~$\H^1$-a.e.~$x\in\Sigma$.
In case~$\Sigma$ is the jump set of an~$\SBV$-map~$\M$,
$\ttau\colon\S_{\M}\to\SS^1$ is an orientation for~$\S_{\M}$
if and only if $\ttau(x)\cdot\nnu_{\M}(x) = 0$ 
for~$\H^1$-a.e.~$x\in\S_{\M}$.

\begin{lemma} \label{lemma:samebd}
 Let~$\Omega\subseteq\R^2$ be a bounded, convex domain
 and let~$a_1$, \ldots, $a_{2d}$ be distinct points
 in~$\Omega$. Let~$\Q\in W^{1,1}(\Omega, \, \NN)\cap 
 W^{1,2}_{\loc}(\Omega\setminus\{a_1, \, \ldots, a_{2d}\}, \, \NN)$
 be a map with a non-orientable singularity at each~$a_j$.
 Let~$\{L_1, \, \ldots, \, L_d\}$ be a minimal connection
 for~$\{a_1, \, \ldots, \, a_{2d}\}$. Up to 
 relabelling, we assume that~$L_j$
 is the segment of endpoints~$a_{2j - 1}$, $a_{2j}$,
 for any~$j\in\{1, \, \ldots, \, d\}$.
 Let~$\M\in\SBV(\Omega, \, \SS^1)$ be a lifting for~$\Q$
 such that~$\S_{\M}\csubset\Omega$. Then,
 there exist $\H^1$-measurable sets~$T_j\subseteq L_j$
 and an orientation~$\ttau_{\M}$ for~$\S_{\M}$
 such that, for any~$\varphi\in C^\infty_\mathrm{c}(\R^2)$, 
 there holds
 \[
  \int_{\S_{\M}} \nabla\varphi\cdot\ttau_{\M}\,\d\H^1
  = \sum_{j = 1}^{d} \left(\varphi(a_{2j-1}) - \varphi(a_{2j})\right)
  - 2 \sum_{j = 1}^{d} \int_{T_j} \nabla\varphi\cdot
  \frac{a_{2j-1} - a_{2j}}{\abs{a_{2j-1} - a_{2j}}}\,\d\H^1
 \]
\end{lemma}
\begin{proof}
 Let~$\M^*\in\SBV(\Omega, \, \SS^1)$ be the lifting of~$\Q$
 given by Lemma~\ref{lemma:goodlifting}.
 By construction, $\S_{\M^*}$ coincides with~$\cup_j L_j\csubset\Omega$ 
 up to~$\H^1$-negligible sets.
 Since we have assumed that~$\S_{\M}\csubset\Omega$,
 there exists a neighbourhood~$U\subseteq\overline{\Omega}$
 of~$\partial\Omega$ in~$\overline{\Omega}$
 such that~$\M\in W^{1,1}(U, \, \SS^1)$,
 $\M^*\in W^{1,1}(U, \, \SS^1)$. A map that
 belongs to $W^{1,1}(U, \, \NN)$ has at most
 two different liftings in~$W^{1,1}(U, \, \SS^1)$,
 which differ only for the sign \cite[Proposition~2]{BallZarnescu}.
 Therefore, since both~$\M$ and~$\M^*$ are liftings of~$\Q$ in~$U$,
 we have that either~$\M = \M^*$ a.e.~in~$U$
 or~$\M = -\M^*$ a.e.~in~$U$. Changing the sign of~$\M^*$
 if necessary, we can assume that~$\M = -\M^*$ a.e.~in~$U$.
 Then, the set
 \[
  A := \{x\in\Omega\colon\M(x)\cdot\M^*(x) = 1\}
 \]
 is compactly contained in~$\Omega$.
 
 The Leibnitz rule for BV-functions 
 (see e.g.~\cite[Example~3.97]{AmbrosioFuscoPallara})
 implies that~$\M\cdot\M^*\in\SBV(\Omega; \, \{-1, \, 1\})$
 As a consequence, $A$ has finite perimeter in~$\Omega$
 (see e.g.~\cite[Theorem~3.40]{AmbrosioFuscoPallara});
 since~$A\csubset\Omega$, $A$ has also finite
 perimeter in~$\R^2$. By the Gauss-Green formula
 (see e.g.~\cite[Theorem~3.36, Eq.~(3.47)]{AmbrosioFuscoPallara}),
 for any~$\varphi\in C^\infty_\mathrm{c}(\R^2)$ we have
 \begin{equation} \label{samebd1}
  0 = \int_A \curl \nabla\varphi 
  = \int_{\partial^* A} \nabla\varphi\cdot\ttau_A\,\d\H^1,
 \end{equation}
 where~$\partial^* A$ is the reduced boundary of~$A$
 and~$\ttau_A$ is an orientation for~$\partial^* A$.
 Up to~$\H^1$-negligible sets, 
 $\partial^* A$ coincides with~$\S_{\M\cdot\M^*}$
 (see e.g.~\cite[Example~3.68 and Theorem~3.61]{AmbrosioFuscoPallara}). 
 By the Leibnitz rule for BV-functions, $\S_{\M\cdot\M^*}$
 coincides with~$\S_{\M}\Delta\S_{\M^*}$ up to~$\H^1$-negligible
 sets, so
 \begin{equation} \label{samebd2}
  \H^1\left(\partial^* A \, \Delta \, (\S_{\M}\Delta\cup_{j=1}^d L_j)\right) = 0
 \end{equation}
 For any~$j\in\{1, \, \ldots, \, d\}$, 
 let~$\ttau_j := (a_{2j-1} - a_{2j})/|a_{2j-1} - a_{2j}|$.
 We define an orientation~$\ttau_{\M}$ for~$\S_{\M}$
 as~$\ttau_{\M} := \ttau_{A}$ on~$\S_{\M}\setminus(\cup_j L_j)$
 (observing that, by~\eqref{samebd2}, $\H^1$-almost all 
 of~$\S_{\M}\setminus(\cup_j L_j)$ is contained in~$\partial^*A$)
 and~$\ttau_{\M} := \ttau_j$ on~$\S_{\M}\cap L_j$, for any~$j$.
 Then, \eqref{samebd1} and~\eqref{samebd2} imply
 \begin{equation} \label{samebd3}
  \begin{split}
   \int_{\S_{\M}} \nabla\varphi\cdot\ttau_\M\,\d\H^1 
   - \sum_{j=1}^d \int_{L_j} \nabla\varphi\cdot\ttau_j\,\d\H^1 
   + \sum_{j=1}^d \int_{L_j\setminus\S_{\M}} 
    (1 + \ttau_A\cdot\ttau_j)\nabla\varphi\cdot\ttau_j\,\d\H^1 
   = 0,
  \end{split}
 \end{equation}
 On~$\H^1$-almost all of~$L_j\setminus\S_{\M}$,
 both~$\ttau_j$ and~$\ttau_A$ are tangent to~$L_j$.
 Therefore, for $\H^1$-a.e.~$x\in L_j\setminus\S_{\M}$ 
 we have $\ttau_A(x) \cdot\ttau_j(x) \in \{-1, \, 1\}$.
 If we define~$T_j := \{x\in L_j\setminus\S_{\M}\colon 
 \ttau_A(x)\cdot\ttau_j(x) = 1\}$, then
 the lemma follows from~\eqref{samebd3}.
\end{proof}

Lemma~\ref{lemma:samebd} can be reformulated in terms of currents.
We recall a few basic definitions in the theory of currents, because
they will be useful to complete the proof of Proposition~\ref{prop:minconn}.
Actually, we will only work with currents of dimension~$0$ or~$1$.
We refer to, e.g., \cite{Federer, Simon-GMT} for more details.

A $0$-dimensional current, or~$0$-current, in~$\R^2$ is just a distribution on~$\R^2$,
i.e. an element of the topological dual of~$C^\infty_\mathrm{c}(\R^2)$
(where~$C^\infty_\mathrm{c}(\R^2)$ is given a suitable topology).
A $1$-dimensional current, or~$1$-current,
in~$\R^2$ is an element of the topological dual
of~$C^\infty_\mathrm{c}(\R^2; \, (\R^2)^\prime)$, where~$(\R^2)^\prime$
denotes the dual of~$\R^2$ and~$C^\infty_\mathrm{c}(\R^2; \, (\R^2)^\prime)$
is given a suitable topology, in much the same way 
as~$C^\infty_\mathrm{c}(\R^2)$.
In other words, a $1$-dimensional current is an~$\R^2$-valued
distribution. The boundary of a $1$-current~$T$
is the $0$-current~$\partial T$ defined by
\[
 \langle \partial T, \, \varphi\rangle 
  := \langle T, \, \d\varphi\rangle 
  \qquad \textrm{for any } \varphi\in C^\infty_\mathrm{c}(\R^2).
\]
The mass of a~$1$-current~$T$ is defined as
\[
 \mathbb{M}(T) := \sup\left\{\langle T, \, \omega\rangle\colon 
 \omega\in C^\infty_{\mathrm{c}}(\R^2; \, (\R^2)^\prime), \
 \abs{\omega(x)} \leq 1 \quad \textrm{for any } x\in\R^2\right\} \!;
\]
the mass of a~$0$-current is defined analogously.

We single out a particular subset of currents,
called integer-multiplicity rectifiable currents or 
rectifiable currents for short. A rectifiable $0$-current
is a current of the form
\begin{equation} \label{rectcurr0}
 T = \sum_{k=1}^p n_k \, \delta_{b_k},
\end{equation}
where~$k\in\N$, $n_k\in\Z$ and~$b_k\in\R^2$.
A rectifibiable~$0$-current has finite mass:
for the current~$T$ given by~\eqref{rectcurr0},
we have $\mathbb{M}(T) = \sum_{k=1}^p \abs{n_k}$.
A~$1$-current is called rectifiable if there exist a countably
$1$-rectifiable set~$\Sigma\subseteq\R^2$
with~$\H^1(\Sigma) < +\infty$, an orientation~$\ttau\colon\Sigma\to\SS^1$ 
for~$\Sigma$ and an integer-valued, $\H^1$-integrable 
function~$\theta\colon\Sigma\to\Z$ such that
\begin{equation} \label{rectcurr}
 \langle T, \, \omega\rangle 
 = \int_\Sigma \theta(x) \langle\ttau(x), \, \omega(x) \rangle \, \d\H^1(x)
 \qquad \textrm{for any }
 \omega\in C^\infty_{\mathrm{c}}(\R^2; \, (\R^2)^\prime).
\end{equation}
The current~$T$ defined by~\eqref{rectcurr}
is called the rectifiable $1$-current carried by~$\Sigma$,
with multiplicity~$\theta$ and orientation~$\ttau$; it satisfies
\[
 \mathbb{M}(T) = \int_\Sigma \abs{\theta(x)} \, \d\H^1(x) < +\infty.
\]
The set of rectifiable $0$-currents, respectively rectifiable $1$-currents,
is denoted by~$\mathscr{R}_0(\R^2)$, respectively~$\mathscr{R}_1(\R^2)$.


Given a Lipschitz, injective map~$\f\colon [0, \, 1]\to\R^2$,
we denote by~$\f_{\#}I$ the rectifiable~$1$-current
carried by~$\f([0, \, 1])$, with unit multiplicity
and orientation given by~$\f^\prime$. The mass of~$\f_{\#} I$
is the length of the curve parametrised by~$\f$
and~$\partial(\f_{\#} I) = \delta_{\f(1)} - \delta_{\f(0)}$;
in particular, $\partial(\f_{\#}I) = 0$ if~$\f(1) = \f(0)$.
The assumption that~$\f$ is injective can be relaxed;
for instance, if the curve parametrised by~$\f$
has only a finite number of self-intersections,
then $\f_{\#}I$ is still well-defined and the properties above 
remain valid.


We take a bounded, convex domain~$\Omega\subseteq\R^2$,
distinct points~$a_1$, \ldots $a_{2d}$ and a map
$\Q\in W^{1,1}(\Omega, \, \NN)\cap 
W^{1,2}_{\loc}(\Omega\setminus\{a_1, \, \ldots, \, a_{2d}\}, \, \NN)$ 
with a non-orientable singularity at each~$a_i$.
Let~$\M\in\SBV(\Omega, \, \SS^1)$ be a lifting of~$\Q$
such that~$\S_{\M}\csubset\Omega$. 
By Federer-Vol'pert theorem 
(see e.g.~\cite[Theorem~3.78]{AmbrosioFuscoPallara}),
the set~$\S_{\M}$ is countably $1$-rectifiable.
We claim that $\H^1(\S_{\M})<+\infty$. 
Indeed, since~$\Q$ has no jump set,
by the BV-chain rule (see e.g.~\cite[Theorem~3.96]{AmbrosioFuscoPallara})
we deduce that~$\M^+(x) = - \M^-(x)$ at $\H^1$-a.e. point~$x\in\S_{\M}$.
This implies
\[
 \H^1(\S_{\M}) \leq \frac{1}{2} \int_{\S_\M} |\M^+ - \M^-| \, \d\H^1
 \lesssim \abs{\D\M}(\Omega) < +\infty,
\]
as claimed. In particular, there is a well-defined,
rectifiable~$1$-current carried by~$\S_{\M}$, with unit multiplicity
and orientation~$\ttau_{\M}$ given by Lemma~\ref{lemma:samebd};
we denote it by~$\llbracket\S_{\M}\rrbracket$.
Lemma~\ref{lemma:samebd} provides information
on the boundary of~$\llbracket\S_{\M}\rrbracket$.
More precisely, Lemma~\ref{lemma:samebd} implies
\begin{equation} \label{bdjump}
 \partial\llbracket\S_{\M}\rrbracket = \sum_{i=1}^{2d} \delta_{a_i}
 + 2 \, \partial Q 
\end{equation}
where~$Q$ is a rectifiable $1$-chain, defined as
\begin{equation} \label{bdjumpQ}
 \langle Q, \, \psi\rangle :=  \sum_{j = 1}^{d} \int_{T_j} \left\langle
  \psi(x), \frac{a_{2j-1} - a_{2j}}{\abs{a_{2j-1} - a_{2j}}} \right\rangle\,\d\H^1(x)
\end{equation}
for any~$\psi\in C^\infty_{\mathrm{c}}(\R^2, \, (\R^2)^\prime)$.
The~$T_j$'s are $1$-rectifiable sets that depend
only on~$\M$, not on~$\psi$, as given by Lemma~\ref{lemma:samebd}.

\begin{lemma} \label{lemma:structlift}
 Let~$\Omega$, $\Q$ be as above.
 Let~$\M\in\SBV(\Omega, \, \SS^1)$ be a lifting 
 of~$\Q$ with~$\S_{\M}\csubset\Omega$. Then, there exist
 countably may Lipschitz functions~$\f_j\colon [0, \, 1]\to\R^2$,
 with finitely many self-intersections,
 a rectifiable $1$-current~$R\in\mathscr{R}_1(\R^2)$
 and a permutation $\sigma$ of the indices~$\{1, \, \ldots, \, 2d\}$
 such that the following properties hold:
 \begin{gather}
  \llbracket\S_{\M}\rrbracket 
   = \sum_{j\geq 1} \f_{j,\#}I + 2R \label{structlift} \\
  \mathbb{M}(\llbracket\S_{\M}\rrbracket) 
   = \sum_{j\geq 1} \mathbb{M}\left(\f_{j,\#}I\right)
   \label{structmass} \\
  \partial(\f_{j,\#}I) = \delta_{\sigma(2j)} - \delta_{\sigma(2j-1)}
   \quad \textrm{if } j\in\{1, \, \ldots, \, d\},
   \qquad \partial(\f_{j,\#}I) = 0 \quad \textrm{otherwise.}
   \label{structbd}
 \end{gather}
\end{lemma}

 \begin{figure}[t]
  \centering
  \includegraphics[height=.3\textheight]{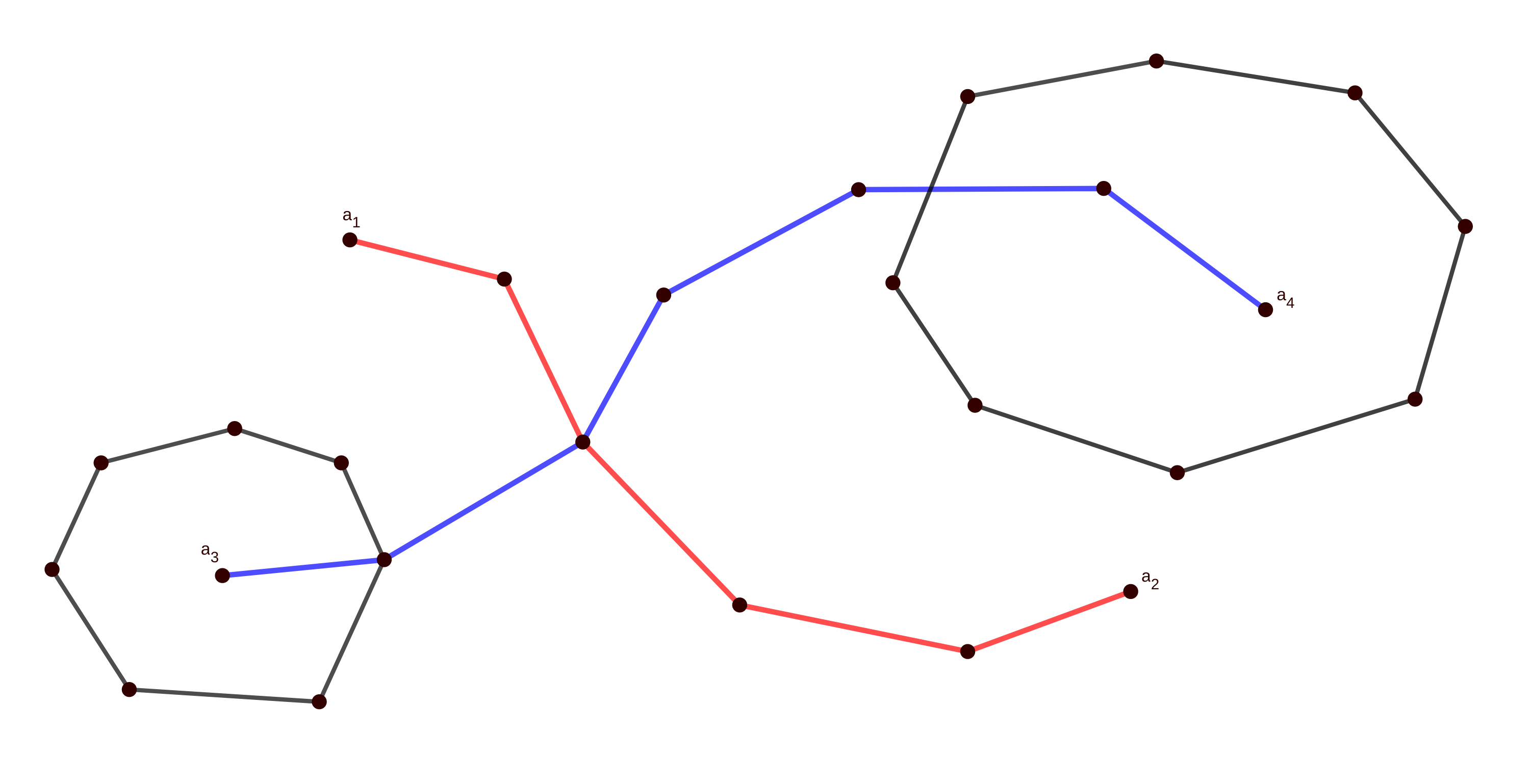}
  \caption{A decomposition of the graph~$\mathscr{G}$,
  as defined in the proof of Lemma~\ref{lemma:structlift},
  into edge-disjoint trails~$\mathscr{E}_1$ (in red)
  and~$\mathscr{E}_2$ (in blue).
  In addition to the edges of~$\mathscr{G}$, there may be
  other cycles, carried by the curves~$\g_j([0, \, 1])$
  with~$j\geq q + 1$; they are shown in black.}
  \label{fig:graph}
\end{figure}

\begin{proof}
 By applying, e.g., \cite[Theorem~6.3]{Ziemer} 
 or~\cite[Corollary~4.2]{AmbrosioWenger}, we find
 rectifiable $1$-currents~$T$, $R\in\mathscr{R}_1(\R^2)$
 such that $\mathbb{M}(T) = \mathbb{M}(\llbracket\S_{\M}\rrbracket)
 = \H^1(\S_{\M})$, $\partial T\in\mathscr{R}_0(\R^2)$ and 
 \begin{equation} \label{structlift1}
  T = \llbracket\S_{\M}\rrbracket + 2R.
 \end{equation}
 By taking the boundary of both sides of~\eqref{structlift1},
 and applying~\eqref{bdjump}, we obtain
 \begin{equation} \label{structlift2}
  \partial T = \sum_{i=1}^{2d} \delta_{a_i} + 2P
 \end{equation}
 with~$P :=\partial(R + Q)$ (and~$Q$ as in~\eqref{bdjumpQ}). 
 The current~$2P = \partial T -\sum_{i=1}^{2d}\delta_{a_i}$
 is rectifiable, so~$\mathbb{M}(P)<+\infty$. Moreover,
 $P$ isthe boundary of a rectifiable~$1$-current. Then, 
 Federer's closure theorem \cite[4.2.16]{Federer}
 implies that~$P$ itself is rectifiable. As a consequence,
 we can re-write~\eqref{structlift2} as
 \begin{equation} \label{structlift3}
  \partial T = \sum_{i=1}^{2d} \delta_{a_i} 
  + 2\sum_{k=1}^p n_k \, \delta_{b_k},
 \end{equation}
 for some integers~$n_k$ and some distinct points~$b_k\in\R^2$.
 By applying~\cite[4.2.25]{Federer}, we find
 countably many Lipschitz, injective maps~$\g_j\colon[0, \, 1]\to\R$
 such that
 \begin{equation} \label{structlift4}
  T = \sum_{j\geq 1} \g_{j,\#}I, \qquad
   \sum_{j\geq 1} \left(\mathbb{M}(\g_{j,\#}I) 
   + \mathbb{M}(\partial(\g_{j,\#}I))\right)
   = \mathbb{M}(T) + \mathbb{M}(\partial T) < +\infty.
 \end{equation}
 For any~$j$, we have either~$\partial(\g_{j,\#}I) = 0$
 (if~$\g_{j,\#}(1) = \g_{j,\#}(0)$) 
 or~$\mathbb{M}(\partial(\g_{j,\#}I)) = 2$ (otherwise).
 Therefore, by~\eqref{structlift4}, there are only finitely many
 indices~$j$ such that~$\g_{j,\#}(1) \neq \g_{j,\#}(0)$.
 Up to a relabelling of the~$\g_j$'s, we assume that there is
 an integer~$q$ such that $\g_{j,\#}(1) \neq \g_{j,\#}(0)$
 if and only if~$j\leq q$.
 
 Now, the problem reduces to a combinatorial, or graph-theoretical, one.
 We consider the finite (multi-)graph~$\mathscr{G}$ whose edges are the curves 
 parametrised by~$\g_1$, \ldots, $\g_q$, and whose vertices
 are the endpoints of such curves. There can be
 two or more edges that join the same pair of vertices.
 However, we can disregard the orientation of the edges:
 changing the orientation of the curve parametrised by~$\g_j$
 corresponds to passing from the current~$\g_{j,\#}I$
 to the current~$-\g_{j,\#}I$; the difference 
 $\g_{j,\#}I - (-\g_{j,\#}I) = 2\g_{j,\#}I$ can be absorbed
 into the term~$2R$ that appears in~\eqref{structlift}.

 We would like to partition the set of edges of~$\mathscr{G}$
 into~$d$ disjoint subsets~$\mathscr{E}_1$, \ldots $\mathscr{E}_d$,
 where each~$\mathscr{E}_j$ is a trail (i.e., a sequence 
 of distinct edges such that each edge is adjacent to the next one)
 and, for a suitable permutation~$\sigma$ of~$\{1, \, \ldots, \, 2d\}$,
 the trail~$\mathscr{E}_j$ connects~$a_{\sigma(2j-1)}$ with~$a_{\sigma(2j)}$.
 If we do so, then we can define~$\f_j\colon [0, \, 1]\to\R^2$ 
 for~$j\in\{1, \, \ldots, d\}$
 as a Lipschitz map that parameterises the trail~$\mathscr{E}_j$,
 with suitable orientations of each edge;
 for~$j\geq d+1$, we define~$\f_j := \g_{q + j - d}$.
 With this choice of~$\f_j$, the lemma follows.
 It is possible to find~$\mathscr{E}_1$, \ldots $\mathscr{E}_d$
 as required because the graph~$\mathscr{G}$ has the following property:
 any~$a_i$ is an endpoint of an \emph{odd} number 
 of edges of~$\mathscr{G}$; conversely, any vertex of~$\mathscr{G}$
 other than the~$a_i$'s is an endpoint of an \emph{even}
 number of edges of~$\mathscr{G}$. This property 
 follows from~\eqref{structlift3}. Then, we can construct~$\mathscr{E}_1$,
 \ldots $\mathscr{E}_d$ by reasoning along the lines of, e.g.,
 \cite[Theorem~12]{Bollobas}.
\end{proof}

We can now conclude the proof of Proposition~\ref{prop:minconn}.

\begin{proof}[Proof of Proposition~\ref{prop:minconn}]
 We consider the decomposition of~$\llbracket\S_{\M}\rrbracket$
 given by Lemma~\ref{lemma:structlift}.
 Thanks to~\eqref{structbd}, for any~$j\in\{1, \, \ldots, \, d\}$
 the curve parametrised by~$\f_j$ 
 joins~$a_{\sigma(2j - 1)}$ with~$a_{\sigma(2j)}$. Then,
 \[
  \H^1(\S_{\M}) = \mathbb{M}(\llbracket\S_{\M}\rrbracket)
  \geq \sum_{j=1}^d \mathbb{M}(\f_{j,\#}I)
  \geq \sum_{j=1}^d \abs{a_{\sigma(2j)} - a_{\sigma(2j-1)}}
  \geq \mathbb{L}(a_1, \, \ldots, \, a_{2d}).
 \]
 The equality can only be attained if there are exactly~$d$
 maps~$\f_j$ and each of them parametrises
 a straight line segment.
\end{proof}

\section{Properties of~$f_\eps$}
\label{app:feps}

The aim of this section is to prove Lemma~\ref{lemma:feps}.
We first of all, we characterise the zero-set of 
the potential~$f_\eps$, in terms of the (unique) 
solution to an algebraic system depending on~$\eps$ and~$\beta$.

\begin{lemma} \label{lemma:Xeps}
 For any~$\eps>0$, the algebraic system
 \begin{equation} \label{Xeps}
  \begin{cases}
   X(X - 1 - \beta^2\eps)^2 = \dfrac{\beta^2\eps^2}{2} \\
   X > 1 + \beta^2\eps
  \end{cases}
 \end{equation}
 admits a unique solution~$X_\eps$, which satisfies
 \[
  X_\eps = 1 + \frac{1}{\sqrt{2}}\left(\sqrt{2}\beta + 1\right)  \beta\eps
  - \frac{1}{4} \left(\sqrt{2}\beta + 1\right) \beta^2 \eps^2
  + \mathrm{o}(\eps^2) \qquad \textrm{as } \eps\to 0.
 \]
\end{lemma}
\begin{proof}
 The function~$P(X) := X(X - 1 - \beta^2\eps)^2$ 
 is continuous and strictly increasing
 in the interval~$[1+\beta^2\eps, \, +\infty)$, because
 $P^\prime(X) = (X-1-\beta^2\eps)(3X - 1 - \beta^2\eps)>0$ 
 for~$X>1+\beta^2\eps$.
 Moreover, $P(1+\beta^2\eps) = 0$ and~$P(X)\to +\infty$ as~$X\to+\infty$.
 Therefore, the system~\eqref{Xeps} admits a unique solution.
 Let~$Y_\eps>0$ be such that
 \[
  X_\eps = 1 + \beta^2\eps + \beta\eps \, Y_\eps.
 \]
 Then, \eqref{Xeps} can be rewritten as
 \begin{equation} \label{Xeps1}
  Y_\eps^2 = \frac{1}{2 + 2\beta^2\eps + 2\beta\eps \, Y_\eps}, 
 \end{equation}
 which implies $Y_\eps\to1/\sqrt{2}$ as~$\eps\to 0$.
 Using~\eqref{Xeps1} again, we obtain
 \[
  Y_\eps = \frac{1}{\left(2 + 2\beta^2\eps
   + \sqrt{2}\beta\eps + \mathrm{o}(\eps)\right)^{1/2}}
  = \frac{1}{\sqrt{2}} 
   - \frac{1}{4}\left(\sqrt{2}\beta + 1\right) \beta \eps + \mathrm{o}(\eps)
 \]
 as~$\eps\to 0$, and the lemma follows.
\end{proof}

For any~$\eps> 0$, we define
\begin{equation} \label{seps}
 s_\eps := X_\eps^{1/2}, \qquad
 \lambda_\eps := \left(\frac{X_\eps - 1}
  {X_\eps - 1 - \beta^2\eps}\right)^{1/2}.
\end{equation}
Lemma~\ref{lemma:Xeps} implies, via routine algebraic manipulations, that
\begin{gather}
  s_\eps = 1 + \frac{1}{2\sqrt{2}}\left(\sqrt{2}\beta + 1\right)\beta\eps 
   + \mathrm{o}(\eps), \qquad
  \lambda_\eps^2 = \sqrt{2}\beta + 1
   + \frac{1}{2}\left(\sqrt{2}\beta + 1\right) \beta^2\eps 
   + \mathrm{o}(\eps) \label{s,lambda_eps}
\end{gather}
as~$\eps\to 0$.

\begin{lemma} \label{lemma:Neps}
 A pair~$(\Q, \, \M)\in\Sz\times\R^2$ 
 satisfies~$f_\eps(\Q, \, \M) = 0$ if and only if
 \begin{equation*} 
  \begin{split}
   \abs{\M} = \lambda_\eps, \qquad 
   \Q = \sqrt{2}\, s_\eps \left(\frac{\M\otimes\M}
    {\lambda_\eps^2} - \frac{\I}{2}\right) \!. 
  \end{split}
 \end{equation*}
\end{lemma}
\begin{proof}
 By imposing that the gradient of~$f_\eps$ is equal to zero,
 we obtain the system
 \begin{gather}
   (\abs{\Q}^2 - 1)\Q = \beta \eps \left(\M\otimes\M
     - \dfrac{\abs{\M}^2}{2}\I\right) \label{Neps1-Q} \\
   (\abs{\M}^2 - 1)\M = 2\beta \, \Q\M. \label{Neps1-M}
 \end{gather}
 Suppose first that~$\M = 0$. Then, Equation~\eqref{Neps1-Q}
 implies that either~$\Q = 0$ or~$\abs{\Q} = 1$.
 The pair~$\Q = 0$, $\M = 0$ is not a minimiser for~$f_\eps$,
 because $\nabla_{\Q}^2 f_\eps(0, \, 0) = - \I < 0$.
 If~$\abs{\Q} = 1$, $\M=0$, then
 $\nabla_{\M}^2 f_\eps(\Q, \, 0) = -\eps(\I + 2\beta\Q)$.
 Since~$\Q$ is non-zero, symmetric and trace-free, 
 there exists~$\n\in\SS^1$ such that~$\Q\n\cdot\n > 0$. Then, 
 $\nabla_{\M}^2 f_\eps(\Q, \, 0)\n\cdot\n <0$,
 so the pair~$(\Q, \, \M=0)$ is not a minimiser of~$f_\eps$.
 It remains to consider the case~$\M\neq 0$. In this case, we
 have~$\Q\neq 0$ and~$\abs{\Q}\neq 1$, due to~\eqref{Neps1-Q}.
 Solving~\eqref{Neps1-Q} for~$\Q$, and then substituting
 in~\eqref{Neps1-M}, we obtain
 \begin{gather*}
   \abs{\M}^2 - 1 = \frac{\beta^2\eps\abs{\M}^2}{\abs{\Q}^2 - 1}
 \end{gather*}
 and hence, solving for~$\abs{\M}^2$,
 \begin{gather} \label{Neps2-M}
   \abs{\M}^2 = \frac{\abs{\Q}^2 - 1}{\abs{\Q}^2 - 1 - \beta^2\eps} .
 \end{gather}
 By taking the squared norm of 
 both sides of~\eqref{Neps1-Q}, we obtain
 \begin{equation*} 
  (\abs{\Q}^2 - 1)^2\abs{\Q}^2 = \frac{\beta^2\eps^2}{2} \abs{\M}^4 
 \end{equation*}
 and hence, using~\eqref{Neps2-M},
 \begin{equation} \label{Neps2-Q}
  \abs{\Q}^2 = \frac{\beta^2\eps^2}{2(\abs{\Q}^2 - 1 - \beta^2\eps)^2} 
 \end{equation}
 We either have~$\abs{\Q}^2 < 1$
 or~$\abs{\Q}^2 > 1 + \beta^2\eps$, because of~\eqref{Neps2-M}.
 On the other hand, by imposing that the second 
 derivative of~$f_\eps$ with respect to~$\Q$
 is non-negative, we obtain~$\abs{\Q}^2\geq 1$. Therefore, 
 we conclude that~$\abs{\Q}^2 = X_\eps$ is the unique solution
 to the system~\eqref{Xeps} and, taking~\eqref{Neps2-M} 
 into account, the proposition follows.
\end{proof}

We can now prove Lemma~\ref{lemma:feps}. 
For convenience, we recall the statement here.

\begin{lemma} \label{lemma:fepsapp}
 The potential~$f_\eps$ satisfies the following properties.
 \begin{enumerate}[label=(\roman*)]
  \item The constant~$\kappa_\eps$ in~\eqref{f}, uniquely defined by imposing
  the condition~$\inf f_\eps = 0$, satisfies
  \begin{equation} \label{kepsapp}
   \kappa_\eps = \frac{1}{2} \left(\beta^2 + \sqrt{2} \beta\right) \eps
   + \kappa_*^2 \, \eps^2 + \o(\eps^2) 
  \end{equation}
  In particular, $\kappa_\eps\geq 0$ for~$\eps$ small enough.
  
  \item If~$(\Q, \, \M)\in\Sz\times\R^2$ is such that
  \begin{equation} \label{MQ-almostoptimal}
   \abs{\M} = (\sqrt{2}\beta + 1)^{1/2},
  \qquad \Q = \sqrt{2}\left(\frac{\M\otimes\M}{\sqrt{2}\beta + 1} 
   - \frac{\I}{2}\right)
  \end{equation}
  then $f_\eps(\Q, \, \M) = \kappa_* \, \eps^2 + \o(\eps^2)$.
  
  \item If~$\eps$ is sufficiently small, then
  \begin{align} 
   \frac{1}{\eps^2} f_\eps(\Q, \, \M) 
    &\geq \frac{1}{4\eps^2}(\abs{\Q}^2 - 1)^2
     - \frac{\beta}{\sqrt{2}\eps} \abs{\M}^2 
     \, \abs{\abs{\Q} - 1} \label{flowdapp-bis} \\
   \frac{1}{\eps^2} f_\eps(\Q, \, \M) 
    &\geq \frac{1}{8\eps^2}(\abs{\Q}^2 - 1)^2 
    - \beta^2\abs{\M}^4 \label{flowbdapp}
  \end{align}
  for any~$(\Q, \, \M)\in\Sz\times\R^2$.
 \end{enumerate}
\end{lemma}
 
\noindent
\emph{Proof of Statement~(i)}. 
 Let~$(\Q^*_*, \, \M^*)\in\Sz\times\R^2$ be a minimiser for~$f_\eps$,
 i.e.~$f_\eps(\Q^*_*, \, \M^*) = 0$. By Lemma~\ref{lemma:Neps}, we have
 \begin{equation*} 
  \begin{split}
   \kappa_\eps &= -\frac{1}{4} (|\Q^*_*|^2-1)^2
     - \frac{\eps}{4} (|\M^*|^2-1)^2 + \beta\eps \, \Q^*_*\M^*\cdot\M^* \\
   &= -\frac{1}{4} (s_\eps^2 - 1)^2 - \frac{\eps}{4}(\lambda_\eps^2-1)^2 
   + \frac{\beta\eps}{\sqrt{2}} \, s_\eps \lambda_\eps^2
  \end{split}
 \end{equation*}
 We expand~$s_\eps$, $\lambda_\eps$ in terms of~$\eps$,
 as given by~\eqref{s,lambda_eps}. Equation~\eqref{kepsapp} 
 then follows by standard algebraic manipulations.

\medskip
\noindent
\emph{Proof of Statement~(ii)}. 
 The assumption~\eqref{MQ-almostoptimal} implies
 \begin{equation*} 
  \abs{\Q} = 1, \qquad \Q\M\cdot\M = \sqrt{2}
  \left(\frac{\abs{\M}^4}{\sqrt{2}\beta + 1} - \frac{1}{2}\abs{\M}^2\right)
  = \frac{\sqrt{2}}{2}\left(\sqrt{2}\beta + 1\right) = \beta + \frac{\sqrt{2}}{2} 
 \end{equation*}
 Therefore,
 \begin{equation*}
  \begin{split}
   f_\eps(\Q, \, \M) = \frac{\eps \, \beta^2}{2}
    - \beta \eps \left( \beta + \frac{\sqrt{2}}{2} \right)
    + \kappa_\eps
    = - \frac{\eps}{2} \left(\beta^2 + \sqrt{2}\beta\right)
    + \kappa_\eps \stackrel{\eqref{kepsapp}}{=} \kappa_* \, \eps^2 + \o(\eps^2)
  \end{split}
 \end{equation*}
 
 \medskip
 \noindent
 \emph{Proof of Statement~(iii)}. 
 When~$\Q = 0$, we have $f_\eps(0, \, \M) \geq 1/4 + \kappa_\eps$
 and~$\kappa_\eps>0$ is positive for~$\eps$ small enough,
 due to~\eqref{kepsapp}. Then, \eqref{flowdapp-bis} follows.
 When~$\Q\neq 0$, it is convenient to
 make the change of variables we have introduced in
 Section~\ref{sect:changevar}.  We write
 \begin{equation*} 
  \Q = \frac{\abs{\Q}}{\sqrt{2}} 
  \left(\n\otimes\n - \m\otimes\m\right) 
 \end{equation*}
 where~$(\n, \, \m)$ is an orthonormal basis of eigenvalues for~$\Q$.
 We define~$\u = (u_1, \, u_2)\in\R^2$ as~$u_1 := \M\cdot\n$,
 $u_2 := \M\cdot\m$. The potential~$f_\eps$ can be expressed
 in terms of~$\Q$, $\u$ as (see Equation~\eqref{hf}),
 \begin{equation*}
  \begin{split}
    \frac{1}{\eps^2} f_\eps(\Q, \, \M) 
     &= \frac{1}{4\eps^2}(\abs{\Q}^2 - 1)^2 + \frac{1}{\eps} h(\u)
     + \frac{\beta}{\sqrt{2}\,\eps} (1 - \abs{\Q}) \, (u_1^2 - u_2^2) \\
     &\qquad\qquad + \frac{\kappa_\eps}{\eps^2} 
      - \frac{1}{2\eps}(\beta^2 + \sqrt{2}\beta) 
  \end{split}
 \end{equation*}
 where~$h$ is defined in~\eqref{h}. By Lemma~\eqref{lemma:h},
 we know that~$h\geq 0$. Moreover, Equation~\eqref{kepsapp} implies
 \[
  \frac{\kappa_\eps}{\eps^2} 
      - \frac{1}{2\eps}(\beta^2 + \sqrt{2}\beta) = \kappa_*^2 + \o(1)\geq 0
 \]
 for~$\eps$ small enough. Then, \eqref{flowdapp-bis} follows.
 Equation~\eqref{flowbdapp} follows from~\eqref{flowdapp-bis}, as
 \[
  \begin{split}
   \frac{\beta}{\sqrt{2}\eps} \abs{\M}^2 \abs{\abs{\Q} - 1} 
   \leq \beta^2\abs{\M}^4 + \frac{1}{8\eps^2} (\abs{\Q} - 1)^2
   \leq  \beta^2\abs{\M}^4 + \frac{1}{8\eps^2} (\abs{\Q}^2 - 1)^2 
  \end{split}
 \]
 \qed

\section{Proof of Lemma~\ref{lemma:interp}}
\label{sect:interp}

The aim of this section is to prove Lemma~\ref{lemma:interp-app},
which we recall here for the convenience of the reader. 
We recall that~$g_\eps\colon\Sz\to\R$ is the function defined in~\eqref{geps}.

\begin{lemma} \label{lemma:interp-app}
 Let~$B = B_r(x_0)\subseteq\Omega$ be an open ball.
 Suppose that~$\Q^*_\eps\rightharpoonup \Q^*$ 
 weakly in~$W^{1,2}(\partial B)$ and that
 \begin{equation} \label{hp:interp-app}
   \begin{split}
     \int_{\partial B} \left(\frac{1}{2}\abs{\nabla\Q^*_\eps}^2
     + g_\eps(\Q^*_\eps) \right) \d\H^1 
     \leq C
   \end{split}
  \end{equation}
  for some constant~$C$ that may depend on the radius~$r$, but not on~$\eps$.
  Then, there exists a map~$\Q_\eps\in W^{1,2}(B, \, \Sz)$ such that
  \begin{gather} 
   \Q_\eps = \Q_\eps^* \quad \textrm{on } \partial B, \qquad 
   \abs{\Q_\eps} \geq \frac{1}{2} \quad \textrm{in } B \label{interp1app} \\
   \int_{B}
    \left(\frac{1}{2}\abs{\nabla\Q_\eps}^2 + g_\eps(\Q_\eps) \right) \d x
   \to \frac{1}{2} \int_{B} \abs{\nabla\Q^*}^2 \, \d x \label{interp2app}
  \end{gather}
\end{lemma}

Lemma~\ref{lemma:interp-app} is inspired by interpolation results in the 
literature on harmonic maps (see e.g.~\cite[Lemma~1]{Luckhaus-PartialReg}). 
As we work in a two-dimensional domain, we can simplify some
points of the proof in~\cite{Luckhaus-PartialReg}. On the other hand,
we need to estimate the contributions from the term~$g_\eps(\Q_\eps)$,
which is not present in~\cite{Luckhaus-PartialReg}.

\begin{proof}[Proof of Lemma~\ref{lemma:interp-app}]
  Without loss of generality, we can assume 
  that~$x_0 = 0$. By assumption, we have
  $\Q_{\eps}^*\rightharpoonup\Q^*$ weakly in~$W^{1,2}(\partial B)$
  and hence, by Sobolev embedding, uniformly on~$\partial B$.
  In particular, $\abs{\Q^*_\eps}\to 1$ uniformly on~$\partial B$.
  Let~$\lambda_\eps > 0$ be a small number, to be chosen later on. 
  We consider the decomposition~$B = A_\eps^1\cup A^2_\eps
  \cup A^3_\eps$, where
  \[
   A^1_\eps := B_r \setminus \bar{B}_{r - \lambda_\eps r}, \qquad
   A^2_\eps := \bar{B}_{r - \lambda_\eps r} 
    \setminus \bar{B}_{r - 2\lambda_\eps r}, 
   \qquad A^3_\eps := \bar{B}_{r - 2\lambda_\eps r}
  \]
  We define the map~$\Q_\eps$ using polar 
  coordinates~$(\rho, \, \theta)$, as follows.
  If~$x = \rho e^{i\theta}\in A^1_\eps$, we define
  \[
   \Q_\eps(x)
    :=  t_\eps(\rho) \, \Q_\eps^*(re^{i\theta}) + (1 + \kappa_*\eps) 
    (1 - t_\eps(\rho)) \, \dfrac{\Q_\eps^*(re^{i\theta})}{\abs{\Q^*_\eps(re^{i\theta})}}
  \]
  where~$t_\eps\colon\R\to\R$ is an affine function 
  such that $t_\eps(r) = 1$, $t_\eps(r - \lambda_\eps r) = 0$.
  If~$x = \rho e^{i\theta}\in A^2_\eps$,  we define
  \[
   \Q_\eps(x)
   :=  (1 + \kappa_*\eps) \, 
    \frac{s_\eps(\rho) \, \Q^*_\eps(re^{i\theta}) + (1 - s_\eps(\rho)) \, 
     \Q^*(re^{i\theta}) }
    {\abs{s_\eps(\rho) \, \Q^*_\eps(re^{i\theta}) + (1 - s_\eps(\rho)) \, 
     \Q^*(re^{i\theta}) }}
  \]
 where~$s_\eps\colon\R\to\R$ is an affine function such that
 $s_\eps(r - \lambda_\eps r) = 1$, $s_\eps(r - 2\lambda_\eps r) = 0$.
 Finally, if~$x\in A_\eps^3$, we define
 \[
  \Q_\eps(x)
   :=  (1 + \kappa_*\eps) \, \Q^*\left(\frac{x}{1 - 2\lambda_\eps}\right)
 \]
 The map~$\Q_\eps$ is well-defined in~$B$, beacuse~$\abs{\Q_\eps}\to 1$
 uniformly on~$\partial B$. 
 Moreover, we have $\abs{\Q_\eps} \geq 1/2$
 for~$\eps$ small enough, $\Q_\eps\in W^{1,2}(B, \, \Sz)$
 (at the interfaces between~$A_\eps^1$, $A_\eps^2$, $A^3_\eps$,
 the traces of~$\Q_\eps$ on either side of the interface match),
 and~$\Q_\eps = \Q^*_\eps$ on~$\partial B$.

 It only remains to prove~\eqref{interp2app}. First, we estimate the 
 integral of~$g_\eps(\Q_\eps)$. On~$A^2_\eps\cup A^3_\eps$,
 we have~$|\Q_\eps| = 1 + \kappa_*\eps$ and hence,
 substituting in~\eqref{geps},
 \begin{equation} \label{interp4}
  g_\eps(\Q_\eps) 
  = \kappa_*^2\left(\frac{1}{4}(2 + \kappa_*\eps)^2 - 1\right) 
  = \kappa_*^2\left(\kappa_*\eps + \kappa_*^2\eps^2\right) = \mathrm{O}(\eps) 
 \end{equation}
 We consider the annulus~$A_\eps^1$. By Lemma~\ref{lemma:geps},
 we have
 \begin{equation*}
  \begin{split}
   g_\eps(\Q_\eps) \leq
    \left(\frac{1}{\eps}(|\Q_\eps| - 1) - \kappa_*\right)^2 
    + \frac{C}{\eps^2}(|\Q_\eps| - 1)^2
  \end{split}
 \end{equation*}
 For~$x\in A^1_\eps$, 
 we have~$|\Q_\eps(x)| = t_\eps\abs{\Q^*_\eps(rx/\abs{x})}
 + (1 - t_\eps)(1 + \kappa_*\eps)$,
 with~$t_\eps = t_\eps(\rho)\in [0, \, 1]$. As a consequence,
 \begin{equation} \label{interp5}
  \begin{split}
   \int_{A^1_\eps} g_\eps(\Q_\eps) \, \d x
   &\lesssim
    \lambda_\eps
    \int_{\partial B} \left(\frac{1}{\eps}(\abs{\Q^*_\eps} - 1) - \kappa_*\right)^2 \d\H^1 
    + \frac{\lambda_\eps}{\eps^2}\int_{\partial B} 
     (\abs{\Q^*_\eps} - 1)^2 \, \d\H^1 + \lambda_\eps\kappa_*^2  
  \end{split}
 \end{equation}
 On the other hand, as~$\abs{\Q^*_\eps}\to 1$
 uniformly on~$\partial B$,
 from Lemma~\ref{lemma:geps} we deduce that
 \begin{equation} \label{interp6}
  \begin{split}
   g_\eps(\Q^*_\eps) \geq \left(\frac{1}{\eps}(|\Q^*_\eps| - 1) - \kappa_*\right)^2
   - \frac{3}{4\eps^2}(|\Q^*_\eps| - 1)^2
   \geq \frac{1}{8\eps^2}(|\Q^*_\eps| - 1)^2 - 7\kappa_*^2
  \end{split}
 \end{equation}
 at any point of~$\partial B$, for~$\eps$ small enough. 
 Combining~\eqref{interp5} and~\eqref{interp6},
 we obtain
 \begin{equation} \label{interp7}
  \int_{A^1_\eps} g_\eps(\Q_\eps) \, \d x
  \lesssim \lambda_\eps\int_{\partial B} g_\eps(\Q^*_\eps) \, \d\H^1
  + \lambda_\eps\kappa_*^2 
  \stackrel{\eqref{hp:interp-app}}{\lesssim} \lambda_\eps
 \end{equation}
 If we choose~$\lambda_\eps$ in such a way that
 $\lambda_\eps\to 0$ as~$\eps\to 0$, then~\eqref{interp4}
 and~\eqref{interp7} imply
 \begin{equation} \label{interp8}
  \int_{B} g_\eps(\Q_\eps) \, \d x \to 0
  \qquad \textrm{as } \eps\to 0.
 \end{equation}
 
 Finally, we estimate the gradient term. An explicit computation
 shows that
 \begin{equation*}
  \begin{split}
   \int_{A^1_\eps\cup A^2_\eps} \abs{\nabla\Q_\eps}^2 \,\d x
   &\lesssim \lambda_\eps \int_{\partial B}
    \left( \abs{\nabla\Q^*_\eps}^2 + \abs{\nabla\Q^*}^2
    + \frac{1}{\lambda_\eps^2} \abs{\Q^*_\eps - \Q^*}^2
    + \frac{1}{\lambda_\eps^2} \left(\abs{\Q^*_\eps} - 1 - \kappa_*\eps\right)^2\right) \d\H^1 \\
   &\stackrel{\eqref{interp6}}{\lesssim}
   \lambda_\eps \int_{\partial B}
    \left( \abs{\nabla\Q^*_\eps}^2 + \abs{\nabla\Q^*}^2
    + \frac{1}{\lambda_\eps^2} \abs{\Q^*_\eps - \Q^*}^2
    + \frac{\eps^2}{\lambda_\eps^2}g_\eps(\Q^*_\eps)
    + \frac{\eps^2\kappa_*}{\lambda_\eps^2}  \right) \d\H^1 \\
  \end{split}
 \end{equation*}
 By the assumption~\eqref{hp:interp-app}, we deduce
 \begin{equation} \label{interp9}
  \begin{split}
   \int_{A^1_\eps\cup A^2_\eps} \abs{\nabla\Q_\eps}^2 \,\d x
   &\stackrel{\eqref{interp6}}{\lesssim}
   \lambda_\eps + \frac{\eps^2}{\lambda_\eps}
   + \frac{1}{\lambda_\eps} 
   \int_{\partial B} \abs{\Q^*_\eps - \Q^*}^2 \, \d\H^1 \\
  \end{split}
 \end{equation}
 We take
 \begin{equation} \label{lambdaeps}
  \lambda_\eps := \eps + 
   \left(\int_{\partial B}\abs{\Q^*_\eps - \Q^*}^2\,\d\H^1\right)^{1/2}
 \end{equation}
 By assumption, we have~$\Q^*_\eps\rightharpoonup\Q^*$ weakly
 in~$W^{1,2}(\partial B)$, hence strongly in~$L^2(\partial B)$. 
 Therefore, $\lambda_\eps\to 0$ as~$\eps\to 0$. Moreover,
 \eqref{interp9} and~\eqref{lambdaeps} imply
 \begin{equation} \label{interp1app0}
  \begin{split}
   \int_{A^1_\eps\cup A^2_\eps} \abs{\nabla\Q_\eps}^2 \,\d x
   \to 0 \qquad \textrm{as } \eps\to 0,
  \end{split}
 \end{equation}
 On the other hand, we have
 \begin{equation} \label{interp1app1}
  \begin{split}
   \int_{A^3_\eps} \abs{\nabla\Q_\eps}^2 \,\d x
   = \int_{B} \abs{\nabla\Q^*}^2 \,\d x
  \end{split}
 \end{equation}
 for any~$\eps$. Therefore, \eqref{interp2app} follows from~\eqref{interp8}, 
 \eqref{interp1app0} and~\eqref{interp1app1}.
\end{proof}

\end{appendix}

\bibliographystyle{plain}
\bibliography{singular_set}

\newcommand{\noop}[1]{}
\begin{thebibliography}{10}

\bibitem{AlicandroPonsiglione}
R.~Alicandro and M.~Ponsiglione.
\newblock Ginzburg-{L}andau functionals and renormalized energy: a revised
  {$\Gamma$}-convergence approach.
\newblock {\em J. Funct. Anal.}, 266(8):4890--4907, 2014.

\bibitem{AlmgrenBrowderLieb}
F.~Almgren, W.~Browder, and E.~H. Lieb.
\newblock Co-area, liquid crystals, and minimal surfaces.
\newblock In {\em Partial differential equations ({T}ianjin, 1986)}, volume
  1306 of {\em Lecture Notes in Math.}, pages 1--22. Springer, Berlin, 1988.

\bibitem{AmbrosioFuscoPallara}
L.~Ambrosio, N.~Fusco, and D.~Pallara.
\newblock {\em Functions of bounded variation and free discontinuity problems}.
\newblock Oxford Mathematical Monographs. The Clarendon Press, Oxford
  University Press, New York, 2000.

\bibitem{AmbrosioWenger}
L.~Ambrosio and S.~Wenger.
\newblock Rectifiability of flat chains in {B}anach spaces with coefficients in
  $\mathbb{Z}_p$.
\newblock {\em Mathematische Zeitschrift}, 268:477--506, 2009.

\bibitem{BadalCicalese}
Rufat Badal and Marco Cicalese.
\newblock Renormalized energy between fractional vortices with topologically
  induced free discontinuities on $2$-dimensional {R}iemannian manifolds.
\newblock Preprint arXiv~2204.01840, 2022.

\bibitem{Badal_et_al}
Rufat Badal, Marco Cicalese, Lucia De~Luca, and Marcello Ponsiglione.
\newblock {$\Gamma$}-convergence analysis of a generalized {$XY$} model:
  fractional vortices and string defects.
\newblock {\em Comm. Math. Phys.}, 358(2):705--739, 2018.

\bibitem{Baldo1990}
S.~Baldo.
\newblock Minimal interface criterion for phase transitions in mixtures of
  {C}ahn-{H}illiard fluids.
\newblock {\em Annales de l'Institut Henri Poincare (C) Non Linear Analysis},
  7(2):67--90, 1990.

\bibitem{BallZarnescu}
J.~M. Ball and A.~Zarnescu.
\newblock Orientability and energy minimization in liquid crystal models.
\newblock {\em Arch. Rational Mech. Anal.}, 202(2):493--535, 2011.

\bibitem{BaumanParkPhillips}
P.~Bauman, J.~Park, and D.~Phillips.
\newblock Analysis of nematic liquid crystals with disclination lines.
\newblock {\em Archive for Rational Mechanics and Analysis}, 205(3):795--826,
  Sep 2012.

\bibitem{BBH-degree_zero}
F.~Bethuel, H.~Brezis, and F.~H{\'e}lein.
\newblock Asymptotics for the minimization of a {G}inzburg-{L}andau functional.
\newblock {\em Cal. Var. Partial Differential Equations}, 1(2):123--148, 1993.

\bibitem{BBH}
F.~Bethuel, H.~Brezis, and F.~H{\'e}lein.
\newblock {\em Ginzburg-{L}andau {V}ortices}.
\newblock Progress in Nonlinear Differential Equations and their Applications,
  13. Birkh\"auser Boston Inc., Boston, MA, 1994.

\bibitem{BethuelChiron}
F.~Bethuel and D.~Chiron.
\newblock Some questions related to the lifting problem in {S}obolev spaces.
\newblock {\em {Contemporary Mathematics}}, 446:125--152, 2007.

\bibitem{BethuelZheng}
F.~Bethuel and X.~Zheng.
\newblock Density of smooth functions between two manifolds in {S}obolev
  spaces.
\newblock {\em J. Funct. Anal.}, 80(1):60 -- 75, 1988.

\bibitem{bisht2020}
K.~Bisht, Y.~Wang, V.~Banerjee, and A.~Majumdar.
\newblock Tailored morphologies in two-dimensional ferronematic wells.
\newblock {\em Physical Review E}, 101(2):022706, 2020.

\bibitem{bisht2019}
Konark Bisht, Varsha Banerjee, Paul Milewski, and Apala Majumdar.
\newblock Magnetic nanoparticles in a nematic channel: A one-dimensional study.
\newblock {\em Physical Review E}, 100(1):012703, 2019.

\bibitem{Bollobas}
B.~Bollobas.
\newblock {\em Modern Graph Theory}.
\newblock Springer-Verlag New York, 1998.

\bibitem{BrezisCoronLieb}
H.~Brezis, J.-M. Coron, and E.~H. Lieb.
\newblock Harmonic maps with defects.
\newblock {\em Comm. Math. Phys.}, 107(4):649--705, 1986.

\bibitem{BrezisNguyen}
H.~Brezis and H.-M. Nguyen.
\newblock The {J}acobian determinant revisited.
\newblock {\em Invent. Math.}, 185(1):17--54, 2011.

\bibitem{Brochard}
F~Brochard and PG~De~Gennes.
\newblock Theory of magnetic suspensions in liquid crystals.
\newblock {\em Journal de Physique}, 31(7):691--708, 1970.

\bibitem{burylov1}
S.~V. Burylov and Y.~L. Raikher.
\newblock Orientation of a solid particle embedded in a monodomain nematic
  liquid crystal.
\newblock {\em Physical review A, Atomic, molecular, and optical physics},
  50(1):358--367, 1994.

\bibitem{burylov2}
S.~V. Burylov and Y.~L. Raikher.
\newblock Macroscopic properties of ferronematics caused by orientational
  interactions on the particle surfaces. {I}. extended continuum model.
\newblock {\em Molecular Crystals and Liquid Crystals Science and Technology.
  Section A.}, 258(1):107--122, 1995.

\bibitem{canevarizarnescu}
G.~Canevari and A.~Zarnescu.
\newblock Design of effective bulk potentials for nematic liquid crystals via
  colloidal homogenisation.
\newblock {\em Math. Models Methods Appl. Sci.}, 30(2):309--342, 2020.

\bibitem{dalby2022}
James Dalby, Patrick~E Farrell, Apala Majumdar, and Jingmin Xia.
\newblock One-dimensional ferronematics in a channel: Order reconstruction,
  bifurcations, and multistability.
\newblock {\em SIAM Journal on Applied Mathematics}, 82(2):694--719, 2022.

\bibitem{dg}
P.~G. De~Gennes and J.~Prost.
\newblock {\em The {P}hysics of {L}iquid {C}rystals}.
\newblock International series of monographs on physics. Clarendon Press, 1993.

\bibitem{DeGiorgiAmbrosio}
Ennio De~Giorgi and Luigi Ambrosio.
\newblock Un nuovo funzionale nel calcolo delle variazioni.
\newblock {\em Atti Accad. Naz. Lincei Rend. Cl. Sci. Fis. Mat. Nat. (8)},
  82(2):199--210 (1989), 1988.

\bibitem{DelPinoFelmer}
M.~del Pino and P.~L. Felmer.
\newblock Local minimizers for the {G}inzburg-{L}andau energy.
\newblock {\em Math. Z.}, 225(4):671--684, 1997.

\bibitem{Federer}
H.~Federer.
\newblock {\em Geometric measure theory}.
\newblock Die Grundlehren der mathematischen Wissenschaften, Band 153.
  Springer-Verlag New York Inc., New York, 1969.

\bibitem{FonsecaTartar}
Irene Fonseca and Luc Tartar.
\newblock The gradient theory of phase transitions for systems with two
  potential wells.
\newblock {\em Proceedings of the Royal Society of Edinburgh: Section A
  Mathematics}, 111(1-2):89–102, 1989.

\bibitem{GiaquintaModicaSoucek-I}
M.~Giaquinta, G.~Modica, and J.~Sou{\oldv{c}}ek.
\newblock {\em Cartesian currents in the calculus of variations.}, volume
  37--38 of {\em Ergebnisse der Mathematik und ihrer Grenzgebiete. 3. Folge. A
  Series of Modern Surveys in Mathematics [Results in Mathematics and Related
  Areas. 3rd Series. A Series of Modern Surveys in Mathematics]}.
\newblock Springer-Verlag, Berlin, 1998.
\newblock Cartesian currents.

\bibitem{GoldmanMerletMillot}
Michael Goldman, Benoit Merlet, and Vincent Millot.
\newblock A {G}inzburg-{L}andau model with topologically induced free
  discontinuities.
\newblock {\em Ann. Inst. Fourier (Grenoble)}, 70(6):2583--2675, 2020.

\bibitem{GolovatyMontero}
D.~Golovaty and J.~A. Montero.
\newblock On minimizers of a {L}andau-de {G}ennes energy functional on planar
  domains.
\newblock {\em Arch. Rational Mech. Anal.}, 213(2):447--490, 2014.

\bibitem{IgnatLamy}
R.~Ignat and X.~Lamy.
\newblock Lifting of $\mathbb{{RP}}^{d-1}$-valued maps in {BV} and applications
  to uniaxial {Q}-tensors. with an appendix on an intrinsic {BV}-energy for
  manifold-valued maps.
\newblock {\em Calculus of Variations and Partial Differential Equations},
  58(2):68, Mar 2019.

\bibitem{iserles2009first}
Arieh Iserles.
\newblock {\em A first course in the numerical analysis of differential
  equations}.
\newblock Number~44. Cambridge university press, 2009.

\bibitem{Jerrard}
R.~L. Jerrard.
\newblock Lower bounds for generalized {G}inzburg-{L}andau functionals.
\newblock {\em SIAM J. Math. Anal.}, 30(4):721--746, 1999.

\bibitem{JerrardSoner-Jacobians}
R.~L. Jerrard and H.~M. Soner.
\newblock Functions of bounded higher variation.
\newblock {\em Indiana Univ. Math. J.}, 51(3):645--677, 2003.

\bibitem{lagerwall-scalia}
J.P.F. Lagerwall and G.~Scalia.
\newblock A new era for liquid crystal research: Applications of liquid
  crystals in soft matter nano-, bio- and microtechnology.
\newblock {\em Current Applied Physics}, 12(6):1387--1412, 2012.

\bibitem{Lin96}
Fang~Hua Lin.
\newblock Some dynamical properties of {G}inzburg-{L}andau vortices.
\newblock {\em Comm. Pure Appl. Math.}, 49(4):323--359, 1996.

\bibitem{Lin99}
Fang~Hua Lin.
\newblock Vortex dynamics for the nonlinear wave equation.
\newblock {\em Comm. Pure Appl. Math.}, 52(6):737--761, 1999.

\bibitem{Luckhaus-PartialReg}
S.~Luckhaus.
\newblock Partial {H}\"older continuity for minima of certain energies among
  maps into a {R}iemannian manifold.
\newblock {\em Indiana Univ. Math. J.}, 37(2):349--367, 1988.

\bibitem{maity2022}
Ruma~Rani Maity, Apala Majumdar, and Neela Nataraj.
\newblock Parameter dependent finite element analysis for ferronematics
  solutions.
\newblock {\em Comput. Math. Appl.}, 103:127--155, 2021.

\bibitem{Merteljetal}
A.~Mertelj, D.~Lisjak, M.~Drofenik, and M.~Copic.
\newblock Ferromagnetism in suspensions of magnetic platelets in liquid
  crystal.
\newblock {\em Nature}, 504(7479):237--241, 2013.

\bibitem{ModicaMortola}
L.~Modica and S.~Mortola.
\newblock Un esempio di {$\Gamma \sp{-}$}-convergenza.
\newblock {\em Boll. Un. Mat. Ital. B (5)}, 14(1):285--299, 1977.

\bibitem{Sandier}
\'E. Sandier.
\newblock Lower bounds for the energy of unit vector fields and applications.
\newblock {\em J. Funct. Anal.}, 152(2):379--403, 1998.
\newblock see Erratum, ibidem 171, 1 (2000), 233.

\bibitem{SchoenUhlenbeck2}
R.~Schoen and K.~Uhlenbeck.
\newblock Boundary regularity and the {D}irichlet problem for harmonic maps.
\newblock {\em J. Differential Geom.}, 18(2):253--268, 1983.

\bibitem{Simon-GMT}
L.~Simon.
\newblock {\em {L}ectures in {G}eometric {M}easure {T}heory}.
\newblock Centre for Mathematical Analysis, Australian National University,
  Canberra, 1984.

\bibitem{Struwe}
M.~Struwe.
\newblock On the asymptotic behavior of minimizers of the {G}inzburg-{L}andau
  model in {$2$} dimensions.
\newblock {\em Differential Integral Equations}, 7(5-6):1613--1624, 1994.

\bibitem{yin2020construction}
Jianyuan Yin, Yiwei Wang, Jeff~ZY Chen, Pingwen Zhang, and Lei Zhang.
\newblock Construction of a pathway map on a complicated energy landscape.
\newblock {\em Physical review letters}, 124(9):090601, 2020.

\bibitem{Ziemer}
W.~P. Ziemer.
\newblock Integral currents mod 2.
\newblock {\em Transactions of the American Mathematical Society},
  105(3):496--524, 1962.

\end{thebibliography}

\end{document}